\documentclass[a4paper,11pt,english,reqno]{amsart}
\usepackage[T1]{fontenc}
\usepackage[latin9]{inputenc}
\usepackage{geometry}
\usepackage{color,enumitem}
\usepackage{amstext}
\usepackage{amsthm}
\usepackage{amssymb}
\usepackage{graphicx}
\usepackage{xcolor}
\usepackage{hyperref}
\usepackage{verbatim}
\usepackage[foot]{amsaddr}

\newcommand{\EE}{\mathbb{E}}

\newcommand{\NN}{\mathbb{N}}
\newcommand{\PP}{\mathbb{P}}

\newcommand{\cA}{\mathcal{A}}
\newcommand{\cB}{\mathcal{B}} 
\newcommand{\cI}{\mathcal{I}}
\newcommand{\cC}{\mathcal{C}}
\newcommand{\cD}{\mathcal{D}} 
\newcommand{\cF}{\mathcal{F}}

\newcommand{\cG}{\mathcal{G}} 
\newcommand{\cJ}{\mathcal{J}}
\newcommand{\cK}{\mathcal{K}}
\newcommand{\cL}{\mathcal{L}}

\newcommand{\cN}{\mathcal{N}}

\newcommand{\cT}{\mathcal{T}}
\newcommand{\cS}{\mathcal{S}} 
\newcommand{\cR}{\mathcal{R}} 
\newcommand{\cQ}{\mathcal{Q}}

\newcommand{\cY}{\mathcal{Y}}

\newcommand{\cZ}{\mathcal{Z}}

\renewcommand{\a}{\alpha}
\newcommand{\D}{\Delta} 
\renewcommand{\d}{\delta}
\newcommand{\eps}{\varepsilon}

\newcommand{\g}{\gamma}

\renewcommand{\b}{\beta} 
 
\newcommand{\Om}{\Omega}
\newcommand{\om}{\omega} 
 
\newcommand{\s}{\sigma}
\newcommand{\ph}{\varphi}

\newcommand{\up}{\uparrow}
\newcommand{\dn}{\downarrow}
\newcommand{\oo}{\infty}
\newcommand{\ol}{\overline}
\newcommand{\sm}{\setminus}
\newcommand{\se}{\subseteq}
\newcommand{\es}{\varnothing}
\newcommand{\eqdist}{\overset{{\scriptscriptstyle \mathrm{(d)}}}{=}}

\newcommand{\floor}[1]{\lfloor #1\rfloor}
\newcommand{\ceil}[1]{\lceil #1\rceil}
\newcommand{\one}{\hbox{\rm 1\kern-.27em I}}

\newcommand{\be}{\begin{equation}}
\newcommand{\ee}{\end{equation}}

\newcommand{\R}{\mathbb{R}}
\newcommand{\E}{\mathbb{E}}

\newcommand{\eqdef}{:=}

\newcommand{\signExploration}{\mathcal{L}}
\newcommand{\lapProcess}{\mathcal{K}}

\newcommand{\fS}{\mathfrak{S}}

\newcommand{\his}{\mathrm{his}}

\makeatletter
\numberwithin{equation}{section}
\numberwithin{figure}{section}
\theoremstyle{plain}
\newtheorem{thm}{\protect\theoremname}[section]
  \theoremstyle{definition}
  \newtheorem{definition}[thm]{\protect\definitionname}
  \theoremstyle{plain}
  
  \theoremstyle{plain}
  \newtheorem{proposition}[thm]{\protect\propositionname}
  \theoremstyle{remark}
  
  \theoremstyle{plain}
  \newtheorem{corollary}[thm]{\protect\corollaryname}
\theoremstyle{plain}
  \newtheorem{lemma}[thm]{\protect\lemmaname}

\usepackage{color}
  \newcounter{constant}

\makeatother

\usepackage{babel}
  \addto\captionsbritish{\renewcommand{\corollaryname}{Corollary}}
  \addto\captionsbritish{\renewcommand{\definitionname}{Definition}}
  \addto\captionsbritish{\renewcommand{\factname}{Fact}}
  \addto\captionsbritish{\renewcommand{\propositionname}{Proposition}}
  \addto\captionsbritish{\renewcommand{\remarkname}{Remark}}
  \addto\captionsbritish{\renewcommand{\theoremname}{Theorem}}
  \addto\captionsenglish{\renewcommand{\corollaryname}{Corollary}}
  \addto\captionsenglish{\renewcommand{\definitionname}{Definition}}
  \addto\captionsenglish{\renewcommand{\factname}{Fact}}
  \addto\captionsenglish{\renewcommand{\propositionname}{Proposition}}
  \addto\captionsenglish{\renewcommand{\remarkname}{Remark}}
  \addto\captionsenglish{\renewcommand{\theoremname}{Theorem}}
  \providecommand{\corollaryname}{Corollary}
  \providecommand{\definitionname}{Definition}
  \providecommand{\factname}{Fact}
  \providecommand{\propositionname}{Proposition}
  \providecommand{\remarkname}{Remark}
\providecommand{\theoremname}{Theorem}
\providecommand{\lemmaname}{Lemma}

\newcommand{\cycles}{\mathfrak{C}}            
\newcommand{\cycstruc}{\mathfrak{X}}
\newcommand{\bal}{\mathfrak{B}}
\newcommand{\PD}{\cY}

\usepackage{epstopdf}
\def\cross{\lower .5mm\hbox{\includegraphics{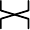}}}
\def\dbar{\lower .5mm\hbox{\includegraphics{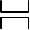}}}

\makeatletter   
 \def\l@subsection{\@tocline{2}{0pt}{2pc}{6pc}{}}
  \makeatother

\begin{document}

\title[Interchange process  with 
reversals]{The interchange process  with 
reversals \\ on the complete graph}

\author{J. E. Bj\"ornberg$^1$}
\address{$^1$Department of Mathematics,
Chalmers University of Technology and the University of Gothenburg,
Sweden}
\author{Micha\l{} Kotowski$^2$}
\address{$^2$Institute of Mathematics, Faculty of Mathematics, Informatics, and Mechanics, University of Warsaw, Banacha 2, 02-097 Warsaw, Poland}
\author{B. Lees$^3$}
\address{$^3$Fachbereich Mathematik, Technische Universit\"at Darmstadt,  Germany}
\author{P. Mi\l o\'s$^4$ }
\address{$^4$ Institute of Mathematics of the Polish Academy of Sciences, Warsaw, Poland}

\email{$^1$jakob.bjornberg@gu.se, $^2$michal.kotowski@mimuw.edu.pl}
\email{$^3$lees@mathematik.tu-darmstadt.de, $^4$pmilos@mimuw.edu.pl}

\begin{abstract} 
We consider an extension of the interchange process
on the complete graph, in which a
fraction of the transpositions are replaced by `reversals'.  The
model is motivated by statistical physics, where it plays a role in
stochastic representations of \textsc{xxz}-models.
We prove convergence to PD($\tfrac12$) of the rescaled cycle sizes,
above the critical point for the appearance of macroscopic cycles.
This extends a result of Schramm on convergence to PD(1) for the usual
interchange process.
\end{abstract}
\maketitle

\tableofcontents

\section{Introduction}

Recent years have seen a growing interest in the cycle structure of
large random permutations.  A major example is the interchange process,
or random-transposition random walk.   One motivation for
studying this process is that it plays a key role in a
stochastic representation of the most important quantum spin system, 
the ferromagnetic Heisenberg model.  This representation was developed 
by T\'oth in the early 1990's \cite{T} (after an earlier observation
 by Powers \cite{Powers}).

At about the same time, a closely related stochastic representation
was discovered for the \emph{anti}-ferromagnetic Heisenberg model, by
Aizenman and Nachtergaele \cite{A-N}.  Very roughly speaking, in
the ferromagnetic Heisenberg model the 
interaction between neighbouring electrons 
behaves like a transposition of the spins. In the antiferromagnetic
model the interaction involves a `reversal', which
Aizenman and Nachtergaele depicted as on the right in Figure
\ref{fig:transpositions}. 

\begin{figure}[hb]
\centering
\includegraphics{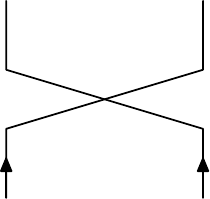} \hspace{3cm}
\includegraphics{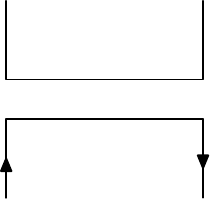}
\caption{Pictorial representation of a transposition (left) and a
  `reversal' (right).
}
\label{fig:transpositions}
\end{figure}

In both cases, the stochastic representation of the spin system 
involves randomly placing these objects in the product of the graph
with an interval.
In the case of the ferromagnetic
model, the relevant measure has transpositions appearing 
randomly at each edge, in the manner
of independent Poisson processes.  For the antiferromagnetic model, the
structure is the same except that the transpositions are replaced by
 `reversals'
as on the right in Figure \ref{fig:transpositions}.
Many quantities of interest for the spin systems, such as correlation
functions, may be expressed using expected values of suitable random
variables in these processes.  

Recently, Ueltschi \cite{U} explained that weighted
combinations of the two processes described above also lead to  
representations of certain quantum spin systems (known as
\textsc{xxz}-models).  The relevant measure has independent
Poisson processes on the edges as before, but the objects are now
randomly chosen to be either transpositions or `reversals', 
independently over the points of
the process and with some fixed probability (see Figure
\ref{fig:loopexample}). 
In this paper we study such a process defined on the complete
graph.  Our main result is that the correlation structure in this
model, above a critical point, is described by a probability
distribution on random partitions called the 
Poisson--Dirichlet distribution with
parameter $\tfrac12$.
To state our results more precisely,  let us give the relevant
definitions.

\subsection{Definitions}

We consider the complete graph $K_n = (V_n, E_n)$ on $n>2$ vertices. 
The vertex set is $V_n=\{1,2,\dotsc,n\}$
and the edge-set consists of all pairs $\{i,j\}$ of vertices $i\neq j$.
To
each edge and vertex we attach a circle of circumference $1$, which we
denote by $S^1$.  We will sometimes identify $S^1$ with the unit
interval $[0,1)$.
A \emph{configuration} $\om$ 
is a finite subset of 
$E_n \times S^1 \times \{\cross,\dbar\}$, 
where $\cross,\dbar$ are two possible
marks which we call a \emph{cross} and a \emph{bar},
respectively.  
The collection of configurations is denoted $\Om$.
An element $(e,\varphi,m)\in\om$ of the configuration is called a
\emph{link}
and if $(e,\varphi,m)$ is a link
then we say that $\om$
\emph{has a link at} $(e,\varphi)\in E_n\times S^1$.

We will primarily be interested in configurations obtained as
samples of a (marked) Poisson point process defined in the following
way. Fix $\nu \in [0,1)$ 
and $\beta > 0$.  For each edge $e\in E_n$ we consider
a Poisson point process with intensity $\tfrac\b {n-1}$ on
$e\times S^1$, these Poisson processes
being independent for different edges $e$. 
This defines a configuration of unmarked links. 
The configuration $\omega$ is then obtained by assigning to each link
a mark, independently of all other links, which is either a
cross, $\cross$, with probability $\nu$, or a bar, $\dbar$, with
probability $1 - \nu$. The probability measure corresponding to this
point process will be denoted by $\PP_{\b}$ (we consider $\nu$ to be
fixed and it will be suppressed in the notation), and the corresponding
expectation will be denoted by $\EE_{\b}$.  
We will refer to this process as the \emph{interchange process with
  reversals}
(the usual interchange process would correspond to taking $\nu=1$).

Such a configuration $\om$ gives rise to a set of \emph{loops} 
$\g\se V_n\times S^1$. 
We first give an informal description and then a more precise
definition.
For a fixed point $(v,\ph)\in V_n\times S^1$ the
unique loop, $\g(v,\ph)$, containing it is constructed by the following
process.  Starting from $(v,\ph)$ we move on the associated circle in
the positive direction, i.e. after time $\mathrm{d}t$ we are at a
point $(v,\ph + \mathrm{d}t)$. If we encounter a point $(v, \ph')$ 
such that $\om$ has a link at $(\{v,w\},\ph')$ for 
some $w\in V_n$, 
then we traverse this link to $(w,\ph')\in V_n\times S^1$. 
We then continue moving in the positive direction if the link was a
cross, or in the negative direction
if the link was a bar. Each time we encounter
a link we follow this rule of traversing the link, reversing our
direction if the link was a bar. Continuing until we
arrive back at $(v,\ph)$ we have traced out a single loop, $\g(v,\ph)$.

More formally, following \cite{G-U-W,U} we may define the
loops as follows.  A loop of length $L$ is a function
$\gamma:[0,L)\to V_n\times S^1$ such that,
writing $\gamma(t)=(v(t),\ph(t))$, the
following properties hold:
\begin{enumerate}
\item $\gamma$ is injective and satisfies 
$\lim_{t\uparrow L}\gamma(t)=\gamma(0)$. 
\item $\gamma$ is piecewise continuous, 
and if it is continuous on
the interval $I\subset [0,L)$ then $v(t)$ and 
$\tfrac{\mathrm{d}}{\mathrm{d}t}\ph(t)$ are
constant on $I$, with
$\tfrac{\mathrm{d}}{\mathrm{d}t}\ph(t)\in\{-1,1\}$.
\item $\gamma$ is discontinuous at the point $t$ if and only if
$\om$ has a link
at $(\{v(t-),y\},t)$ for some $y\neq v(t-)$, in which case $v(t+)=y$.
\item If $I_1=(t_1,t_2)$ and $I_2=(t_2,t_3)$ with $\gamma$ continuous
on $I_1$ and $I_2$ but discontinuous at $t_2$ then for any $s_1\in
I_1$ and $s_2\in I_2$ we have that
$\tfrac{\mathrm{d}}{\mathrm{d}t}\ph(s_1)=\tfrac{\mathrm{d}}{\mathrm{d}t}\ph(s_2)$
if the link at $(\{v(t_2-),v(t_2+)\},t_2)$ is a cross and
$\tfrac{\mathrm{d}}{\mathrm{d}t}\ph(s_1)=-\tfrac{\mathrm{d}}{\mathrm{d}t}\ph(s_2)$
if  
it is a bar.
\end{enumerate}
Loops with the same support but different parameterisations are
identified. This means that the functions 
$\gamma(t)$, $\gamma(-t)$
and, for $s\in \R$, $\gamma(s\pm t)$ are identified. 
From this description we can give $\omega$ a natural pictorial
representation, see Figure \ref{fig:loopexample}.

\begin{figure}[t]
\includegraphics[width=12cm,height=8cm]{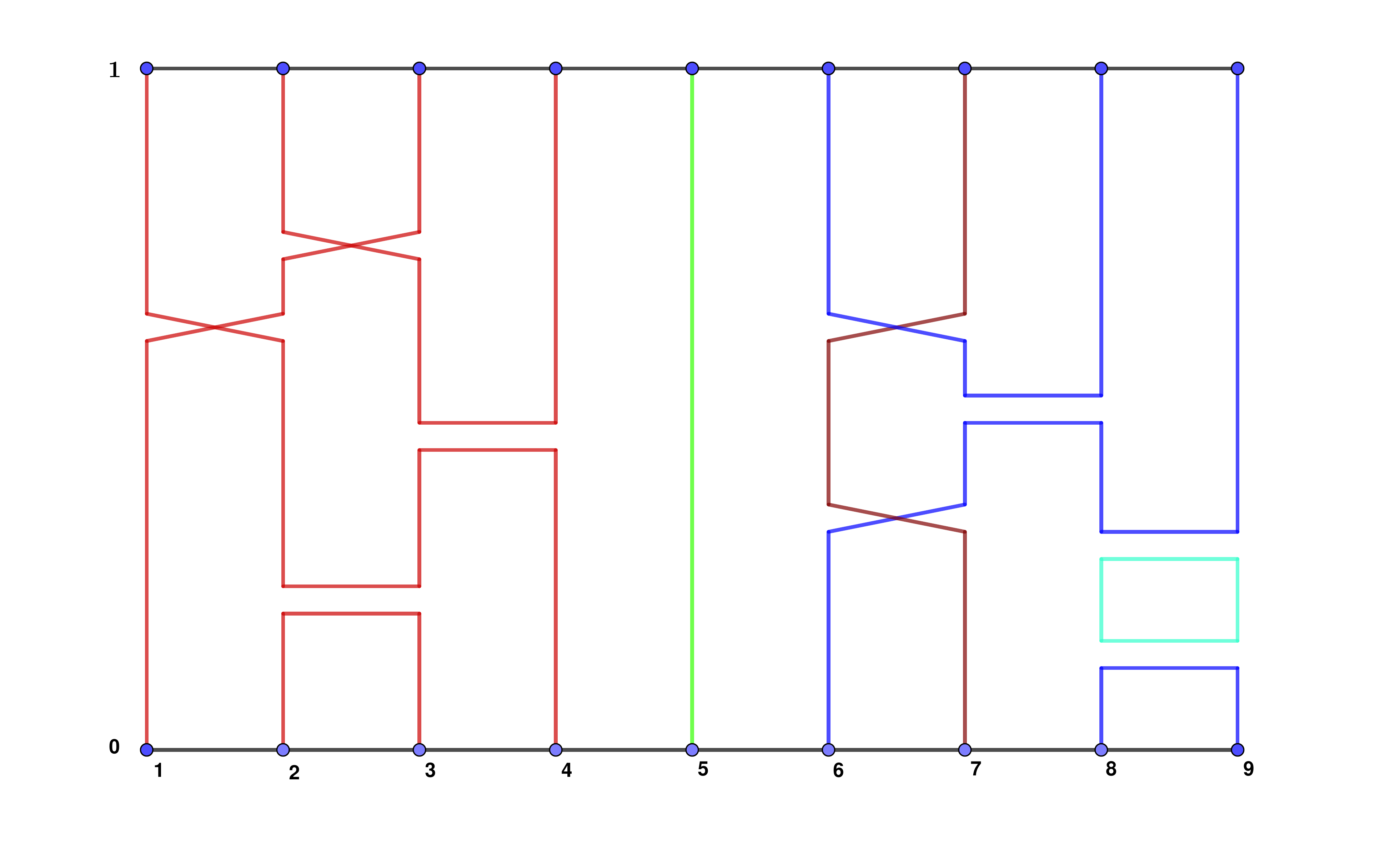}
\centering
\caption{An example of a realisation of links with loops coloured.
The circles $\{v\}\times S^1$ are represented as intervals placed
vertically, thus there are periodic boundary conditions vertically.
The cycles are $(1^\up,3^\up,2^\dn,4^\up)$,
$(5^\up)$, $(6^\up,9^\up,8^\dn)$ and $(7^\up)$.  There is one small
loop which does not give rise to a cycle.
}
\label{fig:loopexample}
\end{figure}

A \emph{cycle} is a sequence of vertices $v\in V_n$
such that the points $(v,0)$ are visited by a single loop. 
Namely, suppose we start at a point $(v_1,0)$ and follow the loop 
$\g=\g(v_1,0)$, in either direction, until we
return to the starting point.  If we enumerate the successive 
visits to $V_n\times \{0\}$ as $(v_1,0),\dotsc,(v_\ell,0),(v_1,0)$
then the corresponding cycle
is
\[
	\cC=(v_1^{d_1},v_2^{d_2},\dotsc,v_\ell^{d_\ell}), \qquad v_i\in V_n.
\]
Here the cycle $\cC$ has length $|\cC|=\ell$, and the
$d_i\in\{\up,\dn\}$ denote the direction in which we pass through the
point $(v_i,0)\in V_n\times S^1$, with $\up$ corresponding to
the positive direction ($\tfrac{\mathrm{d}}{\mathrm{d}t}\ph(t)=+1$)
and $\dn$ corresponding to negative direction
 ($\tfrac{\mathrm{d}}{\mathrm{d}t}\ph(t)=-1$). 
Note that the
directions $d_i$ in a cycle are defined up to an overall reversal
and that we
made an arbitrary choice of the first vertex $v_1$. 
It is also
worth noting that not every loop gives rise to a cycle, see Figure
\ref{fig:loopexample}.
A fixed configuration of links, $\omega$, has an associated set of
cycles which we denote by $\cycles_{\omega}$. 

\subsection{Main result}

Let $\omega$ be sampled from the measure $\PP_{\b}$. Consider the
random graph where an edge is present between vertices $u$ and $v$ if
there is at least one link on $\{u,v\}\times S^1$ in $\omega$. 
By the Erd\H{o}s-R\'enyi theorem,
if $\b>1$ then the largest
connected component of this graph, $V_G^\b$,
has size approximately $zn$ where $z$ is the positive solution to 
$1-z=e^{-\b z}$.
(If $\b<1$ the largest component has size smaller than $(\log n)^2$
and the same holds for the largest cycle.)
Let $\cycstruc_\omega$ denote 
the list $(|\cC|/|V_G^{\b}| \colon \cC\in\cycles_\omega)$ of
rescaled cycle sizes, ordered by
decreasing size (we make it into an infinite list by appending
infinitely many 0's).
 Our main result is the following.
\begin{thm}\label{thm:main}
Let $\b>1$ and $\nu \in [0,1)$.  Let $\omega$ be sampled 
from the corresponding measure $\PP_{\b}$. As $n\to\infty$ the law of 
$\cycstruc_\omega$ converges weakly to the
Poisson--Dirichlet distribution \textup{PD($\tfrac{1}{2}$)}.
\end{thm} 

More precisely, we will show that for given $\b>1$ and
$\varepsilon>0$ there exists $n(\b,\varepsilon)$ such that for
$n>n(\b,\varepsilon)$
there is a coupling of the interchange process with reversals
with a PD($\tfrac{1}{2}$) sample $\PD$ such that
\begin{equation}\label{eq:convergencetoPD}
\PP\big(\big\|\PD-\cycstruc\big\|_{\infty}
<\varepsilon\big)>1-\varepsilon.
\end{equation} 
Note that this result holds for any $\nu<1$.

The Poisson--Dirichlet distribution with parameter $\theta>0$,
PD($\theta$), can be defined via the `stick-breaking' construction
as follows.  Let $B_1,B_2,\dots$ be independent
Beta($1,\theta$) random variables, thus 
$\PP(B_i >s)=(1-s)^{\theta}$ for $s\in[0,1]$. 
We construct a random partition
$\{P_i\}_{i\in\NN}$ of $[0,1)$ using the $B_i$ by letting $P_1=B_1$ and
$P_{k+1}=B_{k+1}(1-P_1-\dots-P_k)$. 
We can think of constructing
$\{P_i\}_{i\in\NN}$ by progressively breaking off pieces of $[0,1)$,
with the $(k+1)^{\mathrm{th}}$ removed piece being a fraction
$B_{k+1}$ of what remained after $k$ pieces had been removed. The
law of the partition $\{P_i\}_{i\in\NN}$ is called the GEM($\theta$) 
distribution. 
The PD($\theta$) distribution is obtained 
by sorting the $P_i$ in order of decreasing size.

Returning to the context of Theorem \ref{thm:main}, let us comment on
the case $\nu=1$  (only crosses allowed) 
which is excluded by our result.
This is the (usual) interchange process,
and was considered on the complete graph in a famous paper by 
Schramm \cite{schramm}.   To be precise, he considered the closely
related process where the configuration $\om$ is obtained by placing
the crosses successively one after the other, uniformly and independently at
each step.  Viewing the crosses as transpositions, as above, the
process is a random-transposition random walk on the set of
permutations of $n$ objects.  In this case Schramm proved that, when
the number of transpositions exceeds $cn$ for $c>\tfrac12$,
then the rescaled cycle sizes of the resulting
random permutation converge in distribution to
PD(1).  

The main tool in Schramm's argument was
a coupling with a split-merge process which has
PD(1) as an invariant distribution.  Roughly speaking, the important
feature is what happens to an existing cycle when a uniformly chosen
transposition is applied.  If the transposition 
transposes two points which
belonged to different cycles then those cycles merge;  if they
belonged to the same cycle then the cycle is split.
A similar principle applies to the loops, which on the addition of
another cross to $\om$ either
merge, if the ends are in different loops, or split, if
both ends are in the same loop (Figure \ref{fig:splittwist}).

\begin{figure}[h!]
\includegraphics[width=10cm,height=5cm]{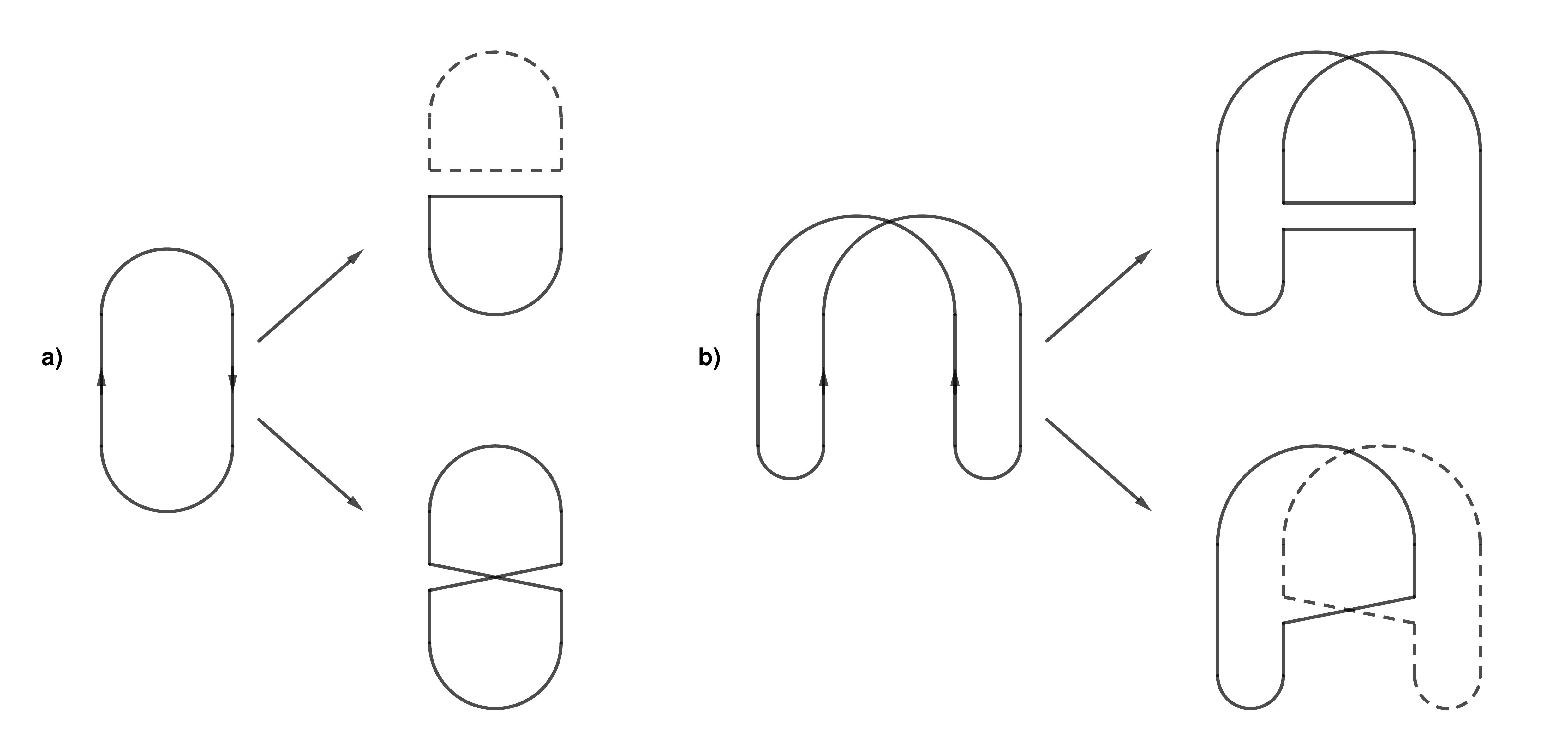}
\centering
\caption{Geometric interpretation of the effect of 
adding a cross or bar with
both endpoints in the same loop: a) in the case when the points have
opposite vertical orientation within the loop, a bar splits the loop
whereas a cross `twists';  b) in the case of the same
vertical orientation, a cross splits the loop whereas a
bar `twists'.
}
\label{fig:splittwist}
\end{figure}

Now we may explain how the case $\nu<1$ is different from $\nu = 1$, and why we
get PD($\tfrac12$) rather than PD(1).  The key point is that the
presence of bars ($\dbar$) introduces changes of orientation within
the loops.  
This means that on adding a link (cross or bar)
with both endpoints in
the same loop, this loop will not always split. 
Whether or not the loop splits depends on the orientation of the loop
at the points where the new link is placed.  Specifically, if the link
is a cross then a split occurs if and only if the orientation is the
same; if it is a bar then the opposite applies.
The situation is
depicted in Figure \ref{fig:splittwist}.

Intuitively, when $\nu<1$ 
one would expect large loops to
encounter many bars.   Hence a uniformly chosen pair of points on 
a large loop (with the same $S^1$-coordinate)
should have probability  close to $\tfrac12$ of having
the same orientation,
 meaning that the probability of
splitting is close to $\tfrac12$.
The corresponding split-merge dynamics, where proposed splits occur
with probability $\tfrac12$, has PD($\tfrac12$) as its invariant
distribution.  
 
\subsection{Outline and related works}

In order to prove Theorem \ref{thm:main} we need three ingredients. Firstly, we need that with high probability 
(converging to 1 as $n\to\oo$) there are cycles of size $\Theta(n)$, and that these large cycles 
occupy almost the entire giant component $V_G^\b$. This is proved in Section \ref{sec:large} by a straightforward
adaptation of arguments in \cite{schramm}. Secondly, we need that in large cycles
roughly half of all vertices are passed through in the positive
direction ($\up$) and roughly half in the negative direction
($\dn$).   In fact, we need the stronger statement that
the large cycles are `well-balanced', namely: 
one may partition them into much smaller segments such that each segment
consists of roughly half $\up$ and half $\dn$
(ruling out, for example, a situation in which a cycle of size $k$
consists of a block of $k/2$ vertices passed in 
direction $\uparrow$ followed by a block of
$k/2$ vertices with direction $\downarrow$).
This is the main novel contribution of the present paper, and is the
content of  Section \ref{sec:balance}. 
In proving this result we rely on a process which we call the 
\emph{exploration process}, which we study in 
Section \ref{exploration-sec}.
Thirdly, we show that Schramm's coupling, when combined with the
previous two ingredients,
can be adapted to couple a PD($\tfrac12$) sample with
$\cycstruc$  such that the two samples are close.  This appears in 
Section \ref{sec:PDcoupling}.

We now briefly summarise some other related works
apart from Schramm's paper \cite{schramm}.
First note that interchange processes, with reversals
($\nu<1$) or without
($\nu=1$), can be defined on more general graphs,
by placing independent Poisson processes of links on the edges of the
graph.  Most papers dealing with graphs other than the
complete graph have studied the question of
whether there can be large cycles.
The case when the graph is a
hypercube, and $\nu=1$, has been investigated
by Koteck\'y, Mi{\l}o\'s and Ueltschi \cite{K-M-U}.  
The case of  Hamming graphs, also for $\nu=1$,
has been investigated by Mi{\l}o\'s and \c{S}eng\"ul 
\cite{M-S} and by Adamczak, Kotowski and Mi{\l}o\'s
\cite{A-K-M}.
In the case when the graph is an infinite tree one may 
ask about the occurrence of \emph{infinite} cycles.
For $\nu=1$ this question was investigated by  
Angel \cite{A} and by Hammond \cite{H1,H2};
and for $\nu<1$ by Bj\"ornberg and Ueltschi  \cite{Bj-U}
as well as Hammond and Hegde \cite{H-H}.

As mentioned above, the original interest in the process 
was due to its connections with quantum spin systems. When the measure $\PP_{\b}$
defining the process is given an additional weighting of
$\vartheta^{\# \mathrm{loops}}$ for $\vartheta\in \NN$,
the loop-model is essentially equivalent to 
a spin system on the same graph.
This was first proved by T\'oth \cite{T} in the case
$\nu=1$ (spin-$\tfrac12$ Heisenberg ferromagnet for $\vartheta=2$)
and Aizenman and Nachtergaele \cite{A-N} in the case
$\nu=0$ (Heisenberg antiferromagnet, provided the graph is
bipartite).   
This connection was extended to the case $\nu\in[0,1]$ by
Ueltschi \cite{U}. 
From a probabilistic point of view, any $\vartheta>0$ makes sense. 
Such models  have been considered on trees
\cite{Bj-U2,B-E-L} and on the Hamming graph \cite{A-K-M}.
In very recent work there has been some limited progress in the direction
of establishing Poisson--Dirichlet structure in these and related loop
models \cite{BU,BFU}.
For the Heisenberg model ($\nu=1$ and $\vartheta=2$)
on the complete graph, the critical point for
the appearance of cycles of diverging length was established 
already in the early 1990's by T\'oth
and by Penrose \cite{penrose,toth-bec}.

\subsection*{Acknowledgements}

The research of JEB is supported by Vetenskapsr{\aa}det grant
2015-05195.  
BL gratefully acknowledges support from the 
Alexander von Humboldt Foundation.
JEB and BL also gratefully acknowledge support from 
\emph{Stiftelsen Olle Engkvist Byggm\"astare}. MK acknowledges support from the National Science Centre, Poland, grant no. 2015/18/E/ST1/00214. PM is supported by the National Science Centre, Poland, grant no. 2014/15/B/ST1/02165. We thank Rados\l aw Adamczak for useful discussions about this project.

\newpage 

\section*{Notation}
\begin{center}
	\begin{tabular}{ l l }

$V_n = \{1,2,...,n\}$ & vertex set of the complete graph; typical elements denoted $u,v,w,\dotsc$ \\
$E=\binom{V}{2}$ & edge set of the complete graph \\
$m\in\{\cross,\dbar\}$ & the `mark' of a link as either a cross or a bar \\
$\ph\in S^1$ & `phase' or vertical coordinate \\
$\Omega$ & space of configurations, i.e.\ finite subsets of $E\times S^1\times\{\cross,\dbar\}$ \\
$\om$, $\om_A$ & element of $\Om$, its restriction to $A\se E\times S^1$ \\
$\vec\om$, $\vec\om_k$ & ordered sequence of links in $\om$, its first $k$ elements  \\
$\b>0$ & intensity parameter \\
$\PP_{\b}$ & measure of the Poisson links process with intensity 
$\tfrac\b{n-1}$ per edge \\
$\nu\in[0,1)$ & probability of marking a link as a cross $\cross$ \\
$\cycles_\omega$ & set of cycles defined by $\om$ \\
$\cycstruc_\omega$ & list of rescaled cycle sizes \\
$V_G^\b$ & the largest cluster in the random graph whose edges
            support links \\
$X_t = X_t(v,\varphi)$ & exploration process for loops (started at $(v,\varphi)\in V_n\times S^1$) \\
$Y_t = Y_t(v,\varphi)$ & simple exploration process (started at $(v,\varphi)\in V_n\times S^1$) \\
$\tau^{X(Y)}(v,\varphi)$ & closing time of (simple) exploration started from $(v,\varphi)$ \\
$\{\cF_t\}_{t\geq 0}$ & natural filtration of the exploration process \\
$\cI^{X(Y)}_t$ & number of times $X_t$ (resp. $Y_t$) traverses a link \\
$\cJ^{X(Y)}_t$ & number of links discovered by $X_t$ (resp. $Y_t$) \\
$\cL^{X(Y)}_t$ & number of windings of $X_t$ (resp. $Y_t$) around $S^1$ \\
$\cK^{X(Y)}_t$ & number of visits to $\varphi=0\in S^1$ by $X_t$ (resp. $Y_t$) \\
$\ell_k$ & frontier times of the simple exploration \\ 
$\Delta_k$ & $\ell_{k+1}-\ell_k$ \\
$c_\omega(v,k)$ & set of the next $k$ vertices in a cycle starting from $v$  \\
$c^{\uparrow(\downarrow)}_{\omega}(v,k)$ & number of vertices in
                                           $c_\omega(v,k)$ passed in
                                           positive or negative direction \\ 
$\bal_\omega(v,k)$ & $\big||c_{\omega}^{\uparrow}(v,k)|-|c_{\omega}^{\downarrow}(v,k)|\big|$ \\
PD($\theta$) & Poisson--Dirichlet distribution with parameter $\theta$ \\
$\PD,\cZ$ & samples from PD($\theta$) \\
\end{tabular}

\end{center}

\newpage 

\section{Large cycles}
\label{sec:large}

In this section we
show that, for $\b>1$,
with high probability
there are cycles of length of order $n$.  The precise statement
appears in Lemma \ref{large-lem} at the end of the section.
The argument is a minor adaptation of Schramm's
\cite[Section~2]{schramm}.  

As in \cite{schramm}, we will actually work not with configurations $\om$ sampled from the Poisson measure
$\PP_\b$, but instead with configurations constructed sequentially
one link at a time.  
Given any configuration $\om\in\Om$, note that the
set of cycles $\cycles_\om$ only depends on the relative order of the
links of $\om$ as well as their position relative to $0\in S^1$, 
but not on their precise $S^1$-coordinates.  
Given $\omega\in\Omega$, let us
order its elements (links) with respect to the $S^1$-coordinate, 
namely we write
$\om=\{(e_{1},\ph_{1},m_1),\ldots,(e_{|\om|},\ph_{|\om|},m_{|\om|})\}$,
with $0<\ph_{1}<\ldots<\ph_{|\om|}<1$ (we can assume that there are no
distinct links with $\ph_i = \ph_j$, since under $\PP_\b$ this occurs
with probability $1$). 
We denote by $\vec{\om}=((e_{1},m_{1}),\ldots,(e_{|\om|},m_{|\om|}))$
 the ordered list of links with $S^1$-coordinates suppressed.  
With a slight abuse of terminology we will also refer to the 
entries $(e_{i},m_{i})$ of $\vec\om$ as links.
As noted above, $\cycles_\om$ is a function
of $\vec\om$ only, hence we may write $\cycles_{\vec\om}$.

In the rest of this section we will work with a random $\vec\om$
obtained by sequentially laying down a fixed number $t$ of random
links.  More precisely, first let $e_1$ be chosen uniformly from the
edge-set $E_n$ and let $m_1\in\{\cross,\dbar\}$ be chosen independently
of $e_1$,  with probability $\nu$ for $\cross$.  Next,
given the first $s$ links $(e_1,m_1),\dotsc,(e_s,m_s)$, 
we select $e_{s+1}$ uniformly from $E_n$ and the mark 
$m_{s+1}\in\{\cross,\dbar\}$, with
probability $\nu$ for $\cross$, independently of each other and of the
previous choices.   Write 
$\vec\om_s=((e_1,m_1),\dotsc,(e_s,m_s))$ and let
$\cycles_s:=\cycles_{\vec\om_s}$ denote
the set of cycles after $s\leq t$ steps.
Note that, if
$t$ is taken to be Poisson-distributed with 
mean $\tfrac\b {n-1}\binom{n}{2}=\tfrac\b 2 n$, then 
$\cycles_t$ is equal in distribution to $\cycles_\om$
for $\om$ sampled from $\PP_\b$.
Due to concentration properties of the Poisson-distribution, there
is very little difference between $\cycles_t$ for 
$t=\floor{\tfrac\b 2 n}$ on the one hand, and 
$\cycles_\om$ for $\om$ sampled from $\PP_\b$ on the other.
We will not make this statement more precise at this point, deferring
this to later (see Section \ref{sec:proofofmainthm}). 

We now describe in detail the effect that appending the next link $(e_{s+1},m_{s+1})$ to 
$\vec\om_s$ has on the cycles, that is, the transition $\cycles_s\to\cycles_{s+1}$.  
See Figures \ref{fig:cycles-merge}, \ref{fig:cycles-split}
and \ref{fig:cycles-twist} for illustrations in the case when
$m_{s+1}=\dbar$. 
 \begin{figure}[h!]
\includegraphics{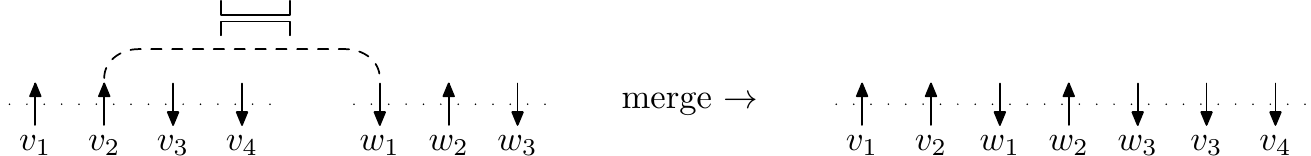}
\centering
\caption{Example of two cycles merging.
}
\label{fig:cycles-merge}
\end{figure}

 \begin{figure}[h!]
\includegraphics{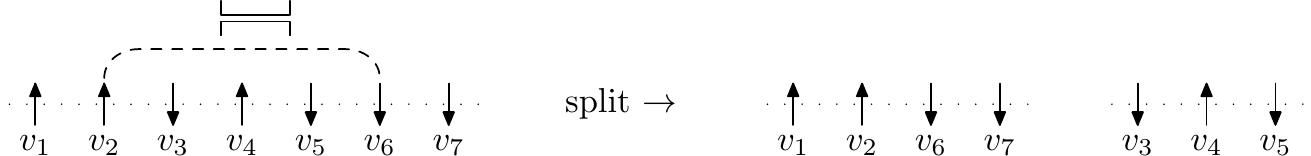}
\centering
\caption{Example of a cycle splitting.}
\label{fig:cycles-split}
\end{figure}

 \begin{figure}[h!]
\includegraphics{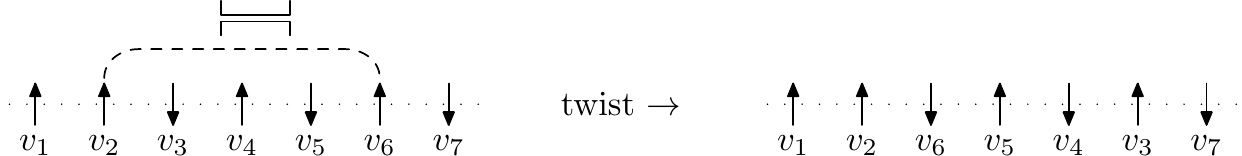}
\centering
\caption{Example of a cycle `twisting', i.e.\ it stays intact but is
  restructured internally.}
\label{fig:cycles-twist}
\end{figure}
For $d\in\{\up,\dn\}$ we write $-d$ for the reversed arrow.
We have the following:
\begin{itemize}[leftmargin=*]
\item If the endpoints of $e_{s+1}$ are in different cycles of
$\cycles_s$ then those cycles merge.
\item If the endpoints are in the same cycle $\cC$ then the result
depends on the mark $m_{s+1}$ in the following way.  Let us assume
that $\cC=(v_1^{d_1},v_2^{d_2},\dotsc,v_\ell^{d_\ell})$ and that
$e_{s+1}=\{v_i,v_j\}$ where $i<j$.
Without loss of generality (since directions are defined up to an
overall reversal) we may assume
that $d_i=\up$.
\begin{itemize}
\item If $m_{s+1}=\cross$ 
is a cross then $\cC$ splits if and only if
$d_j=\up$; in this case the two resulting cycles $\cC'$ and $\cC''$ are given by:  
\[
\cC'=(v_1^{d_1},\dotsc,v_i^{d_i},v_{j+1}^{d_{j+1}},\dotsc,v_\ell^{d_\ell}),
\qquad
\cC''=(v_{i+1}^{d_{i+1}},v_{i+2}^{d_{i+2}}\dotsc,v_{j}^{d_j}).
\]
On the other hand, if $d_j=\dn$ then $\cC$ is not split;
instead it is modified into $\cC'$ where
\[
\cC'=(v_1^{d_1},\dotsc,v_i^{d_i},v_{j-1}^{-d_{j-1}},
v_{j-2}^{-d_{j-2}},\dotsc,v_{i+1}^{-d_{i+1}},v_j^{d_j},
\dotsc,v_\ell^{d_\ell}).
\]
\item If $m_{s+1}=\dbar$ 
is a bar then $\cC$ splits if and only if
$d_j=\dn$; in this case the two resulting cycles $\cC'$ and $\cC''$ are given by:  
\[
\cC'=(v_1^{d_1},\dotsc,v_i^{d_i},v_{j}^{d_{j}},\dotsc,v_\ell^{d_\ell}),
\qquad
\cC''=(v_{i+1}^{d_{i+1}},v_{i+2}^{d_{i+2}}\dotsc,v_{j-1}^{d_{j-1}}).
\] 
On the other hand, if $d_j=\up$ then $\cC$ is not split but
modified into $\cC'$ where
\[
\cC'=(v_1^{d_1},\dotsc,v_i^{d_i},v_{j}^{-d_{j}},v_{j-1}^{-d_{j-1}},
\dotsc,v_{i+1}^{-d_{i+1}},v_{j+1}^{d_{j+1}},
\dotsc,v_\ell^{d_\ell}).
\]
\end{itemize}
\end{itemize}

Note that
the edge $e_{s+1}$ may be selected by first choosing $v_i$ uniformly
from $E_n$ and then $v_j$ uniformly from 
$E_n\sm\{v_i\}$.  In particular we see that,
just as in \cite[Lemma 2.1]{schramm}, we have:

\begin{lemma}\label{le:splits} 
In the step from $\vec\om_s$ to $\vec\om_{s+1}$,
the probability that some cycle is split
into two cycles, with at least one containing at
most $k$ vertices, is at most $2k/(n-1)$.
\end{lemma}

Building on this, and replicating the arguments of Schramm
\cite{schramm}, we obtain the following sequence of lemmas.
Lemma \ref{lem:EV} is proved
exactly as \cite[Lemma 2.2]{schramm}.  Lemma \ref{lem:growth}
is a version of \cite[Lemma 2.3]{schramm} and is proved in a similar way,
see \cite[Lemma 5.1]{A-K-M}
for details.  Briefly,
the reason that these results hold exactly as in \cite{schramm} is
that, firstly, if the endpoints of $e_{s+1}$ are in
different cycles then those cycles always merge, and,
secondly, the cycle may or may not split if the
endpoints are in the same cycle.  This means that 
both the upper bounds on the probability of
splitting, as well as the lower bounds on the probability of merging, 
are identical to \cite{schramm}.  This is all that we need.

In the following statements
we consider a random graph $G^s$ with vertex set $V_n$
obtained by placing an
edge between a pair $\{i,j\}$, $i\neq j$, if in $\vec\om_s$ there is at
least one link $(e_r,m_r)$, $r\leq s$, such that $e_r=\{i,j\}$.
We write  $V^{s}_G(k)$ for the set of vertices in connected components
of $G^s$ containing at least $k$ vertices.  Similarly,
we write $V_{\cycles}^{s}(k)$ for the set of vertices belonging to cycles of
$\cycles_s$ which are of length at least $k$.

\begin{lemma}\label{lem:EV}
For any $s\geq0$,
\begin{equation} 
\E|V^{s}_G(k)\setminus V_{\cycles}^{s}(k)|\leq
\frac{4sk^2}{n-1}
\end{equation}
\end{lemma}

\begin{lemma} \label{lem:growth}
Let $t_0\in \NN$, $\delta \in (0,1]$, 
$\varepsilon \in (0,1/8)$ and $j\in \NN$ be such that 
$ 2^j \le \varepsilon \delta n$. 
Assume that the following conditions hold in the transition from
$\vec\om_s$ to $\vec\om_{s+1}$:
\begin{enumerate}[leftmargin=*]
\item there exists $c_1 > 0$ such that for any 
$s,k \in \NN$, 
we have
\begin{align*} 
\PP(\text{some }\cC\in\cycles_s\text{ is split
into }\cC', \cC'' \text{ s.t.\ }
\min(|\cC'|, |\cC''|)\leq k \mid \cycles_s) \leq c_1 \frac{k}{n}.
\end{align*}
\item there exists $c_2 > 0$ such that for any $s \in \NN$ and any two
cycles $\cC', \cC'' \in \cycles_s $ 
we have
\begin{equation*} 
\PP(\cC', \cC'' \text{ are merged} \mid \cycles_s) \geq c_2
\frac{|\cC'||\cC''|}{n^2}.
\end{equation*}
\end{enumerate} 
Then there exist $c_3, c_4>0$, depending only on $c_1,
c_2$, such that if
\begin{equation}\label{eq:growing_interval} t_1 := t_0 + \lceil\Delta
t\rceil, \quad \Delta t = c_3
\delta^{-1}\frac{n}{2^j}\log_2\left(\frac{n}{2^j}\right),
\end{equation} then
\begin{align}\label{eq:Schramm-ineq} 
\E \Big(|
V_{\cycles}^{{t_0}}(2^j) \setminus
V_{\cycles}^{{t_1}}(\varepsilon \delta n)|\; \big|\;
\cycles_{t_0} \Big) 
\one_{\{ |V_{\cycles}^{{t_0}}(2^j)| \geq \delta n\}} 
\le c_4 \delta^{-1}\varepsilon |\log_2(\varepsilon\delta)|
n.
\end{align}
\end{lemma}

Note that in our case condition (1) is satisfied because of Lemma
\ref{le:splits}
and condition (2) is trivially satisfied. In notation of \cite[Lemma 5.1]{A-K-M} this corresponds to taking the stopping time $\tau = + \infty$, i.e. conditions (1) and (2) hold for all times $s \in \mathbb{N}$.
The lemma states that if, at some time $t_0$, enough vertices are in
reasonably large cycles (size $\geq 2^j$) 
then at some carefully chosen later time most
of these vertices will be in cycles of size of the order $n$.
Here one should think of $2^j$ as approximately $n^{1/4}$ and 
of $t_0\geq c_0n$ for some $c_0>\tfrac12$.
Then $|V^{t_0}_G(2^j)|\approx zn$
by the Erd\H{o}s--R\'enyi theorem, hence by Lemma~\ref{lem:EV}
also $|V^{t_0}_\cycles(2^j)|\approx zn$.
Note that if $2^j=n^{1/4}$ then $\Delta t$ is of the order
$n^{3/4}\log n\ll n$, thus for any $c>\tfrac12$ and $t_1\geq cn$
we may select $c_0>\tfrac12$ such that 
$t_0=t_1-\lceil \Delta t\rceil\geq c_0n$.

The final result of this section paraphrases
\cite[equation (2.4) in Lemma 2.4]{schramm}.
It  tells us that most of 
the vertices in $V_G^{t}$ (the largest connected component 
in $G^{t}$) belong to large cycles.
The proof is precisely as in \cite{schramm} as the only appeal
to the particular structure of the cycles is through invoking the
previous lemmas, which all hold as in \cite{schramm}. 
\begin{lemma} \label{large-lem}
Fix $c>1/2$, take $t\geq cn$, $t\in\NN$. 
There is some $C_2>0$ such that
for any $\eps\in(0,1)$, if $n$ is large enough we have
\be\label{large1}
\EE\big[ |V_G^t\sm V_{\cycles}^t(\eps n)|\big]
\leq C_2 \eps \log(\tfrac1\eps) n.
\ee
\end{lemma}

\section{Exploration processes}
\label{exploration-sec}

An important tool in proving Theorem \ref{thm:main}
is the \emph{exploration process}, which we will define in this
section.  
The exploration process 
is sometimes also called the
\emph{cyclic-time random walk}, see e.g.\ \cite{H2,H-H}.  
It will allow us to uncover the loop
containing some specified point $(v_0,\ph_0)\in V_n\times S^1$ at the
same time as we uncover the configuration $\om\in\Om$ itself.  
We will also define a process which we call the 
\emph{simple exploration process} which is easier to analyse 
and which may be coupled with the exploration process.
In this
section we work with a random $\om\in\Om$ 
sampled from the Poisson
measure $\PP_\b$ for some fixed $\b>1$
(the definitions will make sense for all $\b>0$).
Recall that $\nu\in[0,1)$ is
fixed throughout.

\subsection{Definitions}

The exploration process will be denoted
$X(\om)=(X_t(\om):t\geq0)$ and takes values in 
$V_n \times S^1 \times \{-1, +1\}$.  Recall that we identify $S^1$ with
the unit interval $[0,1)$ using periodic boundary conditions. 
We will write $X_t=(v_t,\ph_t,d_t)$.
Let $d_0\in\{-1,+1\}$ be an initial direction and set
$X_0=(v_0,\ph_0,d_0)$, where $(v_0,\ph_0)\in V_n\times S^1$.  
The process starts by traversing $\{v_0\}\times S^1$ at unit speed
in the direction specified by $d_0$, meaning that
$v_t=v_0$, $\ph_t=\ph_0+d_0t$ and $d_t=d_0$. This continues
until it either encounters a link of
$\om$, or it returns to its starting point, i.e.\
$(v_t,\ph_t)=(v_0,\ph_0)$.  If a link is
encountered first, say at time $t$ and with other endpoint in 
$\{w\}\times S^1$, then the process jumps
to $\{w\}\times S^1$ and proceeds in a
direction which depends on whether the link was a cross or a bar. That
is, we set $v_t=w$ and
$\ph_{t+s}=\ph_t+d_0s$ and $d_{t+s}=d_0$ if the link was a cross or
$\ph_{t+s}=\ph_t-d_0s$ and $d_{t+s}=-d_0$ if the link was a bar.  We define the process
to be right-continuous (c\`adl\`ag).  The process proceeds in this
way, traversing links and adjusting its direction accordingly, until
it returns to the starting point $(v_0,\ph_0,d_0)$.  We let
\begin{equation} \label{eq:ExplorationDeath}
\tau^{X}=\tau^X(v_0,\ph_0,d_0):=\inf\left\{t>0:X_{t}=X_0 \right\}
\end{equation} 
be the time when this happens.
After this time  the process is no longer useful to
us, but to be definite we declare that the process continues by
repeating itself periodically after time $\tau^X$.
Note that at time $\tau^X$, the loop containing $(v_0,\ph_0)$ has
been fully discovered.

Let us consider those links that,
by time $\tau^X$, have  been traversed by $X$ at least once.
Some of them have been traversed only once, others twice (no link can
be traversed more than twice before time $\tau^X$ as this would entail
visiting a previously visited point $(v_t,\ph_t,d_t)$).
We say that a link is \emph{discovered} at the time of
its first traversal, and \emph{backtracked} on its second traversal 
(if traversed twice). 
Let $\cJ^X_t(v,\varphi)$ denote the number of times the exploration
$X$, started at $(v,\varphi)$ and run for time $t$, has
discovered a  link (`jumped').  Let 
$\cI^X_t(v,\varphi)$ denote the number of times  it has
traversed some link, including backtracking.
Thus $\cJ^X_t(v,\varphi)\leq \cI^X_t(v,\varphi)\leq 
2\cJ^X_t(v,\varphi)$ for all $t\leq \tau^X$.
Next, define the \emph{history} of $X$ as 
\begin{equation}\label{eq:Xhistory}
H_{t}^{X}:=\{(v,\varphi)\in V_n\times S^1\;\big|\; \exists\; s\leq t 
\text{ s.t. } X_s=(v,\varphi,d) \text{ for some }d\in\{-1,1\}\}. 
\end{equation}
This is the set of points in $V_n\times S^1$ visited by $X$ up to time
$t$.   Finally, let 
$\left\{ \cF_{t}\right\} _{t\geq0}$ denote the natural
filtration of the exploration process, namely
 $\cF_t:=\s\big((X_s)_{0\leq s\leq t}\big)$,
 and $\bar{\mathcal{F}}_{t}\eqdef\bigcap_{s>t}\mathcal{F}_{s}$.

When $\om$ is randomly 
sampled from the Poisson measure $\PP_\b$
we may, thanks to the memorylessness of Poisson processes, 
construct (part of) $\om$ itself simultaneously with $X$.  This fact
is central to our approach.
We formulate the construction as a proposition.
In the following result we will be using a Poisson process $N$ on
$[0,\oo)$ and we will say that $N$ \emph{rings at time} $t$
if it has an arrival at that time.
\begin{proposition}[Construction of the exploration process]
\label{prop:expl}
Let
$v_0\in V_n$, 
$\varphi_{0}\in S^{1}$ and $d_{0}\in\left\{ -1,1\right\} $.
Consider the following independent objects:
\begin{itemize}[leftmargin=*]
\item a Poisson process $N=(N_t:t\geq0)$ with intensity 
$n\tfrac\beta{n-1}$,
\item a sequence $\left\{ v_{i}\right\} _{i\in\mathbb{N}}$ of
i.i.d.\ random variables distributed uniformly on $V_n$,
\item a sequence $\left\{ \xi_{i}\right\} _{i\in\mathbb{N}}$ of i.i.d.\
random variables taking values $\pm 1$
and satisfying $\PP(\xi_{i}=+1)=\nu$.
\end{itemize}
When $\om$ has law $\PP_\b$, then the law
of the  exploration $X$,
started at $X_0=(v_{0},\varphi_{0},d_{0})$, may be
constructed as follows:
\begin{enumerate}[leftmargin=*]
\item The process starts at $X_0:=(v_{0},\varphi_{0},d_{0})$ and
  initially only $\varphi_t$ changes, according to $\ph_t=\ph_0+d_0t$.
\item Whenever $N$ rings, say at time $t$, we inspect vertex $w=v_{N_{t}}$.
We have two cases:
\begin{enumerate}
\item If $(w,\ph_{t-})\in H^{X}_{t-}$ or
$w=v_{t-}$ then nothing happens and the process continues on, i.e.\ $X_t=X_{t-}$.
\item
Otherwise we set
$X_t=(w,\varphi_{t-},\xi_{N_{t}}d_{t-})$ and then 
$\varphi_t$ evolves 
according to $\varphi_{t+s}=\varphi_{t-}+\xi_{N_t}d_{t-}s$.
\end{enumerate}
\item Between successive rings of $N$ the process
may backtrack across previously discovered links.
More precisely, a backtrack occurs at time $t$ if there exists another time
$s<t$ such that $v_s\neq v_{s-}$ and either:
\begin{itemize}
\item $(v_{t-},\varphi_{t-})=(v_{s-},\varphi_{s-})$, in which case we
  set $X_t=(v_s,\varphi_{s},\xi_{N_{s}}d_{t-})$, or 
\item $(v_{t-},\varphi_{t-})=(v_{s},\varphi_{s})$, in which case we
  set $X_t=(v_{s-},\varphi_{s-},\xi_{N_{s}}d_{t-})$.
\end{itemize}
See Figure \ref{fig:backtrack}.
\end{enumerate} 
\end{proposition}
 \begin{figure}[h!]
\includegraphics{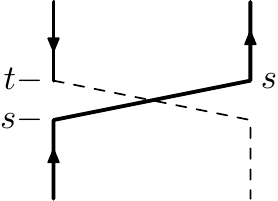} \qquad\qquad
\includegraphics{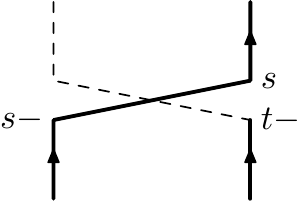} 
\centering
\caption{Two possibilities for backtracking a cross.  
Thick, solid lines belong to  the history $H^X_{t-}$,
dashed lines the future trajectory.
Left: $(v_{t-},\varphi_{t-})=(v_{s-},\varphi_{s-})$.  Right:
$(v_{t-},\varphi_{t-})=(v_{s},\varphi_{s})$.  
}
\label{fig:backtrack}
\end{figure}
The construction of Proposition \ref{prop:expl}
is fairly standard and has been used previously in
for example \cite{A-K-M,H2,H-H}, 
hence we do not give a proof. 
 Let us however draw attention to the condition that,
when $(w,\ph_{t-})\in H^{X}_{t-}$ or
$w=v_{t-}$, then the jump proposed by $N$ is canceled.  
This means that $X$ cannot jump to
a previously visited point
$(w,\ph)$, which effectively amounts to a
reduction of the intensity of jumps (see Lemma
\ref{le:intensity} below).

The main difficulty in analysing $X$ is that it may discover a new
link which takes it to a previously visited copy of $S^1$, 
i.e.\ it may jump at time $t$ to a point  $(w,\ph)$
satisfying $(\{w\}\times S^1)\cap H^X_s\neq\es$ 
for some $s<t$.  We refer to this as \emph{jumping to the history}.
In this case
$\{w\}\times S^1$ has already been partially explored, making
$X$ quite difficult to analyse directly.

To get around this problem we introduce what we call the
\emph{simple} exploration
process $Y=(Y_t:t\geq0)$, which is easier to analyse and 
(on time intervals which are not too long) 
can be coupled with the exploration process $X$.  Roughly speaking,
the idea is that for $Y$ we replace the vertex set $V_n$ with an
augmented vertex set $\NN\times V_n$, where the $\NN$-coordinate
increases on 
discovering a new link.  The interpretation is 
that each newly discovered link
brings us to a `fresh' circle $\{(k,v)\}\times S^1$.  

Below we give a detailed definition.  Notice that the wording is
very similar to Proposition \ref{prop:expl}, the main difference being
what happens at the jump times of $N$.
\begin{definition}[Simple exploration process]
\label{def:simpleexplorationprocess}
Let  $v_0\in V_n$, 
$\varphi_{0}\in S^{1}$ and $d_{0}\in\left\{ -1,1\right\} $.
We construct the simple exploration process 
$Y_t=(k_t,v_t,\varphi_t,d_t)$
as a c\`{a}dl\`{a}g process,  using the following independent objects:
\begin{itemize}[leftmargin=*]
\item a Poisson process $N=(N_t:t\geq0)$ with intensity 
$n\tfrac\beta{n-1}$,
\item a sequence $\left\{ v_{i}\right\} _{i\in\mathbb{N}}$ of
i.i.d. random variables distributed uniformly on $V_n$,
\item a sequence $\left\{ \xi_{i}\right\} _{i\in\mathbb{N}}$ of i.i.d.\
random variables taking values $\pm 1$
and satisfying $\PP(\xi_{i}=1)=\nu$. 
\end{itemize}
Using these sources of randomness the process is constructed as follows:
\begin{enumerate}[leftmargin=*]
\item The process starts at $Y_0:=(0,v_{0},\varphi_{0},d_{0})$ and
  initially only $\varphi_t$ changes, according to $\ph_t=\ph_0+d_0t$.
\item Whenever $N$ rings, say at time $t$, we inspect vertex $w=v_{N_{t}}$;
we have two cases:
\begin{enumerate}
\item If $w=v_{t-}$ then nothing happens and the process continues on,
  i.e.\ $Y_t=Y_{t-}$.
\item If $w\neq v_{t-}$ we set
$Y_t=(N_t,w,\varphi_{t-},\xi_{N_{t}}d_{t-})$ and then $\varphi_t$
evolves according to $\varphi_{t+s}=\varphi_{t-}+\xi_{N_t}d_{t-}s$.
\end{enumerate}
\item 
Between successive rings of $N$ the process
may backtrack across previously discovered links.
Now there is only one possibility for backtracking:
a backtrack occurs at time $t$ if there exists another time
$s<t$ such that $(k_s,v_s)\neq (k_{s-},v_{s-})$ and 
$(k_{t-},v_{t-},\varphi_{t-})=(k_s,v_{s},\varphi_{s})$, 
and then we
set $Y_t=(k_{s-},v_{s-},\varphi_{s-},\xi_{N_{s}}d_{t-})$.
\end{enumerate} 
\end{definition}

As we did for $X$ we let
\be\label{eq:simpleExplorationDeath}
\tau^{Y}:=\inf\left\{ t>0:Y_{t}=Y_0\right\}
\ee
be the first time at which the simple exploration process $Y$
arrives back at the starting point.  
Note that $\tau^Y$, in contrast to $\tau^X$, may take the value
$+\oo$,  see Proposition \ref{fact:loop_closure}.
If $\tau^Y<\oo$
we assume that the process $Y$ stops evolution after
$\tau^Y$.

For $Y$ we define the history by 
\begin{equation}\label{eq:Yhistory}
H_{t}^{Y}:=
\{(k,v,\varphi)\in \NN\times V_n\times S^1
\big| \exists s\leq t \text{ s.t. } Y_s=(k,v,\varphi,d) 
\text{ for some }d\in\{-1,1\}\}.
\end{equation}
By slight abuse of terminology, 
if $Y_0=(0,v,\varphi,d)$ for some $d\in\{-1,+1\}$
then we say that $Y$ \emph{started at} $(v,\varphi)$. As for $X$ we denote by $\cJ^Y_t(v,\varphi)$ 
(respectively, $\cI^Y_t(v,\varphi)$) the number of times
the simple exploration process has discovered 
(respectively, traversed) a link when 
started at $(0,v,\varphi)$ and run for time $t$.
We denote by $Y'$ 
the restriction of $Y$ to $V_n\times S^1\times\{-1,+1\}$, that is if
$Y_t=(k,v,\ph,d)$ then $Y_t'=(v,\ph,d)$.

The important point which makes $Y$ simpler to analyse than $X$ is
that, each time $Y$ discovers a new link 
(i.e.\ $N$ rings and $v_{N_t} \neq v_{t-}$), we set
$k_t$ to a previously unused value, namely $N_t$.
This means that $Y$, by
construction, can never jump to its history.  Crucially, it does still
backtrack across previously discovered links.

\subsection{Coupling with the simple exploration process}

The following lemma shows that when $\om$ is randomly sampled from
$\PP_\b$, one can couple the exploration process $X$ and the simple
exploration process $Y$ so that they evolve in the same way on a
sufficiently short time scale. One should think of $T=o(n^{1/2})$.

\begin{lemma}[Coupling with the simple exploration process]
\label{fact:coupling}
Fix $T>0$.
Let $X$ be the exploration process 
and let $\sigma$ be a stopping time
with respect to the filtration
$\{\bar{\mathcal{F}}_t\}_{t\geq 0}$ such that $X$ 
jumps to a previously unvisited vertex 
at time $\sigma$.  Conditionally on $\bar{\mathcal{F}}_{\sigma}$,
there exists a coupling $\PP$ of
a process $\tilde{X}$ with a process $Y$ such that:
\begin{enumerate}[leftmargin=*]
\item 
The process $Y$ is
a simple exploration starting from 
$Y_{0}=(0,X_{\sigma}) = (0, \tilde v_0, \tilde\varphi_0, \tilde d_0)$
where
$X_{\sigma} = (\tilde v_0, \tilde\varphi_0, \tilde d_0)$.
\item $\PP\big(\big\{ \tilde{X}_{t}\big\} _{t\geq0}\in\cdot\mid
 \bar{\mathcal{F}}_{\sigma}\big)
=\PP\big(\left\{ X_{t+\sigma}\right\}_{t\geq0}\in\cdot\mid
 \bar{\mathcal{F}}_{\sigma}\big)$.
\item 
$\PP\big(\forall_{t<\tau^{Y}\wedge T}\;\tilde{X}_{t}=Y_{t}' \mid
 \bar{\mathcal{F}}_{\sigma}\big)
\geq1-4\b T (\cJ^X_\s+\b T)/n$.
\end{enumerate}
\end{lemma} 

If at some time $t\leq\tau^{Y}$ we
have $\tilde{X}_{t}\neq Y_{t}'$ then we consider the
coupling as \emph{failed} at time $t$. 
The history $H^{\tilde X}_t$ of $\tilde X$ is defined 
as in \eqref{eq:Xhistory}. 

\begin{proof} 
In the following we write simply $\PP(\cdot)$ for 
$\PP(\cdot\mid \bar{\mathcal{F}}_{\sigma})$.
We will construct 
$\big\{\tilde{X}_{t}\big\} _{t\geq0}$ using the same sources of
randomness as for $Y$,
namely the same $N$, $\{ v_{i}\} _{i\in\mathbb{N}}$ and 
$\{\xi_{i}\} _{i\in\mathbb{N}}$ as given
in Definition \ref{def:simpleexplorationprocess}.
We write $\tilde X_t=(\tilde v_t,\tilde \varphi_t,\tilde d_t)$.

\begin{enumerate}[leftmargin=*]
\item The process $\tilde{X}$ starts at $\tilde X_0:=X_{\sigma}$ 
and initially only $\tilde \varphi_t$ changes,
by $\tilde \varphi_t=\tilde\varphi_0+t\tilde d_0$.
\item Whenever $N$ rings, say at time $t$, we inspect vertex $w=v_{N_{t}}$;
we have two cases:
\begin{enumerate}
\item If $(w,\tilde\ph_{t-})\in H^{\tilde X}_{t-}\cup H^X_\s$ 
or $w=\tilde{v}_{t-}$ then nothing happens and the process
continues on.
\item Otherwise the process jumps to 
$(w,\tilde{\varphi}_{t-},\xi_{N_{t}}\tilde{d}_{t-})$.
\end{enumerate}
\item Between successive rings of $N$ the process may backtrack as
before, using links of both $X$ and $\tilde X$.
\end{enumerate} 
It follows from Proposition \ref{prop:expl}
that $\tilde X$ and $X$ have the same distribution, 
giving statements (1) and (2) from the lemma.

For the proof of (3), let
\[\begin{split}
V^X_{\s}&:=\{v\in V_n : (v,\ph)\in H^X_s\mbox{ for some } s<\s, \ph\in
S^1\},\\
V^{\tilde X}_{t}&:=\{v\in V_n : (v,\ph)\in H^{\tilde X}_s\mbox{ for some } s<t, \ph\in
S^1\}
\end{split}\]
be the sets of vertices visited by $X$ up to time 
$\sigma$ and by $\tilde{X}$ up to time $t$, respectively.  
We define 
\[
\rho:=\inf\big\{ t\geq0: N \mbox{ rings at time }t\mbox{ and }
v_{N_t}\in   V_{t}^{\tilde{X}}\cup V_{\s}^{X}\big\}.
\]
Until time $\rho\wedge\tau^{Y}$ the processes $\tilde{X}_{t}$ and
$Y_{t}'$ are equal, thus
\[
\PP\big(\forall_{t<\tau^{Y}\wedge T}\tilde{X}_{t}=Y_{t}'\big)
\geq\PP(\rho\geq T).
\]
Each time $t$ when $N$ rings there is a  chance 
that $v_{N_t}\in   V_{t}^{\tilde{X}}\cup V_{\s}^{X}$.
This has probability at most the number of previously visited
vertices divided by $n$.
Since the number of visited vertices cannot
exceed the number of discovered links by more than 1, we get
\[\begin{split} 
\PP(\rho< T)
&\leq \E\Bigg[\sum_{i=1}^{N_{T}}
\frac{(\mathcal{J}_{\sigma}^{X}+i)}{n}\Bigg]
=\frac{\mathcal{J}_{\sigma}^{X}}{n}\E[N_T]
+\frac{1}{2n}\E\big[N_T(N_T+1)\big]
\\
&=\frac{\mathcal{J}_{\sigma}^{X}\beta T}{n-1}+\frac{\beta T}{n-1}+
(\beta T)^2\frac{n}{2(n-1)^2},
\end{split}
\]
where in the last equality we used that $N_T$ has Poisson 
distribution with mean $n\tfrac\b{n-1} T$.
Since $\cJ^X_\s\geq1$ this gives the claimed bound on 
$\PP\big(\forall_{t\leq\tau^{Y}\wedge T}\tilde{X}_{t}=Y_{t}'\big)$.
\end{proof}

\subsection{Properties of the exploration processes}

Next we present some basic properties of
the processes $X$ and $Y$, starting with the simple exploration $Y$. 

First note that $\cJ^Y_t$, the number of links discovered by time $t$,
is a Poisson process with rate $\b$, stopped at time $\tau^Y$
(at which time $Y$ itself terminates).  It will be convenient to
extend this process beyond time $\tau^Y$.  For this purpose we let
$N'_t$ denote a Poisson process of rate $\b$ which agrees with
$\cJ^Y_t$ up to time $\tau^Y$.

Many relevant properties of $Y$
can be understood in terms of the process $Z$
given by $Z_{t}\eqdef N'_{t}-t$.
For example, if $Z_t$ hits $-1$ then this
corresponds to $Y$ returning to its starting point, that is to say we
have that 
$\tau^Y=\inf\{t\geq0:Z_t=-1\}$.
To see this, note that  $N'_t+1$ counts the number 
of copies of $S^1$ that $Y$ has visited by time $t$. 
If $Z_t=-1$
then $N'_t+1=t$, which means that the total time 
spent equals the number of
$S^1$'s visited.  Hence at this time $Y$ has 
explored the entirety of each 
copy of $S^1$ it has visited, meaning that it must have
returned to its starting point.   See Figure \ref{fig:Z-graph}
 \begin{figure}[h!]
\includegraphics{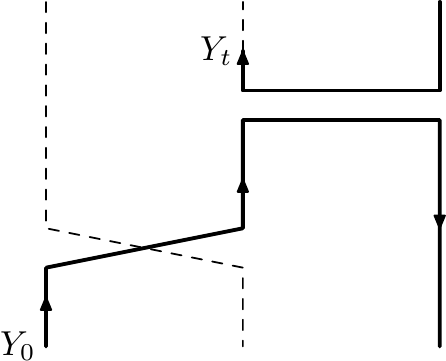} 
\qquad\qquad
\includegraphics{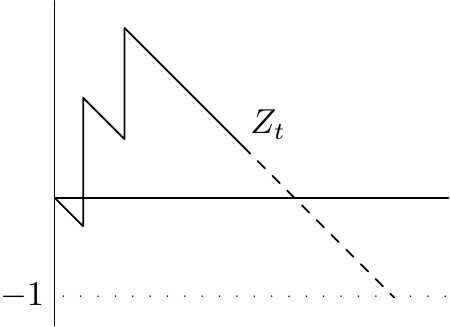} 
\centering
\caption{Left:  in bold, simple exploration $Y$ up to time $t$;
dashed, the future trajectory of $Y$ assuming no more links are
discovered.  Right:  corresponding plot of $Z$, dashed line giving
extrapolation until time $\tau^Y$ assuming no more links are
discovered. 
}
\label{fig:Z-graph}
\end{figure}

Note that $\b>1$ implies that $Z_t\to+\oo$ almost surely.  We define a
sequence of random times which we call \emph{frontier times} $\ell_k$,
as well as processes $Z^{(k)}=(Z_{\ell_k+t}-Z_{\ell_k})_{t\geq0}$, as
follows.  First, we let $\ell_0:=0$ and $Z^{(0)}:=Z$.  Next, we let
$\ell_1$ be the time when $Z_{t-}$ attains its \emph{global} minimum
(note that as $Z_{t} \to +\oo$ almost surely, this time is almost
surely finite).  This is necessarily a jump time of $Z$ (equivalently,
of $N'$) but it is not a stopping time.  Inductively, $\ell_{k+1}$ is
the time when $Z^{(k)}_{t-}$ attains its global minimum.  We also
write $\D_k=\ell_{k+1}-\ell_k$ for the time spent between successive
frontier times. See Figure \ref{drift-fig} for a sample trajectory of the process $Z_t$ with frontier times marked.

\begin{figure}[hbt]
\centering
\includegraphics[scale=.7]{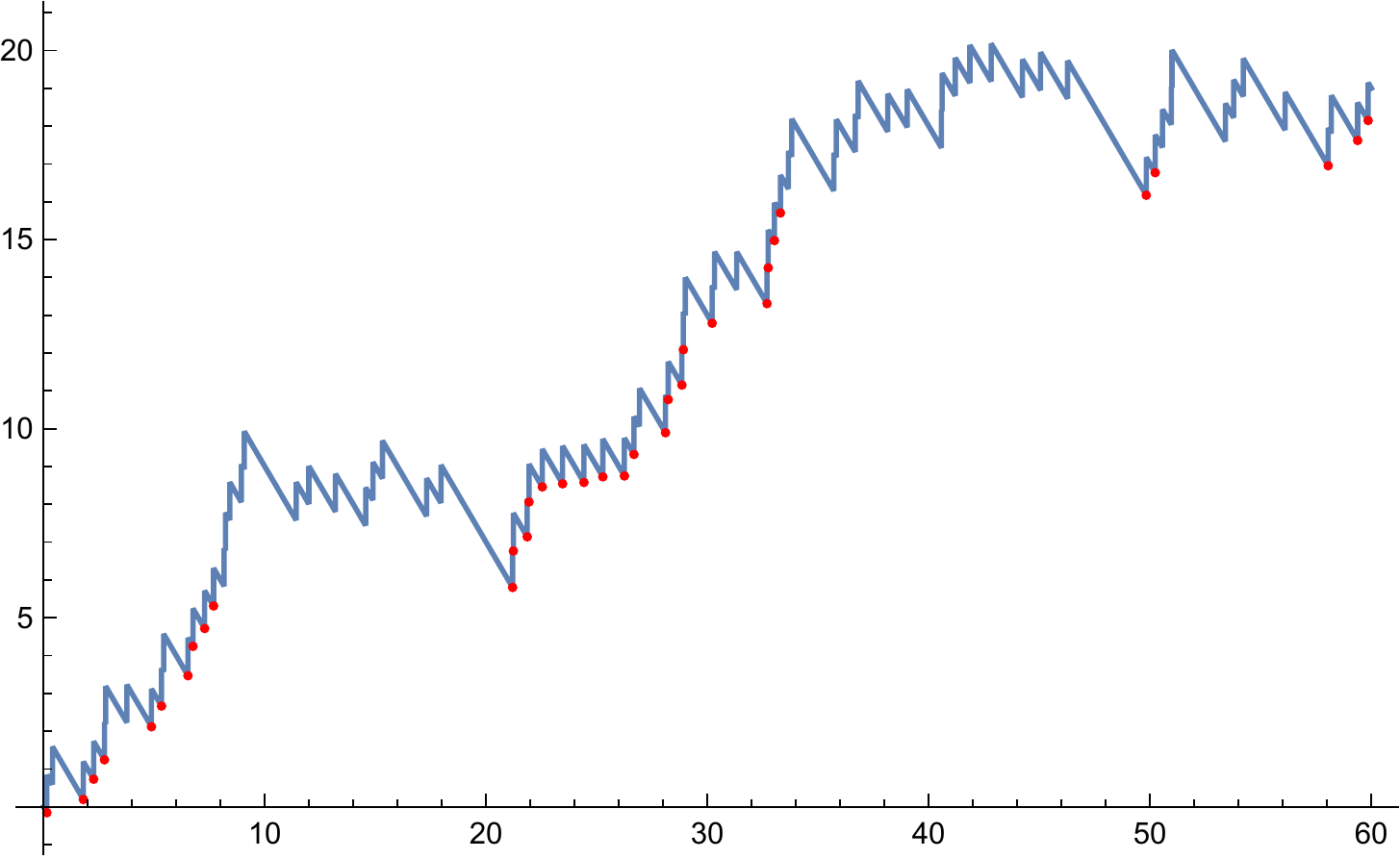}
\caption{A simulation of the process $Z_t$ with $\b=1.3$.  The frontier times are
  marked red.}
\label{drift-fig}
\end{figure}

In terms of the simple exploration
$Y$, the frontier times $\ell_k$ play the following role.
Recall that the jump times of $Z$ are exactly 
the times when $Y$ discovers a new link.  The frontier times are the
times when $Y$ discovers a new link 
\emph{which is never backtracked}.

\begin{proposition}[Survival and increments of the simple exploration]
\label{fact:loop_closure}
Let  $\beta>1$ and write $\cS=\{\tau^Y=\oo\}$.
\begin{enumerate}[leftmargin=*]
\item We have
that $\PP(\cS)=z$,
where $z$ is the unique positive solution of $1-z=e^{-\beta z}$. 
\item 
There exists $C,c>0$ such that
\begin{equation}\label{eq:tauYtails}
\PP(\cS^c\cap \{\tau^{Y}\geq t\})\leq Ce^{-ct}.
\end{equation}
\item Conditionally on $\cS$,
the sequence 
$\big\{\big(\D_{k},(Z^{(k)}_t)_{0\leq t<\D_k}\big)\big\}_{k=0}^{+\infty}$
is i.i.d. 
\item There exists $C,c>0$ such that for any $k \geq 0$
\be \label{eq:Dtails}
\PP(\Delta_{k}\geq t\mid \cS)\leq Ce^{-ct},\quad t\geq0.
\ee
\end{enumerate}
\end{proposition}
The proof is based on well-known properties of Poisson processes, for
completeness we provide details 
in Appendix \ref{sep-app}.
The first two parts of the proposition tell us
that the simple exploration either continues indefinitely,
or it closes `quickly'. 
 Intuitively, the former scenario
parallels the situation when the (true) exploration process $X$
explores a large cycle.
The other two parts tell us that, conditionally on $Y$ `surviving',  
the frontier times $\ell_k$ are
renewal times, and the renewal intervals $\D_k$ are typically 
short.

We now turn to discussing some properties of the 
exploration process $X$.  
Recall that
$\mathcal{J}^X_t := \#\{s\le t\colon \textrm{$X$ discovers a new link at time
$s$}\}$.  
Let 
\[
\mathcal{N}^X_{t}:= 
\#\{v\in V_n: (v,\ph)\in H^X_t\mbox{ for some } \ph\in
S^1\}
\]
denote the total number of vertices 
visited by $X$ up to time $t$, and let
$\mathcal{A}^{\his}_{t} := \{ \cN^X_t \leq \tfrac{n}2 \}$
denote the event that no more than $n/2$ vertices have been
visited up to time $t$.
Note that $\cN^X_t\leq \cJ^X_t+1$ and that $\cJ^X_t\leq N_t$
where $N$ is the Poisson process of rate $n\tfrac\b{n-1}$
in Proposition \ref{prop:expl}.  From this and a simple argument using
Laplace transform we see that
\be\label{Ahis-bound}
\PP((\cA^\his_t)^c)=
\PP(\cN^X_t>\tfrac{n}2)\leq 
\exp\big(-\tfrac n3(\log (\tfrac{n-1}{3t\b})-1)\big).
\ee
In particular, if $t=o(n)$ then 
$\PP(\cA^\his_t)\geq 1-e^{-c n}$ for some $c>0$.

By $A^X_{t}$ we will denote the set of vertices available to the
exploration $X$ at time $t$ by means of a new jump, 
i.e.\ $A^X_t = \emptyset$ if  $t\geq \tau^X$, otherwise if 
$X_t=(v,\varphi,d)$ then
\[
A^X_t := \{w \in V_n\setminus \{v\} \colon 
%\forall_{\varphi\in S^1} \, 
(w,\varphi) \notin H^{X}_t \}.
\]
Recall that a 
\emph{counting process} is a nondecreasing, integer valued
c\`adl\`ag stochastic process starting at zero and with jumps equal to
one.  Let $J$ be an $\mathcal{F}_t$-adapted counting process. 
We will
say that a nonnegative process $\lambda$ is an 
\emph{intensity} of $J$ if
$\lambda$ is $\mathcal{F}_t$-progressively measurable, 
$\int_0^t\lambda_u du < \infty$ a.s.\ for all $t$,
and the process $J_t - \int_0^t\lambda_s \, ds$ 
is an $\mathcal{F}_t$-martingale. 

\begin{lemma}[Intensity of jumps]\label{le:intensity}
The processes $\mathcal{J}^X$ and 
$\mathcal{N}^X$ are counting processes with intensities
$\lambda$, $\mu$ given respectively by
\begin{displaymath}
	\lambda_t = \frac{\beta}{n-1} |A^X_t|
\qquad \mbox{and}\qquad
	\mu_t = \frac{\beta}{n-1} (n-\cN^X_t).
\end{displaymath}
In particular, on the event $\mathcal{A}^{\his}_{t}$ we have
$\mu_t \geq \tfrac{\beta}{2}$.
\end{lemma}

A proof of this rather intuitive statement may be
found (in a more general setting) in \cite[Lemma 3.7]{A-K-M}.
The following lemma also appears in a more general form in
\cite[Lemma A.2]{A-K-M}, we include its proof here for the sake of
completeness. 

\begin{lemma}\label{le:intensity-visits}
Suppose $M$ is a counting process with intensity $\lambda$ and let
$\Lambda_{t} = \int_{0}^{t} \lambda_s \, ds$. Let $\sigma$, $\tau$ be
stopping times such that $\sigma \leq \tau$. Let $\ell > 0$. Then we
have
\[ \PP \left( \{ M_{\tau} - M_{\sigma} \leq \ell / 2 \} \cap \{
\Lambda_{\tau} - \Lambda_{\sigma} \geq \ell \} |
\mathcal{F}_{\sigma}\right) \leq e^{- \ell / 8}.
\]
\end{lemma}

\begin{proof} Consider any $A \in \mathcal{F}_\sigma$ with positive
probability and the process $\widetilde{M}_t = M_{\sigma+t} -
M_\sigma$, which is a counting process with intensity
$\widetilde{\lambda}_{t} = \lambda_{\sigma+t}$ with respect to the
filtration $\widetilde{\mathcal{F}}_t = \mathcal{F}_{\sigma+t}$. Let
$\widetilde{\PP}(\cdot) = \PP(\cdot|A)$. We have
$\widetilde{\Lambda}_t = \int_0^t \widetilde{\lambda}_s ds =
\Lambda_{\sigma+t} - \Lambda_\sigma$. Let $N$ be a Poisson process
with intensity $1$ such that $N_t = M_{\Lambda_t}$ almost surely
(see \cite[Theorem A.1]{A-K-M} and references there). We
get
\begin{align*} \widetilde{\PP}(\left\{ M_{\tau} - M_{\sigma}\leq\ell /
2\right\} \cap\left\{ \Lambda_{\tau} -
\Lambda_{\sigma}\geq\ell\right\}) &=
\widetilde{\PP}(\{\widetilde{M}_{\tau-\sigma} \le \ell/2\}\cap
\{\widetilde{\Lambda}_{\tau-\sigma} \ge \ell\}) \\ &\le
\widetilde{\PP}(N_\ell \le \ell/2).
\end{align*} Using the form of the Laplace transform of $N_\ell$ and
Chebyshev's inequality we obtain
\begin{align*} \widetilde{\PP}(N_{\ell} \le
\ell/2)&\leq\inf_{a\geq0}\exp\left((e^{-a}-1)\ell+a\ell/2\right)\\
&\leq\inf_{a\geq0}\exp\left(\frac{1}{2}a^{2}\ell-
\frac{a\ell}{2}\right)=e^{-\ell/8},
\end{align*} where in the second step we have used the elementary
inequality $e^{-a}-1+a\leq\frac{1}{2}a^{2}$ valid for $a\geq0$.  Thus
we get
\begin{displaymath} \PP(\left\{ M_{\tau}-M_{\sigma}\leq\ell/2\right\}
\cap\left\{ \Lambda_{\tau} - \Lambda_{\sigma}\geq\ell\right\}|A) \le
e^{-\ell/8},
\end{displaymath} for arbitrary $A \in \mathcal{F}_\sigma$ of positive
probability, which implies the lemma.
\end{proof}

\begin{corollary}[Visits to previously unvisited vertices]
\label{cor:visits-intensity}
Let $\sigma$ be a stopping time 
with respect to the filtration of the exploration process $X$ 
and let $\eta^{X}(\sigma)$ be the
first time after $\sigma$ when $X$ makes a jump to a previously
unvisited vertex.
For any $t > 0$
on the event $\mathcal{A}^{\his}_{\sigma}$ we have
\[ 
\PP \left( \eta^{X}(\sigma) \wedge \tau^X - \sigma \geq t \mid
\mathcal{F}_{\sigma}\right) \leq e^{-t/16} %C e^{- ct}.
\]
\end{corollary}

\begin{proof} 
By  definition of $\eta^{X}(\sigma)$,
between times $\sigma$ and $\eta^{X}(\sigma) \wedge \tau^{X}$ there
are no jumps to previously unvisited vertices.
In particular $\eta^{X}(\sigma) \wedge \tau^X -
\sigma \geq t$ implies that $\mathcal{N}^X_{\sigma + t} -
\mathcal{N}^X_{\sigma} = 0$ and that $\mathcal{A}^{\his}_{\sigma + t}$
holds. Thus Lemma \ref{le:intensity} implies,
with $\mu_{t}$ being the
intensity of $\mathcal{N}^X_{t}$ and $\Lambda_{t} = \int_{0}^{t}
\mu_{s} \, ds$, that then $\Lambda_{\sigma + t} - \Lambda_{t} \geq
t/2$. Applying Lemma \ref{le:intensity-visits} with $M_{t} =
\mathcal{N}^{X}_{t}$, $\tau = \sigma + t$ and $\ell = t/2$ easily
gives the desired estimate.
\end{proof}

\section{Balance}\label{sec:balance}

This section contains the main work
of the paper.  The goal of the section is to prove that large cycles
are
`balanced' in the sense that they contain roughly equal numbers of
vertices passed in the directions $\up$ and $\dn$. In fact we show
that, with high probability, in a cycle which is at least 
$\floor{n^{1/2}}$ long each segment of $\floor{n^{1/2}}$
consecutive vertices  is balanced in this sense.   Throughout the
section we work with a random $\om\in\Om$ sampled from the Poisson
measure $\PP_\b$ for some fixed $\b>1$ (recall that $\nu\in[0,1)$ is
fixed).

We start by introducing some notation.
Given  $\om\in\Om$, $v\in V_n$ and $k\in\mathbb{N}$,
let us write 
$\cC_\om(v)=(v_1^{d_1},v_2^{d_2},\dotsc,v_\ell^{d_\ell})\in \cycles_\om$
for the cycle containing $v$.  Without loss of generality we may
assume that $v=v_1$ and
that  $d_1=\up$.  Under these assumptions, we let 
\[
c_\om(v,k):=\{v_i\in\cC_\om(v):1\leq i\leq k\wedge \ell\}
\]
denote the set of the first $k$ vertices of $\cC$ following $v$.
Note that if the cycle containing $v$
has length smaller than $k$ then 
$|c_{\omega}(v,k)|=|\cC_\om(v)|<k$. 
We further let
\[\begin{split}
c^\up_\om(v,k):=\{v_i\in c_\om(v,k):d_i=\up\},\qquad
c^\dn_\om(v,k):=\{v_i\in c_\om(v,k):d_i=\dn\}
\end{split}\]
denote those vertices in $c_\om(v,k)$ which are passed in the
same, respectively opposite, direction as $v$.
Finally we define the \emph{balance}
$\bal_\omega(v,k)$ of the segment of length $k$ after $v$ in the loop, as
\begin{equation*}
\bal_\omega(v,k):=|c_{\omega}^{\uparrow}(v,k)|-
|c_{\omega}^{\downarrow}(v,k)|.
\end{equation*}

The main result of this section is the following proposition, which tells us that 
$\bal_\omega(v,\floor{n^{1/2}})$ is typically of much smaller order
than $n^{1/2}$.

\begin{proposition}[Segments of cycles are balanced]
\label{prop:preBalanceFact} 
Let $\beta_{1}>\beta_{0}>1$. There
exist $C,c>0$ such that for all 
$\beta\in[\beta_{0},\beta_{1}]$ and
for any $v\in V_n$, we have
\[ 
\PP_{\beta}\left(\left\{ 
\big|\bal_\omega(v,\floor{n^{1/2}})\big|
\geq n^{5/12}\log^{3}n\right\} \cap
\left\{ |\cC_\om(v)|\geq \floor{n^{1/2}}\right\}
%\left\{ |c_{\omega}(v,\floor{n^{1/2}})|=\floor{n^{1/2}}\right\}
\right)\leq Ce^{-c\log^{2}n}.
\]
\end{proposition}

This says that cycles of length at least $\floor{n^{1/2}}$
are very likely to have balance 
$|\bal|<n^{5/12}\log^{3}n\ll n^{1/2}$.
Cycles containing
fewer than $\lfloor n^{1/2}\rfloor$ vertices
may possibly be unbalanced, but this does not concern us.
A key feature of this result is that the upper bound kills
any polynomial in $n$, making it possible to use quite crude union
bounds later in the proof of Theorem \ref{thm:main}.
Also note that the bound is claimed to be uniform in $\b\in[\b_0,\b_1]$
for any $\b_1>\b_0>1$. This will allow us to
derive a version of the proposition stated above where we
`remove' a deterministic number of links from $\om$,
which will be important for the coupling
with PD($\tfrac12$) in Section \ref{sec:PDcoupling}.

To formulate the last claim 
precisely, recall the notation $\vec\om$ for the ordered
list of links of $\om$, and note that $\cC_\om(v)$, $c_{\omega}(v,k)$,
$c^{\up}_{\omega}(v,k)$, $c^{\dn}_{\omega}(v,k)$ and
$\bal_\omega(v,k)$ all depend on $\vec\om$ only.
Also recall that, for $s\leq|\om|$, we write 
$\vec{\omega}_{s}=\left\{ (e_{1},m_{1}),\ldots,(e_{s},m_{s})\right\}$
for the sequence of the first $s$ links.  If $X_\om$ is a random
variable which only depends on the relative order of links in $\om$ we
write $X_{\vec\om}$ for its value on any link configuration with the
same relative order.
\begin{proposition}
\label{prop:balanceFact}
Let $\beta>1$ and $\rho\in[0,1)$.  For $1\leq s\leq |\om|$ write 
\[
\cA(v,s):= \left\{ \bal_{\vec{\omega}_s}(v,\lfloor
  n^{1/2}\rfloor)\geq n^{5/12}\log^{3}n\right\} 
\cap
\left\{ |\cC_{\vec{\omega}_{s}} (v)|\geq \floor{n^{1/2}}\right\}.
%\left\{ |c_{\vec{\omega}_{s}}(v,\lfloor n^{1/2}\rfloor)|=
%\lfloor n^{1/2}\rfloor\right\}.
\]
There exist $C,c>0$ such that 
\begin{equation*}
\PP_{\beta}\left(\bigcup_{v\in V_n}\bigcup_{ |\omega|-n^{\rho}\leq s\leq |\omega|}
\cA(v,s) \right) \leq Ce^{-c\log^{2}n}.
\label{eq:balanceFact}
\end{equation*}
\end{proposition}

\noindent
The proofs are given at the end of the section, after
several preparatory results.

\subsection{Winding processes}

Recall the notation $X_t=(v_t,\varphi_t,d_t)$ and 
$Y_t=(k_t,v_t,\varphi_t,d_t)$ for the exploration and simple
exploration, in particular that $d_t\in\{-1,+1\}$ indicates the direction
of motion.
We will use superscripts $X$ and $Y$ on $v_t,\varphi_t,d_t$
to distinguish between the two processes.
Define  the \emph{winding processes} 
$\left\{ \mathcal{L}_{t}^{X}\right\}_{t\geq0}$
and 
$\left\{ \mathcal{L}_{t}^{Y}\right\} _{t\geq0}$ by
\[
\cL^X_t:=\int_0^t d^X_s\,ds,\qquad
\cL^Y_t:=\int_0^t d^Y_s\,ds.
\]
Thus $\cL^X$ increases at rate $1$ when the process $X$
travels in the positive direction,
otherwise it decreases at rate $1$, and the same is true for $\cL^Y$.  

To prove Proposition \ref{prop:preBalanceFact} 
we will first estimate $\cL^X$, and then transfer these estimates to $\bal$.
In order to estimate $\cL^X$ we will use the coupling of $X$ and $Y$
introduced in Lemma \ref{fact:coupling}, 
together with the following estimate on $\cL^Y$:

\begin{proposition}[Winding of the simple exploration process]
\label{prop:windingconcentration}
Let $\beta_{1}>\beta_{0}>1$.  There exist $C,c>0$
such that for any $\beta\in[\beta_{0},\beta_{1}]$, $T>0$ and
$s\in[1,T^{1/2}]$ we have
\[
\PP_\b\Big(\sup_{t\leq T}|\mathcal{L}_{t}^{Y}|\geq sT^{1/2}
\,\big|\, \cS\Big)\leq C\exp(-cs^{2}),
\]
where $\cS=\{\tau^{Y}=+\infty\}$.
\end{proposition}
\begin{proof}
Fix $\b\in[\b_0,\b_1]$.
We use Proposition \ref{fact:loop_closure}  and the notation therein,
we also write 
$\ol{\PP}(\cdot)=\PP_\b(\cdot\mid\cS)$.
For lighter notation, within this proof let 
$\cL_t=\cL^Y_t$.

At each frontier time $\ell_{k}$ the process $Z$ jumps, 
meaning that $Y$ discovers a new link.  We let
$\ell_0^\ast = 0$ and let 
$\{\ell^\ast_{k}\} _{k\geq1}$ be the subsequence consisting of the  
times $\ell_{k}$
at which the link is marked as a bar (i.e.\ $\xi_i=-1$ in the notation of
Definition \ref{def:simpleexplorationprocess}).   
As the choice of markings is independent
of $Z$, using Proposition \ref{fact:loop_closure} 
we conclude that $\D^\ast_{k}\eqdef \ell^\ast_{k+1}-\ell^\ast_{k}$
form an i.i.d.\ sequence under $\ol\PP$, satisfying
\begin{equation*}
\ol\PP(\D^\ast_{k}\geq s)\leq \tilde{C}e^{-\tilde{c}s},
\quad s\geq0, \label{eq:tildeDeltaTails}
\end{equation*}
for some $\tilde{C}, \tilde{c} > 0$. Also, the 
increments $\mathcal{L}_{\ell^\ast_{k+1}}-\mathcal{L}_{\ell^\ast_{k}}$
are independent under $\ol\PP$. 

Now, the key observation is that we have the equality in distribution
\[
\mathcal{L}_{\ell^\ast_{k+2}}-\mathcal{L}_{\ell^\ast_{k+1}}\eqdist
-\big(\mathcal{L}_{\ell^\ast_{k+1}}-\mathcal{L}_{\ell^\ast_{k}}\big),
\]
because upon crossing a bar
the winding processes changes its orientation.
Using these facts we infer that for any $k\in\left\{ 0,1,\ldots\right\} $
\[
\cQ_{k}:=\mathcal{L}_{\ell^\ast_{k+2}}-\mathcal{L}_{\ell^\ast_{k}}=
\big(\mathcal{L}_{\ell^\ast_{k+2}}-\mathcal{L}_{\ell^\ast_{k+1}}\big)
+\big(\mathcal{L}_{\ell^\ast_{k+1}}-\mathcal{L}_{\ell^\ast_{k}}\big)
\]
are symmetric random variables with $\left\{ \cQ_{2k}\right\} _{k\geq0}$
being independent. Moreover, 
by $|\cL_{\ell^\ast_{k+1}}-\cL_{\ell^\ast_{k}}|\leq \D^\ast_{k}$
and \eqref{eq:tildeDeltaTails} they have exponential tails. 

Let us set $K=\floor{c_{1}T},$ for some $c_{1}>0$ to be chosen later. 
We consider the
cases $\sum_{i=0}^{2K}\D^\ast_{i}> T$ and
$\sum_{i=0}^{2K}\D^\ast_{i}\leq T$ separately.
If $\sum_{i=0}^{2K}\D^\ast_{i}> T$ then we can cover $[0,T]$
with the intervals $[\ell^\ast_{2k},\ell^\ast_{2k+2})$ for $k\leq K$. As $\cL_t$
is continuous in $t$, we can replace the supremum
of $|\mathcal{L}_t|$ by maximum. 
The maximum $\max_{0\leq t\leq T}\cL_t$ 
can then be bounded by the maximum at endpoints
$\ell^\ast_{2k}$ plus the maximum increment over all the
intervals $[\ell^\ast_{2k},\ell^\ast_{2k+2})$.
The maximum increment on $[\ell^\ast_{2k},\ell^\ast_{2k+2})$
is in turn bounded by the length 
$\D^\ast_{2k}+\D^\ast_{2k+1}$.  We thus get 
\begin{align*}
\ol{\PP}\left(\max_{t\leq T}|\mathcal{L}_{t}|\geq sT^{1/2}\right)
& 
\leq\ol{\PP}\left(\max_{k\leq K}\Big|\sum_{i=0}^{k}\cQ_{2i}\Big|+
\max_{k\leq K} \big(\D^\ast_{2k}+\D^\ast_{2k+1}\big)\geq
  sT^{1/2}\right)\\
&\qquad+\ol{\PP}\left(\sum_{i=0}^{2K}\D^\ast_{i}\leq T\right).
\end{align*}
The first of these two terms can be bounded by applying a union bound
and Etemadi's inequality \cite[Thm.\ 22.5]{billingsley},
giving
\begin{multline}\label{3terms}
\ol{\PP}\left(\max_{k\leq K}\Big|\sum_{i=0}^{k}\cQ_{2i}\Big|+
\max_{k\leq K} \big(\D^\ast_{2k}+\D^\ast_{2k+1}\big)\geq
  sT^{1/2}\right)\\
\leq 
3\max_{k\leq K }
\ol{\PP}\left(\Big|\sum_{i=0}^{k}\cQ_{2i}\Big|\geq \tfrac16sT^{1/2}\right)
 +K\ol{\PP}\left(\D^\ast_{0}+\D^\ast_{1}\geq \tfrac12sT^{1/2}\right),
\end{multline}
where we have used that $\D^\ast_{k}$ are i.i.d. By symmetry of $\cQ_{2i}$ and Markov's inequality we have 
for any $\theta>0$ that
\[
\ol{\PP}\left(\Big|\sum_{i=0}^{k}\cQ_{2i}\Big|\geq \tfrac16sT^{1/2}\right)
=2\ol{\PP}\left(\sum_{i=0}^{k}\cQ_{2i}\geq \tfrac16sT^{1/2}\right)
\leq2\frac{\left(\ol{\E}[e^{\theta \cQ_{0}}]\right)^{k}}
{\exp(\tfrac16\theta sT^{1/2})}.
\]
For small enough $\theta$ the Laplace transform 
$\ol{\E}[e^{\theta \cQ_{0}}]$ is finite, since the
$\cQ_i$'s have exponential tails.  Using that $\ol{\E}[\cQ_{0}]=0$ we
see that there exist $\theta_0,c_2>0$ such that if $\theta\leq
\theta_0$ then $\ol{\E}[e^{\theta \cQ_{0}}]\leq e^{c_{2}\theta^{2}}$.
Setting $\theta=sT^{-1/2}/(12 c_1 c_2)$ with $c_1$
chosen  large enough so
that $\theta<\theta_0$ we get (using $k\leq K\leq c_1T$ and $s\leq
T^{1/2}$) 
\[
\ol{\PP}\left(\Big|\sum_{i=0}^{k}\cQ_{2i}\Big|\geq
  \tfrac16sT^{1/2}\right)
\leq 2\exp\big(c_1c_2\theta^2 T- \tfrac16\theta sT^{1/2}\big)
\leq 2\exp(-c_{3}s^{2}),
\]
with $c_{3}=(72c_1c_2)^{-1}$. 

For the second term 
on the right in \eqref{3terms} we recall that the
$\D^\ast_{k}$'s
have exponential tails. As $s\in[1,T^{1/2}]$, we 
thus obtain for some $C_{4},c_{4},C_{5},c_{5}>0$
\[
K\ol{\PP}\left(\D^\ast_{0}+\D^\ast_{1}\geq \tfrac12 sT^{1/2}\right)
\leq  K C_{4}e^{-c_{4}sT^{1/2}}\leq C_{5}e^{-c_{5}s^{2}}.
\]
Finally, by standard large deviation considerations for i.i.d. variables we have
\[
\ol{\PP}\left(\sum_{i=0}^{2K}\D^\ast_{i}\leq T\right)
%\leq\bar{\PP}\left(\sum_{i=0}^{2K}1_{\tilde{\Delta}_{i}\geq1}\leq  T\right)
\leq C_{6}e^{-c_{6}K},
\] for some $C_6, c_6 > 0$, provided we choose $c_{1}$ large
enough so that the mean of the sum above is larger than $T$.

This proves the claim for any fixed $\b\in[\b_0,\b_1]$.  The
uniformness over such $\b$ follows since the upper bound can be chosen
as a continuous function of $\b$.
\end{proof}

In order to compare $\cL$ with $\bal$ we need to keep track of how
many times the exploration passes level $0\in S^1$.
To this end we make the following definitions.
Denote by $\cK^X_t(v,\varphi)$ (respectively, $\cK^Y_t(v,\varphi)$)
the number of times $X$ (respectively, $Y$) passes through $0\in S^1$
when started at $(v,\varphi)$ moving in the positive direction
($d_0=+1$) and run for time $t$.
We will write $\cK^X_t$ (respectively, $\cK^Y_t$) when the starting
point is not ambiguous. Recall the definition of
$\cI_{t}^{X}(v,\varphi)$ given right after \eqref{eq:Yhistory}.
\begin{proposition}[Winding and the number of visited vertices]
\label{fact:lapProcessUpperBound}
For any $t>0$ and any $(v,\varphi)\in V_n \times S^1$ we have 
\be\label{laps-ub-eq}
\lapProcess_{t}^{X}(v,\varphi)\leq t+\cI_{t}^{X}(v,\varphi)+1,
\quad\lapProcess_{t}^{Y}(v,\varphi)\leq t+\cI_{t}^{Y}(v,\varphi)+1.
\ee
Moreover, there exists $c=c(\b)>0$ and $t_0=t_0(\b)$ such that
\be\label{laps-lb-eq}
\EE_\b[\cK^Y_t\mid \cS]\geq ct,
\quad \mbox{for all } t\geq t_0,
\ee
where as usual $\cS=\{\tau^Y=+\oo\}$.
\end{proposition}
\begin{proof}
The proofs of \eqref{laps-ub-eq} 
for $X$ and $Y$ are the same. 
We write 
$\lapProcess_{t}$ and $\cI_{t}$, omitting $X$, $Y$, $v$ 
and $\ph$ in order to simplify the notation. 

Define a sequence
of times $\tau_{0}:=0$ and, for $i\geq1$,
\[ 
\tau_{i}:=\inf\left\{ t > \tau_{i-1}:\varphi_t=0\right\} .
\]
We first claim that for $i\geq 1$ we have 
\be\label{I-and-tau}
\mbox{if} \quad \cI_{\tau_{i+1}}-\cI_{\tau_i}=0
\quad\mbox{then}\quad
\tau_{i+1}-\tau_i=1.
\ee
Indeed, at time $\tau_i$ the exploration passed $0\in S^1$ and if it
passes 0 again without traversing a link this means that it has
completed a full lap on one copy of $S^1$.
From \eqref{I-and-tau} we deduce that
\[
(\tau_{i+1}-\tau_i)+(\cI_{\tau_{i+1}}-\cI_{\tau_i})\geq 1,
\mbox{ for all } i\geq1.
\]
Let $k$ be such that $t\in[\tau_{k},\tau_{k+1})$.
Then
\[
\lapProcess_{t}-1=\lapProcess_{\tau_{k}}-1=k-1
\leq\sum_{i=0}^{k-1}\big[(\tau_{i+1}-\tau_{i})+
(\cI_{\tau_{i+1}}-\cI_{\tau_{i}})\big]
=\tau_{k}+\cI_{\tau_{k}}\leq t+\cI_{t},
\]
as claimed.

Now we turn to \eqref{laps-lb-eq}. 
Recall from Proposition \ref{fact:loop_closure}, and the discussion
preceding it, the notation $\D_k$ and $Z^{(k)}$
 as well as the
notion of frontier times.  
Let us use the term \emph{return times} for the jump times of $Z$
which are not frontier times, and \emph{return links} for the
corresponding links traversed by $Y$.
Observe that $\cK^Y_t\geq \cR^Y_t$,
where $\cR^Y_t$ denotes
the number of return links which have been backtracked by $Y$ up to time $t$.
This is because $Y$ does not visit its own history other than by
backtracking, hence between discovering a return link and backtracking
it $Y$ must complete at least one circle. 

Let $R_k$ denote the total number of return times of
$(Z_t)_{\ell_k\leq t\leq \ell_{k+1}}$.  By Proposition
\ref{fact:loop_closure}, conditionally on $\cS$ the sequence 
$\{(\D_k,R_k)\}_{k\geq1}$ is a renewal-reward process, and by the
basic renewal-reward theorem \cite[Theorem 10.5]{gri-sti}
it thus follows that
\[
\frac{\EE[\cK^Y_t\mid \cS]}{t}\geq 
\frac{\EE[\cR^Y_t\mid \cS]}{t}\to 
\frac{\EE[R_1\mid \cS]}{\EE[\D_1\mid \cS]},
\quad\mbox{ as } t\to\oo.
\]
The result \eqref{laps-lb-eq} 
follows from $\EE[R_1\mid \cS]>0$, which is easily
checked.  
For example, it suffices to check that 
$\PP(R_1=1\mid \cS)>0$.   Letting
$\s_1,\s_2,\dotsc$ denote the jump times of $Z$ and using that 
$\PP(R_1=1\mid \cS)=\PP(\ell_1=\s_2\mid\cS)$ we get
\[\begin{split}
\PP(R_1=1\mid \cS)&=\frac{\PP(\s_1<1,\s_2-\s_1\in(1,2-\s_1),
(Y_{\s_2+t}-Y_{\s_2})_{t\geq0}\in\cS)}{\PP(\cS)}\\
&=\PP(\s_1<1,\s_2-\s_1\in(1,2-\s_1))
=e^{-2\b}(e^\b-1-\b)>0.
\end{split}\]
\end{proof}

We now come to the key technical result of the paper, an upper bound
on $\cL^X$ when $X$ explores part of a large cycle.
At the same time we also provide a lower bound on $\cK^X$
since the proof follows a similar structure.

\begin{proposition}[Winding for the exploration process is small]
\label{prop:goodBalanceForExploration}
Let $\beta_{1}>\beta_{0}>1$, and consider the exploration $X$ started
at an arbitrary point $(v,\ph,d)$.
There exist $C_{1},c_{1}>0$ such that for any $\beta\in[\beta_{0},\beta_{1}]$
we have 

\[
\PP_{\beta}\left(|\signExploration_{n^{1/3}}^{X}|
\one_{\{\tau^{X}\geq n^{1/3}\}}\geq 3n^{1/4}\log n\right)
\leq C_{1}e^{-c_{1}\log^{2}n}.\label{eq:windingProcessUpperbound}
\]
Moreover, we have for some $c_{2}>0$
\[
\PP_{\beta}\left(\left\{ \mathcal{K}_{n^{1/3}}^{X}<c_{2}n^{1/3}\right\} \cap\left\{ \tau^{X}\geq n^{1/3}\right\} \right)\leq C_{1}e^{-c_{1}\log^{2}n}.\label{eq:lapProcessLowerBound}
\]
\end{proposition}

Before giving the proof we outline the main ideas.  We want to use the
coupling of the exploration process $X$ to the simple exploration 
process $Y$ from Lemma \ref{fact:coupling}, 
as well as the
concentration result for the latter process, 
Proposition \ref{prop:windingconcentration}.
To get good concentration we will decompose $[0,n^{1/3}]$
into many shorter time intervals $[t_i,t_{i+1})$ of length
approximately $n^{1/6}$ each. On each $[t_i,t_{i+1})$ we will wait
for a `good' coupling with a simple exploration:  
first we wait until the exploration $X$ jumps
to a new vertex so we can start a coupling, then we check if the
simple exploration survives indefinitely, 
which it does with probability $z>0$.  
If so, we can apply Proposition \ref{prop:windingconcentration} in
this interval.
If not, then we repeat the procedure, waiting for a jump of $X$ to a
new vertex and looking at the coupled simple exploration. Typically we
only need to perform this a small number of times until we get a coupling
with a simple exploration which survives.

Let us make these ideas formal and introduce the setup that will be used in the proof. Set $a_{n}:=n^{1/3}$.
We define $t_{i}:=i\cdot b_{n},$ where
$i\in\left\{ 0,1,\ldots,\lfloor n^{1/6}\rfloor\right\} $, 
$b_{n}:=a_{n}/\lfloor n^{1/6}\rfloor$.
Writing $m=\floor{n^{1/6}}-1$,  we decompose 
\[
\signExploration_{a_{n}}^{X}
=\sum_{i=0}^{m}(\signExploration_{t_{i+1}}^{X}-\signExploration_{t_{i}}^{X}).
\]
Fix $i\in\left\{ 0,1,\ldots,m\right\}$.  Let us 
first analyse the change of the winding process on one interval
$[t_{i},t_{i+1})$.  To this end we will define two 
sequences of times
$\left\{ \sigma_{k}^{i}\right\} _{k=0}^{+\infty}$ and 
$\left\{ \tau_{k}^{i}\right\} _{k=0}^{+\infty}$ as well as a sequence
of simple explorations $\{Y^i_k\}_{k=1}^\oo$.
The $\s^i_k$ will form a non-decreasing sequence, taking values in
$[t_i,t_{i+1}]$, 
and will be defined so that, for $k\geq1$ and as long as $\s^i_k<t_{i+1}$,
the process $X$ jumps to a new vertex at time $\s^i_k$.
For such $k$, the process $Y^i_k$ is defined to be
an independent copy of  a simple exploration,  coupled
with $X$ as in Lemma \ref{fact:coupling},
 starting at time $\s^i_k$.
The possibility $\s^i_k=t_{i+1}$ signifies that we have finished 
with the interval $[t_i,t_{i+1})$ and must move on to the
next one.

  We now define the times $\s^i_k$ and $\tau^i_k$.
First we set
$\sigma_{0}^{i}:=t_{i},\tau_{0}^{i}:=0$.  Next,
for $k= 1,2,\ldots $ we set 
\[
\sigma_{k}^{i}:=t_{i+1}\wedge 
\inf\left\{ s\geq\sigma_{k-1}^{i}+\tau_{k-1}^{i}:
X\text{ jumps to a new vertex at time }s\right\}  
\]
where $\tau_{k}^{i}:=\tau^{Y^i_k}$ is the time when $Y^i_k$ terminates
(returns to its starting point).
Note that if $\s^i_{k-1}=t_{i+1}$ then 
$\s^i_k=\s^i_{k+1}=\dotsc=t_{i+1}$.   In particular, this will 
occur if  $\tau^i_{k-1}=\oo$. 
In this case we do not need to define
$Y^i_k,Y^i_{k+1},\dotsc$
Also note that, since the coupling of $Y^i_k$ with $X$ entails
constructing both processes using the same sources of randomness, we
may work with the $\tau^i_k$ as if they are adapted to the filtration
of $X$, even though they are defined in terms of $Y$.

In words, these definitions mean that, firstly,
$Y^i_1$ is a simple exploration coupled with $X$,
started at time $\sigma^i_1$, the first time in $[t_i,t_{i+1})$ that
$X$ jumps to a new vertex.  This coupling is then run either for the
remaining time in $[t_i,t_{i+1})$, or until $Y^i_1$ returns to its starting
point (after time $\tau_{1}^{i}$).  For $k\geq2$,
if the simple exploration
$Y^i_{k-1}$ has returned to its starting point, at time
$\s^i_{k-1}+\tau^i_{k-1}\in [t_i,t_{i+1})$, 
then we wait
until $X$ jumps to a new vertex again.  We call the
time when this occurs $\sigma_k^i$ and we begin a new coupling with a
simple exploration, $Y^i_k$, from the location of $X$ at this time.

Let 
\[
k_0=k_{0}^{i}:=\min\left\{ k\in\mathbb{N}:\tau_{k}^{i}=+\infty
\text{ or }\sigma_{k+1}^{i}= t_{i+1}\right\}.
\]
The first possibility,
$\tau^i_{k_0}=+\oo$, means that at attempt number 
$k_{0}$ the coupled
simple exploration $Y^i_{k_{0}}$ survives (and is the first one with this property).
The other possibility, that $\tau^i_{k_0}<+\oo$ but 
$\sigma_{k_{0}+1}^{i}= t_{i+1}$, means that after time $\s_{k_0}$ the
exploration  $X$ never jumps to a new vertex until the end of the
interval $[t_i,t_{i+1})$.  Included in this possibility is the case
when $X$  closes the loop before jumping again.
Intuitively, $k_0$ is the number of attempts at coupling $X$ with a
simple exploration which survives, until we either succeed or run out
of time.

We now turn to the proof of the proposition.

\begin{proof}[Proof of Proposition
  \ref{prop:goodBalanceForExploration}]
Fix $\b\in[\b_0,\b_1]$.
We first show \eqref{eq:windingProcessUpperbound}
for this $\b$.
Recall the definition of the event $\cA^\his_{t}$ given above \eqref{Ahis-bound}.
First note that it suffices to show that
\be
\PP_{\beta}\left(\cA^\his_{a_n}\cap \big\{|\signExploration_{a_n}^{X}|
\one_{\{\tau^{X}\geq a_n\}}\geq 3n^{1/4}\log n\big\}\right)
\ee
satisfies the claimed bound, due to \eqref{Ahis-bound}.  
Also note that $\cA^\his_{a_n}\se \cA^\his_s$
for $s\leq a_n$.

Consider the interval $[t_i, t_{i+1})$, the stopping times $\s_k^i, \tau_k^i$ and the variable $k_0^i$ from the preceding discussion. Keeping $i$ fixed for now, we will drop it from the superscript on
$\s_k$, $\tau_k$ and $k_0$.
We claim that, under $\PP(\cdot|\mathcal{F}_{t_{i}})$,
the random variable $k_{0}$ is stochastically dominated by a geometric
distribution with parameter $z$, that is to say,
\be\label{k0-tails}
\mbox{for all } k\geq 1, \quad
\PP(k_0\geq k\mid \cF_{t_i})\leq (1-z)^{k-1}.
\ee
Here $z>0$ is the survival probability of a simple exploration, 
see Proposition \ref{fact:loop_closure}. 
The claim is easily established by induction, using 
\[
\PP(k_0\geq k+1\mid \cF_{t_i})=
\PP(k_0\geq k+1\mid \cF_{t_i},\{k_0\geq k\})
 \PP(k_0\geq k\mid \cF_{t_i})
\]
and
\[
\PP(k_0\geq k+1\mid \cF_{t_i},\{k_0\geq k\})
\leq \PP(\tau_k<\oo\mid \cF_{t_i},\{k_0\geq k\})=1-z.
\]
 
We now show that there exist constants $C_{2},c_{2}>0$,
uniform in $n$ and in $i$, such that for all $t>0$,
on the event $\cA^\his_{a_n}\cap\{\tau^X\geq a_n\}$ we have
\begin{equation}\label{eq:tmp17}
\PP(\sigma_{k_{0}} - t_{i} \geq t
\mid \mathcal{F}_{t_{i}}) \leq C_{2}e^{-c_{2}t}.
\end{equation}
First we establish that there are $C_1,c_1>0$ such that
for any $k\geq 0$ and any $t>0$, on $\cA^\his_{a_n}\cap\{\tau^X\geq a_n\}$ we have
\begin{equation}\label{eq:sigma-intemediate}
\PP\left((\s_{k+1}-\s_{k})\one_{\{k_0>k\}}\geq t \mid \mathcal{F}_{\sigma_{k}}\right) 
=\PP(k_0>k,\s_{k+1}-\s_{k}\geq t \mid \mathcal{F}_{\sigma_{k}}) 
\leq C_{1}e^{-c_{1}t}.
\end{equation}
Indeed, for any $k < k_0$ we have $\tau_{k}< \infty$, so 
\[
\begin{aligned}
& \PP \left(k_0>k, \s_{k+1} - \s_{k}\geq t\mid \mathcal{F}_{\s_{k}}\right)  \\
& \leq \PP\left( \{ \s_{k+1} - \s_{k}\geq t \} \cap 
\{ \tau_{k} \leq t/2 \} \mid \mathcal{F}_{\s_{k}}\right) + 
\PP\left( \tau_{k}> t/2, \, \tau_{k} < +\infty \mid 
\mathcal{F}_{\s_{k}}\right).
\end{aligned}
\]
The second term is at most $C e^{-ct}$, for some $C,c > 0$, by
\eqref{eq:tauYtails} from Proposition \ref{fact:loop_closure}. To
estimate the first term, note that $\tau_{k} \leq t/2$
together with $\sigma_{k+1} - \sigma_{k}\geq t$
implies that $\s_{k+1} - (\s_{k} +\tau_{k})\geq t/2$, in particular $X$ does not visit previously unexplored vertices for time at least $t/2$ after $\s_{k} +\tau_{k}$. Thus
\begin{multline}\label{eq:needCor16}
\PP\left( \{ \s_{k+1} - \s_{k}\geq t \} \cap 
\{ \tau_{k} \leq t/2 \} \mid \mathcal{F}_{\s_{k}}\right)
\one_{\cA^\his_{a_n}} \one_{\{\tau^X\geq a_n\}}
\\
\leq \EE\big[
\PP\big( \{ \s_{k+1}\wedge\tau^X - (\s_{k}+\tau_k) \geq t/2 \} \cap 
\cA^\his_{\s_k+\tau_k} \mid \mathcal{F}_{\s_{k}+\tau_k}\big)\mid \cF_{\s_k}\big].
\end{multline}
By Corollary \ref{cor:visits-intensity} 
the probability is at most $e^{-c't}$ for some 
$c' > 0$,
which together with the previous estimate proves
\eqref{eq:sigma-intemediate}.

Now note that
\[
\s_{k_{0}} - t_{i} = \s_{k_0}-\s_0=
\sum_{j=0}^{k_0 -1}  \left( \s_{j+1} - \s_{j} \right),
\]
where by \eqref{eq:sigma-intemediate} each summand, conditionally on
all previous terms, has exponential tails.
Since $k_0$, the number of summands, is by \eqref{k0-tails} itself
dominated by a geometric random variable, one may conclude that the
sum itself has exponential tails,
as claimed in \eqref{eq:tmp17}.  In more detail,
we have for any $k>0$ that
\[
\PP(\s_{k_0}-\s_0\geq t\mid \cF_{t_i})\leq
\PP\Big(\sum_{j=0}^k (\s_{j+1}-\s_j)\one_{\{k_0>j\}}
\geq t \,\Big|\, \cF_{t_i}\Big)+\PP(k_0>k\mid \cF_{t_i}).
\]
Now for any $\theta>0$ we have
\[
\PP\Big(\sum_{j=0}^k (\s_{j+1}-\s_j)\one_{\{k_0>j\}}
\geq t \,\Big|\, \cF_{t_i}\Big)
\leq e^{-\theta t} \EE\Big[\exp\Big(\theta
\sum_{j=0}^k (\s_{j+1}-\s_j)\one_{\{k_0>j\}}\Big)
\,\Big|\, \cF_{t_i} \Big],
\]
where (using that $\{k_0>k-1\}\in\cF_{\s_k}$)
\begin{multline*}
\EE\Big[\exp\Big(\theta
\sum_{j=0}^k (\s_{j+1}-\s_j)\one_{\{k_0>j\}}\Big)
\,\Big|\, \cF_{t_i} \Big]\\=
\EE\Big[\exp\Big(\theta
\sum_{j=0}^{k-1} (\s_{j+1}-\s_j)\one_{\{k_0>j\}}\Big)
\EE\big[e^{\theta  (\s_{k+1}-\s_k)\one_{\{k_0>k\}}}
\mid \cF_{\s_k} \big]
\,\Big|\, \cF_{t_i} \Big].
\end{multline*}
Here the inner factor may be written as
\[
\EE\big[e^{\theta  (\s_{k+1}-\s_k)\one_{\{k_0>k\}}}
\mid \cF_{\s_k} \big]
=\int_0^\oo \PP\big(
e^{\theta  (\s_{k+1}-\s_k)\one_{\{k_0>k\}}}>s
\mid \cF_{\s_k}\big) ds.
\]
Using \eqref{eq:sigma-intemediate} we conclude that 
we may choose $\theta>0$ (depending on constants $c_1$, $C_1$ in \eqref{eq:sigma-intemediate}) such that, 
on $\cA^\his_{a_n}\cap {\{\tau^X\geq a_n\}}$,
\[
\EE\big[e^{\theta  (\s_{k+1}-\s_k)\one_{\{k_0>k\}}}
\mid \cF_{\s_k} \big]\leq e,
\]
say,  for all $k\geq 0$.  It follows by induction that
\[
\EE\Big[\exp\Big(\theta
\sum_{j=0}^k (\s_{j+1}-\s_j)\one_{\{k_0>j\}}\Big)
\,\Big|\, \cF_{t_i} \Big]
\one_{\cA^\his_{a_n}} \one_{\{\tau^X\geq a_n\}}
\leq e^k,
\]
and hence
\[
\PP(\s_{k_0}-\s_0\geq t\mid \cF_{t_i})
\one_{\cA^\his_{a_n}} \one_{\{\tau^X\geq a_n\}}
\leq
e^{k-\theta t}+\PP(k_0>k\mid \cF_{t_i}).
\]
Setting $k=\floor{\tfrac\theta2 t}$ and using \eqref{k0-tails},
this gives \eqref{eq:tmp17}.

Recall the notion of a failed coupling from Lemma
\ref{fact:coupling}. The bound \eqref{eq:tmp17} tells us that typically we don't wait too long for a coupling with a simple exploration process that survives. If the coupling doesn't fail, we will be able to transfer estimates of the winding process from the simple exploration to the process $X$.

To this end we distinguish three possible scenarios of what can happen during a given time interval. We say that the interval $[t_{i},t_{i+1})$ is
\emph{good}, denoting this event by $\mathcal{G}_{i}$, if the
following hold:
\begin{itemize}[leftmargin=*]
\item $\tau^i_{k_{0}}=+\oo$, and
\item none of the $k_0$ attempted couplings failed until time
$T=t_{i+1}-t_{i}\leq n^{1/6}$.
\end{itemize}
On the event $\mathcal{G}_{i}$ the coupling started at time
$\s^i_{k_0}$ survives and it lasts until time $t_{i+1}$, in particular
$X$ cannot close its loop before time $t_{i+1}$ (as this would entail
$X$ returning to some vertex visited before time $\s^i_{k_0}$ and
hence $Y^{k_{0}}$ returning to its starting point, i.e.\ $\tau^{Y^{k_0}}<\oo$).
Thus $\cG_i\se \{ \tau^{X}\geq t_{i+1}\}$. 
Next, we say that the interval $[t_i,t_{i+1})$ is \emph{terminal} 
if $\tau^i_{k_{0}}<+\oo$ and
$\sigma_{k_{0}+1}^{i}=t_{i+1}$, and we denote this
event by $\mathcal{T}_{i}$.  Note that 
$\{ \tau^{X}< t_{i+1}\} \subseteq \cT_i$.
Finally we let $\cB_i=(\cG_i\cup\cT_i)^c$, and if this event occurs we
say that the interval $[t_i,t_{i+1})$ 
is \emph{bad}.
On this event one of the attempted couplings failed.

Let us now estimate the winding process on each of the above events.  We have
\begin{align}\label{eq:winding-process-GBT}
|\signExploration_{a_{n}}^{X}|\one_{\{\tau^{X}\geq a_n\}} \nonumber  
&\leq\sum_{i=0}^{m}
|\signExploration_{t_{i+1}}^{X}-\signExploration_{t_{i}}^{X}|\one_{\{\tau^{X}\geq a_n\}} \\
 & =\sum_{i=0}^m |\signExploration_{t_{i+1}}^{X}-\signExploration_{t_{i}}^{X}|
\big(\one_{\mathcal{G}_{i}}+\one_{\mathcal{B}_{i}}+\one_{\mathcal{T}_{i}}\big)
\one_{\{\tau^{X}\geq a_n\}} \nonumber \\
 & \leq\lfloor n^{1/6}\rfloor\max_{0\leq i\leq m }
|\signExploration_{t_{i+1}}^{X}-\signExploration_{t_{i}}^{X}|\one_{\mathcal{G}_{i}}
+b_{n}\sum_{i=0}^{m}\one_{\mathcal{B}_{i}}
+b_n\sum_{i=0}^{m}\one_{\mathcal{T}_{i}}\one_{\{\tau^{X}\geq a_n\}},
\end{align}

where for the second and third term we used the trivial estimate that \\ $|\signExploration_{t_{i+1}}^{X}-\signExploration_{t_{i}}^{X}| \leq t_{i+1} - t_i = b_n$.

We will now estimate each of the three terms in \eqref{eq:winding-process-GBT} separately. Let us start with the first one. We will estimate $|\signExploration_{t_{i+1}}^{X}-\signExploration_{t_{i}}^{X}|$ on the
event $\mathcal{G}_{i}\cap \cA^\his_{a_n}$. 
Since $\signExploration^{X}$ increases at rate at most $1$, we have the estimate
\be\label{eq:tmp18}
|\signExploration_{t_{i+1}}^{X}-\signExploration_{t_{i}}^{X}|
\leq\sigma_{k_{0}}^{i}-t_{i}+
|\signExploration_{t_{i+1}}^{X}
-\signExploration^X_{\sigma_{k_{0}}^{i}}|
\one_{\{\sigma_{k_{0}}^{i}<t_{i+1}\}}.
\ee
By \eqref{eq:tmp17}, on $\cA^\his_{a_n}$ the first term on the right hand
side in \eqref{eq:tmp18} is at most $\log^2 n$ with probability at
least $1 - C_2 e^{-c_2 \log^2 n}$. 
In the second term we may, on $\cG_i$, replace 
$\signExploration_{t_{i+1}}^{X}-\signExploration^X_{\sigma_{k_{0}}^{i}}$
by $\signExploration^Y_{t_{i+1}-\sigma_{k_{0}}^{i}}$,
where $Y=Y^i_{k_0}$ is the simple exploration started at time 
$\s^i_{k_0}$.
Now we apply Proposition \ref{prop:windingconcentration} with 
$s = \frac{1}{2}\log n$ and $T = b_n$. As
$t_{i+1}-\sigma_{k_{0}}^{i}\leq t_{i+1}-t_i= b_n=n^{1/3}/\floor{n^{1/6}}$, we obtain in particular
\[
\PP\big(\cG_i\cap\big\{|\signExploration^Y_{t_{i+1}-\sigma_{k_{0}}^{i}}|
\geq \tfrac12 b_n^{1/2} \log n\big\}\mid \cF_{\s_{k_0}}
\big)\leq C e^{-c\log^2 n}.
\]
Thus 
\[
\PP\big(
\cA^\his_{a_n}\cap 
\big\{|\signExploration_{t_{i+1}}^{X}-\signExploration_{t_{i}}^{X}|
\one_{\mathcal{G}_{i}}\geq
b_{n}^{1/2}\log n\big\}\big)
\leq C_{3}e^{-c_{3}\log^{2}n},
\]
for some $C_3, c_3 > 0$. 
Furthermore, applying a union bound we obtain that (recalling $m=\floor{n^{1/6}}-1$) 
\[
\PP\Big(\cA^\his_{a_n}\cap 
\Big\{\max_{0\leq i\leq m}
|\signExploration_{t_{i+1}}^{X}-\signExploration_{t_{i}}^{X}|
\one_{\mathcal{G}_{i}}\geq
b_{n}^{1/2}\log n\Big\}\Big) \leq
\floor{n^{1/6}}
C_{3}e^{-c_{3}\log^{2}n}.\label{eq:windingProcessIncreament}
\]

We now move to the second term of \eqref{eq:winding-process-GBT}. 
We need to estimate $\PP(\mathcal{B}_{i}|\mathcal{F}_{t_{i}})$. To
this end notice that by Lemma \ref{fact:coupling} the probability
for any given coupling to fail is bounded above by
\[
4\b b_n (\cJ^X_{a_n}+\b b_n)/n\leq 
4\b a_n (\cJ^X_{a_n}+\b a_n)/n.
\]
Defining
$\cD_0:=\left\{ \mathcal{J}_{a_{n}}^{X}\leq
4\beta  a_{n}\right\}$,
we have, for any $k>0$, that
\[\begin{split}
\PP(\cB_i\mid \cF_{t_i})&\leq
\PP(k_0\geq k\mid \cF_{t_i})+
\PP(\{k_0<k\}\cap\cB_i\cap\cD_0 \mid \cF_{t_i})+
\PP(\cD_0^c\mid \cF_{t_i}) \\
&\leq (1-z)^k+
k\frac{20\b^2 a_n^2}{n} +
\PP(\cD_0^c\mid \cF_{t_i}).
\end{split}\]
Recalling that $a_n=n^{1/3}$ and choosing $k=\floor{n^{1/12}}$
it follows that, for some $C>0$,
\begin{equation}\label{eq:estimateBadInterval}
\PP(\cB_i\mid\cF_{t_i})\leq p_{n,i}:= Cn^{-1/4}+
\PP(\cD_0^c\mid \cF_{t_i}).
\end{equation}
We claim that, on the event 
$\cD_1:=\{\mathcal{J}_{a_{n}}^{X}\leq 2\beta  a_{n}\}$,
we have for large enough $n$ that $p_{n,i}\leq 2Cn^{-1/4}$
with $C$ from \eqref{eq:estimateBadInterval}.
Indeed,
\[\begin{split}
\PP(\cD_0^c\mid \cF_{t_i})&\leq 
\PP(\cJ^X_{t_i}>2\b a_n\mid \cF_{t_i})+
\PP(\cJ^X_{a_n}-\cJ^X_{t_i}>2\b a_n\mid \cF_{t_i}) \\
&=\one\{\cJ^X_{t_i}>2\b a_n\}+
\PP(\cJ^X_{a_n}-\cJ^X_{t_i}>2\b a_n\mid \cF_{t_i}).
\end{split}
\]
On the right-hand-side, the indicator vanishes on $\cD_1$,
and the probability is at most $C_1 e^{-c_1 a_n}$ for some 
$C_1,c_1>0$, since $\cJ^X_t$ is a
counting process with intensity bounded above by $\b$, see Lemma 
\ref{le:intensity} and the argument for  \eqref{Ahis-bound}.  
The claim follows.

Thus, employing \eqref{eq:estimateBadInterval}, we obtain
for large enough $n$ that
\begin{align*}
\PP&\left(\textstyle\sum_{i=0}^{m}\one_{\mathcal{B}_{i}}\geq
  n^{1/12}\log n\right) 
 \leq\PP\left(\Big\{\textstyle\sum_{i=0}^{m}
\one_{\mathcal{B}_{i}}\geq n^{1/12}\log n \Big\} \cap \cD_1\right)
+\PP(\mathcal{D}_1^{c})
\\
 &\qquad \leq\PP\left(\Big\{\textstyle\sum_{i=0}^{m}
(\one_{\mathcal{B}_{i}}-\PP(\mathcal{B}_{i}|\mathcal{F}_{t_{i}}))
\geq n^{1/12}\log n-\sum_{i=0}^m p_{n,i}\Big\}\cap\cD_1\right)
+\PP(\mathcal{D}_1^{c})
\\
 & \qquad\leq\PP\left(\textstyle\sum_{i=0}^m
(\one_{\mathcal{B}_{i}}-\PP(\mathcal{B}_{i}|\mathcal{F}_{t_{i}}))
\geq \frac{1}{2}n^{1/12}\log n \right)+\PP(\mathcal{D}_1^{c}).
\end{align*} 
The sum inside the first
probability is a martingale, with increments
bounded by $1$.  Thus
by the Azuma inequality (see e.g. \cite[Theorem A.10]{Levin-Peres})
we get
\begin{equation}\label{eq:BernoulliLDP}
\PP\left(\textstyle\sum_{i=0}^m
(\one_{\mathcal{B}_{i}}-\PP(\mathcal{B}_{i}|\mathcal{F}_{t_{i}}))
\geq \frac{1}{2}n^{1/12}\log n \right) \leq 2\exp\left(-\frac{n^{1/6}\log^{2}n}{8\floor{n^{1/6}}}\right).
\end{equation}
As before we have that $\PP(\cD_1^c)\leq e^{-c n}$ for some $c>0$.
Taken together, these facts give
\be\label{sumBi}
\PP\left(\textstyle\sum_{i=0}^m
\one_{\mathcal{B}_{i}}
\geq n^{1/12}\log n\right)
\leq C_{4}e^{-c_{4}\log^{2}n}.
\ee
for some $C_{4,}c_{4}>0$.

Finally, let us now consider
$\sum_{i=0}^m \one_{\mathcal{T}_{i}}\one_{\{\tau^{X}\geq a_n\}}$.
Observe that for $i\leq m-1$ 
the event $\mathcal{T}_{i}\cap\left\{ \tau^{X}\geq a_n\right\}$
requires that the exploration neither jumps to an unvisited vertex nor
closes the loop for a time period of at least
$n^{1/6}$.  By a similar application of Corollary \ref{cor:visits-intensity} as for \eqref{eq:needCor16}  
the latter event has
probability smaller than $C_{4}e^{-c_{4}n^{1/6}}$, for some $C_{4},c_{4}>0$. 
We thus have
\[
\begin{aligned}
\PP\left(\textstyle\sum_{i=0}^m
\one_{\mathcal{T}_{i}}\one_{\{\tau^{X}  \geq a_n\}}\geq2\right)
&\leq\PP\left(\textstyle\sum_{i=0}^{m-1}
   \one_{\mathcal{T}_{i}}\one_{\{\tau^{X} \geq a_n\}}\geq1\right)
    \\
&\leq\sum_{i=0}^{m-1}
   \PP\big(\mathcal{T}_{i}\cap\big\{ \tau^{X}\geq a_n\big\} \big)
\leq C_{5}e^{-c_{5}n^{1/6}},\label{eq:probabilityOfTerribleEvent}
\end{aligned}
\]
for some $C_{5},c_{5}>0$. 

From \eqref{eq:winding-process-GBT}, \eqref{eq:windingProcessIncreament}, \eqref{sumBi}
and \eqref{eq:probabilityOfTerribleEvent} we conclude that for some
$C_{6},c_{6}>0$ 
\[
\PP\big(
\cA^\his_{a_n}\cap 
\big\{|\signExploration_{a_{n}}^{X}|\one_{\{\tau^{X}\geq a_n\}}  \geq
\floor{n^{1/6}} b_{n}^{1/2}\log n+b_{n}n^{1/12}\log n+b_n
\big\} \big) \leq 
C_{6}e^{-c_{6}\log^{2}n}.
\]
Since 
$\floor{n^{1/6}} b_{n}^{1/2}\log n+b_{n}n^{1/12}\log n+b_n
\leq 3n^{1/4}\log n$, 
this concludes the proof of \eqref{eq:windingProcessUpperbound}
for a fixed $\b\in[\b_0,\b_1]$.  The
uniformness over such $\b$ follows since the upper bound can be chosen
as a continuous function of $\b$.

Now we turn to \eqref{eq:lapProcessLowerBound}. 
We aim to do a similar decomposition as above,
and as before it suffices to work on the event $\cA^\his_{a_n}$.
For $i\in\left\{ 0,1,\ldots,m\right\}$,
let $Y_i$ be the coupled simple exploration started at
time $\s^i_{k_0}$,
and let $\cS_i$ be the event that $Y_i$ survives.  
On the event $\mathcal{G}_{i}$ we have 
in particular that $\cS_i$ occurs, and we can use
$\cK_{t_{i+1}}^{X}-\cK_{t_{i}}^{X}\geq \cK_{t_{i+1}-\sigma_{k_{0}}^{i}}^{Y_i}$.
On $\cG_{i}^c$ we simply 
use $\cK_{t_{i+1}}^{X}-\cK_{t_{i}}^{X}\geq0$.
Therefore
\[
\cK_{a_{n}}^{X} =
\sum_{i=0}^{m}(\cK_{t_{i+1}}^{X}-\cK_{t_{i}}^{X})
\geq\sum_{i=0}^{m}(\cK_{t_{i+1}}^{X}-\cK_{t_{i}}^{X})
\one_{\cG_{i}}
 \geq\sum_{i=0}^{m}\cK_{t_{i+1}-\s_{k_{0}}^{i}}^{Y_i}
\one_{\cG_{i}}.
\]
Hence 
\[\begin{split}
&\PP(\cA^\his_{a_n}\cap\{\cK_{a_{n}}^{X}< c_2 a_n\}\cap \{\tau^X\geq a_n\})
\\ &\quad\leq 
\PP\Big(\cA^\his_{a_n}\cap\Big\{\sum_{i=0}^{m}
\cK_{t_{i+1}-\s_{k_{0}}^{i}}^{Y_i}
\one_{\cG_{i}}<c_2 a_n\Big\} \cap \{\tau^X\geq a_n\}\Big)\\
&\quad \leq
\PP\Big(\sum_{i=0}^{\floor{n^{1/6}}-1}\cK_{b_n-\log^2n}^{Y_i}
  \one_{\cG_{i}}<c_2 a_n \Big)\\
&\quad\quad+ \PP(\{\exists i\leq m: 
    \s^i_{k_0}-t_i>\log^2 n\}\cap \cA^\his_{a_n}\cap\{\tau^X\geq a_n\}).
\end{split}\]
Using \eqref{eq:tmp17}
we get for some $C_{8},c_{8}>0$
\[
\PP(\{\exists i\leq m: 
    \s^i_{k_0}-t_i>\log^2 n\}\cap \cA^\his_{a_n}\cap\{\tau^X\geq a_n\})
\leq C_8 e^{-c_8\log^2 n}.
\]
To bound the first probability, we use that
\[
\sum_{i=0}^{m}\cK_{b_n-\log^2n}^{Y_i}
  \one_{\cG_{i}}\geq
\sum_{i=0}^{m}\cK_{b_n-\log^2n}^{Y_i}
  \one_{\cS_i}
-\sum_{i=0}^{m}\cK_{b_n-\log^2n}^{Y_i}
  (\one_{\cB_{i}}+  \one_{\cT_{i}})
\]
The processes $Y_i$ are i.i.d.\ and by \eqref{laps-lb-eq} from Proposition \ref{fact:lapProcessUpperBound} we get 
\[
\EE[\cK_{b_n-\log^2n}^{Y_i}  \one_{\cS_i}] 
\geq c z (b_n-\log^2 n).
\]
Hence by standard large deviations
estimates we get for some $C_7,c_7>0$
\[
\PP\Big(\sum_{i=0}^m \cK_{b_n-\log^2n}^{Y_i}
   \one_{\cS_i}
 <c_2 a_n \Big)\leq C_7 e^{-c_7\log^2 n}
\]
provided we pick $c_2$ small enough.   

It remains to bound the contributions involving $\cB_i$ and $\cT_i$.
Recall from Proposition \ref{fact:lapProcessUpperBound} 
that $\cK^{Y_i}_t\leq t+1+\cI^{Y_i}_t\leq t+1+2\cJ ^{Y_i}_t$,
and from the observations preceding \eqref{Ahis-bound} that the number of `jumps'
$\cJ ^{Y_i}_t$ is dominated by a Poisson process with rate 
$\b\tfrac n{n-1}$.  It follows that,
with probability at least
$1-C_{8}e^{-c_{8}\log^{2}n}$, we have that
$\cK_{b_n-\log^2n}^{Y_i}\leq 5\b b_n$ for all $i\leq m$.
Consequently, using also \eqref{sumBi}
\[\begin{split} 
\PP\Big(\sum_{i=0}^m
\cK_{b_n-\log^2n}^{Y_i}\one_{\cB_{i}}\geq 
5\b n^{1/4}\log n\Big)
&\leq\PP\Big(\sum_{i=0}^m
\one_{\mathcal{B}_{i}}\geq n^{1/12}\log n\Big)
+C_{8}e^{-c_{8}\log^{2}n}\\
&\leq C_{9}e^{-c_{9}\log^{2}n}+C_{8}e^{-c_{8}\log^{2}n}.
\end{split}\]
Finally, for the terms involving $\cT_i$
we again use that 
$\cK_{b_n-\log^2n}^{Y_i}\leq 5\b b_n$ for all 
$i\leq m$, with probability at least $1-C_{8}e^{-c_{8}\log^{2}n}$,
combined with \eqref{eq:probabilityOfTerribleEvent} to get
\[
\begin{split}
\PP\Big(\sum_{i=0}^m
\cK_{b_n-\log^2n}^{Y_i} \one_{\cT_i}
\one_{\{\tau^{X}\geq a_n\}} \geq 10\b n^{1/6}\Big)
&\leq\PP\Big(\sum_{i=0}^m
\one_{\cT_i}\one_{\{\tau^{X}\geq a_n\}} \geq2\Big)
+C_{8}e^{-c_{8}\log^{2}n}\\
&\leq C_{10}e^{-c_{10}\log^{2}n} + C_{8}e^{-c_{8}\log^{2}n}.
\end{split}\]
This establishes \eqref{eq:lapProcessLowerBound}.
\end{proof}

\subsection{Proofs of Propositions 
\ref{prop:preBalanceFact} and
\ref{prop:balanceFact}}

Now we turn to the proofs of the main results of 
Section \ref{sec:balance}, namely
Propositions \ref{prop:preBalanceFact} and
\ref{prop:balanceFact}, which concern the balance 
$\bal(v,\floor{n^{1/2}})$ in cycles of length at least 
$\floor{n^{1/2}}$.
We start with the following corollary of 
Proposition \ref{prop:goodBalanceForExploration},
which states that the bounds in that proposition
hold uniformly over all possible starting points 
$(v,\varphi)\in V_n\times S^1$ for the exploration process $X$.  
We use the notation $\signExploration_{t}^{X}(v,\varphi)$,
$\mathcal{K}_{t}^{X}(v,\varphi)$ and
$\tau^{X}(v,\varphi)$ when $X$ starts  at $(v,\varphi)$.
\begin{corollary}
\label{cor:goodBalanceForEverybody} Let $\beta_{1}>\beta_{0}>1$.
There exist $C,c>0$ such that for all $\beta\in[\beta_{0},\beta_{1}]$
we have 
\begin{equation}\label{eq:CorwindingProcessLowerbound}
\PP_{\beta}\left(\exists_{(v,\varphi)\in V_n\times S^1}:
|\signExploration_{n^{1/3}}^{X}(v,\varphi)|
\one_{\{\tau^{X}(v,\varphi) \geq  n^{1/3}\}}
\geq 3n^{1/4}\log^{2}n\right)\leq Ce^{-c\log^{2}n},
\end{equation}
and, for some $c_{2}>0$,
\begin{equation}\label{eq:CorlapProcessLowerbound}
\PP_{\beta}\left(\exists_{(v,\varphi)\in V_n\times S^1}:
\left\{ \cK_{n^{1/3}}^{X}(v,\varphi)<c_{2}n^{1/3}\right\} 
\cap\left\{ \tau^{X}(v,\varphi)
\geq n^{1/3}\right\} \right)
\leq Ce^{-c\log^{2}n}.
\end{equation}
\end{corollary}

In the proof we will use the following notation.
For a measurable subset $A\se E\times S^1$
and $\omega\in\Omega$ we denote the restriction 
\[
\omega_{A}:=\left\{ (e,\varphi,m)\in\omega:
(e,\varphi)\in A\right\}.
\]
Also recall that we will often identify $S^{1}$ with the interval $[0,1)$.

\begin{proof}
We give details for \eqref{eq:CorwindingProcessLowerbound},
the argument for \eqref{eq:CorlapProcessLowerbound}
is very similar.  Write
\[
\cB(v,\varphi)=\big\{
|\signExploration_{n^{1/3}}^{X}(v,\varphi)|
\one_{\{\tau^{X}(v,\varphi) \geq  n^{1/3}\}}
\geq 3n^{1/4}\log^{2}n
\big\}.
\]
We fix $\eps>0$ to be a small enough positive constant
(to be specific, $\eps$ needs to be smaller than the 
constant $c_1$ in the
exponent on the right-hand-side of 
\eqref{eq:windingProcessUpperbound}).
Let $m=\floor{e^{\eps \log^2 n}}$ and 
define the growing sequence
of sets $A_{i}:=E\times[0,\beta_{0}/\beta_{1}+i\d]$, where
$\d=\d_{n}:=\tfrac1m\big(1-\tfrac{\beta_{0}}{\beta_{1}}\big)$
and $i\in\{ 0,1,\ldots,m\}$.  
We will consider the sequence $\om_{A_i}$ which we think of as
revealing the configuration $\om$ in increments of size $\d$.
Consider the event
\[ 
\mathcal{D}:=\left\{ \om:
|\om_{A_{i}\sm A_{i-1}}|\leq1
\text{ for each }i\in\{ 1,\ldots,m\} \right\}
\] 
that each step in the sequence reveals at most one more link. 
Since $|\om_{A_{i}\sm A_{i-1}}|$ is Poisson distributed with mean
$\binom{n}{2}\tfrac\b{n-1}\d_n$  we have
for some $\tilde C,C_{2},c_{2}>0$ that
\begin{equation} \label{eq:Dcomp}
\PP(\mathcal{D}^{c})\leq\sum_{i=1}^m
\PP\left(|\om_{A_{i}\sm A_{i-1}}|\geq2\right)
\leq C_{2}e^{-c_{2}\log^{2}n}.
\end{equation}
Thus it suffices to show that 
$\PP(\cup_{(v,\varphi)}\cB(v,\varphi)\cap \cD)$ satisfies the bound
\eqref{eq:CorwindingProcessLowerbound}.

Now on $\cD$, to determine if there is some 
$(v,\varphi)\in V_n\times S^1$ for which $\cB(v,\varphi)$ holds
it suffices to consider $\varphi$ of the form
$\varphi_i=\beta_{0}/\beta_{1}+i\d$ for $0\leq i\leq m$.
Indeed, if $\varphi$ is arbitrary, let $i$ be such that
$\varphi_{i-1}\leq \varphi\leq \varphi_i$.  Then (on $\cD$)
the exploration started at $(v,\varphi)$ agrees either with that
started at $(v,\varphi_{i-1})$ or that started at $(v,\varphi_i)$
(up to a small time-shift of size at most $\d$ which we will
ignore).  Hence, using Proposition 
\ref{prop:goodBalanceForExploration},
\[
\begin{split}
\PP(\cup_{(v,\varphi)}\cB(v,\varphi)\cap \cD)
\leq \sum_{v\in V_n}\sum_{i=0}^m
\PP(\cB(v,\varphi_i))
\leq n e^{\eps\log^2 n} C_{1}e^{-c_{1}\log^{2}n}.
\end{split}
\]
For $\eps>0$ small enough, this satisfies the claimed bound.
\end{proof}

To proceed we will need some notations and observations
which will allow us to
relate the winding process, $\cL$, to the balance of cycles, 
$\bal$.  
Let $X_s=X_s(v_0,\varphi_0,d_0)$ denote the
exploration started at $(v_0,\varphi_0)$ in the direction 
$d_0\in\{-1,+1\}$, viewed at time $s$.
Let us write $X_s=(v_s,\varphi_s,d_s)$
and define
\[
\bal^X_t(v_0,\varphi_0,d_0)=\sum_{0\leq s\leq t}
d_s \one_{\{\varphi_s=0\}}.
\]
(Although formally the summation is over an uncountable set, 
almost surely there is only a finite number of nonzero terms.)
In words, $\bal^X_s$ totals the number of visits of $X$ to level
$\varphi=0$, counted with the sign given by the direction of travel. 
Note that our previously defined balance-quantity may be written as 
\[
\bal(v,k)=\bal^X_{\tau_k}(v,0,+1)
\]
where $\tau_k$ is the first time at which $X$ has made $k$ visits
to level $\varphi=0$.  

It is easy to see the following:
for any starting point $(v_0,\varphi_0,d_0)$ and any $t\geq0$ we have
that 
\begin{equation}\label{error3}
\big| |\bal^X_t(v_0,\varphi_0,d_0)| - 
|\cL^X_{t}(v_0,\varphi_0,d_0)| \big|\leq 3.
\end{equation}
Indeed, for $t=0$ the two terms are either 1 and 0 (if $\varphi_0=0$)
or 0 and 0 (if $\varphi_0\neq0$).  
As $t$ increases, $|\bal^X_t|$ stays constant until $X$ passes level
$\varphi=0$, at which time it changes by 1.  Until this time $|\cL^X_t|$
can change by at most 1, since if it changes more then this
necessarily means that $X$ passes level $\varphi=0$;
hence the difference in \eqref{error3} is 
certainly bounded by 2.
Between successive visits to $\varphi=0$ it remains bounded by 2 for
the same reason.  Finally, after the last visit to $\varphi=0$ we may
have that $\cL^X_t$ changes by up to 1 while $|\bal^X_t|$
remains constant.  Thus the difference is at most 3.

\begin{proof}[Proof of Proposition \ref{prop:preBalanceFact}]
Let 
\[
\cA_{n}:=\bigcap_{(v,\varphi)\in V_n\times S^1} \Big\{ 
|\signExploration_{n^{1/3}}^{X}(v,\varphi)|
\one_{\{\tau^{X}(v,\varphi)\geq n^{1/3}\}}
\leq 3n^{1/4}\log^{2}n\Big\} 
\]
and 
\[
\cB_{n}:=
\bigcap_{(v,\varphi)\in V_n\times S^1} 
\Big(\Big\{ \mathcal{K}_{n^{1/3}}^{X}(v,\varphi)
\geq c_{2}n^{1/3}\Big\} 
\cup
\Big\{ \tau^{X}(v,\varphi)\leq n^{1/3} \Big\} \Big),
\]
where $c_2$ is as in \eqref{eq:lapProcessLowerBound}. 
(To see that $\cA_n$ and $\cB_n$ are measurable, note that one gets
the same events if $\varphi$ is restricted to rationals.)
By Corollary \ref{cor:goodBalanceForEverybody} 
we have $\PP(\cA_{n}\cap\cB_{n})\geq1-Ce^{-c\log^{2}n}$
for some $C,c>0$, so it suffices to consider
$\om\in\cA_{n}\cap\cB_{n}$. 

Suppose $v$ is such that $|\cC_\om(v)|\geq\floor{n^{1/2}}$
(otherwise there is nothing to prove).
For $i\geq 1$, let $t_{i}=i n^{1/3}$ and let
$i_{0}:=\min\{ i\geq 1: \cK_{t_{i}}^{X}(v,0)\geq n^{1/2}\}$.  
Thus by time $t_{i_0}$ the exploration (started at $(v,0)$)
has visited the first 
$\floor{n^{1/2}}$ vertices in $\cC_\om(v)$ following $v$.
Since $\om\in\cB_{n}$, the contributions to $\cK^X_t$ between
successive $t_i$ are all at least $c_2n^{1/3}$;  using the additivity
of $\cK^X_t$ we conclude that
 $i_{0}\leq c_{2}^{-1}n^{1/6}$.

Let us write $X_{t_i}=(v_i,\varphi_i,d_i)$.
Note that (using \eqref{error3})
\[\begin{split}
|\bal(v,\floor{n^{1/2}})|&\leq \sum_{i=1}^{i_0-1}
|\bal^X_{t_i}(v_{i-1},\varphi_{i-1},d_{i-1})|
+ t_{i_0}-t_{i_0-1}\\
&\leq \sum_{i=1}^{i_0-1} \big(
|\cL^X_{t_i}(v_{i-1},\varphi_{i-1},d_{i-1})|
+3\big)
+ n^{1/3}.
\end{split}\]
As $\omega\in \mathcal{A}_{n}$ we get 
\[
|\bal(v,\floor{n^{1/2}})|
\leq
(i_0-1)(3n^{1/4}\log^{2}n+3)+n^{1/3}
\leq n^{5/12}\log^{3}n,
\]
for $n$ large enough, as required. 
\end{proof}

\begin{proof}[Proof of Proposition \ref{prop:balanceFact}]
This will follow from Proposition \ref{prop:preBalanceFact}
using a similar argument as for Corollary
\ref{cor:goodBalanceForEverybody}.
As in that argument, we fix some small enough $\eps>0$
and we use the same notation $m$, $\d$, $A_i$ and $\cD$.

Note that, on $\cD$,
for each $s\leq |\om|$ there is some (random) $i$
such that $\vec\om_s=\vec\om_{A_i}$.  Hence the probability in 
\eqref{eq:balanceFact} is at most
\[
\begin{split}
&\sum_{v\in V_n}\sum_{i=0}^m
\PP_{\beta}\left(\left\{ \bal_{\vec{\omega}_{A_i}}(v,\floor{n^{1/2}})
\geq n^{5/12}\log^{3}n\right\} 
\cap\left\{ |\cC_{\vec{\omega}_{A_{i}}}(v)|\geq \floor{n^{1/2}}\right\} \right)\\
&\qquad\quad+\PP(\cD^c\cup\{|\om|< n^\rho\}).
\end{split}
\]
We observe that under $\PP_{\beta}$ the distribution of $\vec{\omega}_{A_{i}}$
is the same as the distribution of $\vec{\omega}$ under $\PP_{\bar\beta_{i}}$,
with $\bar\beta_{i}=\beta\left(\beta_{0}/\beta_{1}+i \d\right)\in[\beta_{0},\beta_{1}]$.
Using Proposition \ref{prop:preBalanceFact} and a straightforward
bound on $\PP(\cD^c\cup\{|\om|< n^\rho\})$
we deduce that the probability in 
\eqref{eq:balanceFact} is at most
\[
n e^{\eps\log^{2}n} Ce^{-c\log^{2}n}+C_{2}e^{-c_{2}\log^{2}n}.
\]
Choosing $\eps>0$ small enough
concludes the proof. 
\end{proof}

\section{Poisson--Dirichlet coupling}\label{sec:PDcoupling}

In this section we prove our main result, Theorem \ref{thm:main}.
From the previous sections, Lemma \ref{large-lem} 
tells us that there are
cycles of size of the order $n$ and 
Propositions \ref{prop:preBalanceFact} and \ref{prop:balanceFact}
tell us that these cycles are `balanced'.  
The former lemma is stated in terms
of a sequentially constructed $\vec\om$ with a fixed number of links,
whereas the latter are formulated in terms of $\om$ sampled from the
Poisson law $\PP_\b$, so one of our tasks is to combine these two
descriptions.  Another task is to convert the balance-property of
Propositions \ref{prop:preBalanceFact} and \ref{prop:balanceFact}
into a quantitative result about the probability of splitting cycles
when a uniformly placed link is added, see Lemma \ref{lem:split-prob}.
Following that,
the main task will be to provide a coupling of a PD($\tfrac12$)
sample with the rescaled cycle sizes.  
We begin by introducing some relevant notation and facts,
as well as an outline of the proof.  Throughout this section, $\b>1$
and $\nu\in[0,1)$ are fixed.

\subsection{Preparation and outline}
\label{sec:prep}

The coupling with
PD($\tfrac12$)  will involve sequentially appending a small number of
uniformly, independently placed links to a random configuration $\om$.   
We will do this as follows.  First,  let $\om$ have distribution $\PP_\b$.  
Next, let $q\geq0$ be an integer-valued random variable
which is independent of $\om$ and bounded (as $n\to\oo$).  
Recall that $\vec\om$ denotes the ordered sequence of links in $\om$
and that $\vec\om_s$ denotes the first $s$ links of $\vec\om$.

To start the coupling, we will consider $\vec\om_s$
for $s=|\om|-q$.  
We construct a sequence $\vec\om'_t$ for 
$t\in\{s,s+1,\dotsc,s+q\}$, where $\vec\om'_s=\vec\om_s$
and the following $\vec\om'_t$ are obtained
by sequentially and independently  
appending in total $q$ uniformly placed links
one at a time.
Obviously the final configuration $\vec\om'_{s+q}$
then has $|\om|=s+q$ links, and it agrees in distribution with $\vec\om$.  
Letting $\cycles_t=\cycles_{\vec\om'_t}$ 
denote the cycle structure of $\vec\om'_t$,
it thus suffices to prove that Theorem \ref{thm:main}
holds for  $\cycles_q$.

Before proceeding, let us recall the key features 
of $\vec\om'_s=\vec\om_s$ 
which follow from our work in the
previous sections.  Precise statements are deferred to Section
\ref{sec:proofofmainthm}. 
First, it is clear that
(since $\b>1$)  we can find a constant
$c>\tfrac12$ such that 
the number of links $s$ satisifes $s\geq cn$ with high probability 
(converging to 1 as $n\to\oo$).   On this event 
Lemma \ref{large-lem} applies 
to $\vec\om'_s$, meaning (roughly speaking) that there are cycles of
size of the order $n$ which together occupy a fraction $\approx zn$
of all vertices.  
(Here $z$ is the same as in Proposition \ref{fact:loop_closure}.)
Second, since $q$ is bounded, Proposition \ref{prop:balanceFact}
certainly applies to $\vec\om'_s$.  Thus (with high probability), 
in any of the large cycles of
$\vec\om'_s$, any segment of $\floor{n^{1/2}}$ consecutive vertices in
that cycle has balance $|\bal|<n^{5/12}\log^3n$.

Next, let us describe the evolution of $\cycles_t$, $0\leq t\leq q$,
in a way which is suitable for the coupling with PD($\tfrac12$).
Since PD($\tfrac12$)
is a probability distribution on `continuous' partitions of the
interval $[0,1)$ it is convenient to represent $\cycles_t$ also
as a
(labelled) partition of $[0,1)$ (in the actual proof we will use a
different interval but the idea is the same).  The mapping is
fairly intuitive so we do not give a completely detailed description.  
Each vertex $v\in V_n$ is represented as a subinterval $I(v)$
of the form $[\tfrac in,\tfrac{i+1}n)$ for $0\leq i=i(v)\leq n-1$,
and this mapping is chosen so that the cycles of $\cycles_t$
become disjoint intervals of the form $[\tfrac in,\tfrac jn)$
for $0\leq i<j\leq n-1$, where if $u,v$ are consecutive in a cycle
then $I(u)$ and $I(v)$ are consecutive subintervals of 
$[\tfrac in,\tfrac jn)$ (interpreted cyclically).
The subintervals $I(v)$ are labelled using the labels $\up$, $\dn$ 
consistently with the orientations of the vertices within the cycles.  
See Figure \ref{arrowblocks-fig}.

\begin{figure}[hbt]
\centering
\includegraphics{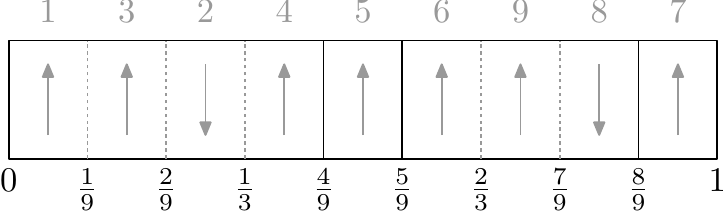}
\caption{
Representation of cycles as subintervals of $[0,1)$.  
Solid vetical lines delimit the cycles and dashed lines the vertices. 
The cycles here
are the same as in Figure \ref{fig:loopexample}, i.e.\
$(1^\up,3^\up,2^\dn,4^\up)$,
$(5^\up)$, $(6^\up,9^\up,8^\dn)$ and $(7^\up)$. 
}
\label{arrowblocks-fig}
\end{figure}

Naturally, this mapping is defined up to (i) cyclic rotations within
each cycle, (ii) overall reversal of all the labels (arrows) in
cycles, and (iii) the relative placement of the intervals
$[\tfrac in,\tfrac jn)$ representing the cycles within $[0,1)$.
Regarding the last item, the canonical way to order the 
intervals would be by
decreasing length, but we wish to keep the flexibility of reordering
them for the time being.

In this setting the dynamics of
uniformly placing links may be constructed 
using two independent uniform random variables $U,U'$ 
in $[0,1)$:
\begin{itemize}[leftmargin=*]
\item We first
sample the mark $m\in\{\cross,\dbar\}$ of the link with
  probability $\nu$ for $\cross$.
\item 
We then sample $U$ and set the first endpoint of the link 
to be $u$ if $U$ falls in the interval $I(u)$.
\item Before selecting the other endpoint we (i) move the 
(interval $[\tfrac in,\tfrac jn)$ representing the) cycle containing
$u$ to the front of $[0,1)$, then (ii) cyclically reorder this cycle 
so that $I(u)=[0,\tfrac1n)$.
\item Now we select the second endpoint by
 setting it to be $v$ if $U'\in I(v)$.
It may happen that $I(v)=I(u)$;  since this has probability
$\tfrac1n$ we will in practice be able to
disregard this possibility, but to
be definite let us say that nothing happens to the cycles in this
case. 
\end{itemize}
Having selected the endpoints of the link as well as its mark, we
apply the rules given in Section \ref{sec:large} for splitting,
merging or twisting cycles. 

Using this construction, the sequence
$\cycles_1,\dotsc,\cycles_q$ may be obtained starting with
$\cycles_0$ and using a sequence $\{(U_t,U'_t,m_t))\}_{t=1}^q$
of independent random variables with the above distributions.  

We now turn to the task of showing that the probability of splitting a
large cycle is close to $\tfrac12$, in a sense which we will make
precise. 
Let us assume that $\vec\om'_s$ belongs to the event
\be\label{eq:bal}
\bigcap_{v\in V_n} \big(
\{|\bal(v,\floor{n^{1/2}})|<n^{5/12}\log^3 n\}
\cup \{|\cC(v)|<\floor{n^{1/2}}\}
\big)
\ee
that any cycle of size at least $\floor{n^{1/2}}$ is `balanced'.
This event holds with high probability due to Proposition
\ref{prop:balanceFact}.
Recalling that the cycles $\cycles_0$ form a partition of the vertex
set $V_n$, we define a refinement $\fS$ of this partition into
`segments' as follows.
For each cycle $\cC\in\cycles_0$ satisfying $|\cC|\geq\floor{n^{1/2}}$
we fix a division of $\cC$ into non-intersecting 
sets of consecutive vertices, each
of size between $\floor{n^{1/2}}$ and $2\floor{n^{1/2}}$.
 If $|\cC|<\floor{n^{1/2}}$ then we declare 
$\cC$ to be a segment on its own.  
On the event in \eqref{eq:bal} we see 
using the triangle-inequality that 
each segment $S\in\fS$
satisfying $|S|\geq\floor{n^{1/2}}$ has balance 
$|\bal(S)|<2n^{5/12}\log^3 n$, where $\bal(S)$ is the difference
between the number of $\up$ and number of $\dn$ in $S$.

As we  proceed by adding links, and
thereby modify the cycle structure, we keep 
the partition $\fS$ into segments
fixed.  That is, at all later steps we
will  `remember' for each vertex
$v\in V_n$ which segment $S$  it belonged to at $t=0$.
After some steps a segment $S$ need no longer be a consecutive set of
vertices within a cycle, for example if a cycle is split in the middle
of $S$.  We say that a segment $S$ is 
\emph{untouched} at step $t\in\{1,\dotsc,q\}$ 
if none of the links placed in steps 
$1,2,\dotsc,t-1$ had an endpoint in $S$, otherwise the segment is
\emph{touched}. 
If $S$ is untouched then it is also `intact' 
in the sense that it is still consists of consecutive vertices in
some cycle, and $|\bal(S)|$ is unchanged from $t=0$.

In the representation of $\cycles_0$ as a collection of marked
subintervals of $[0,1)$, the segments $S$ become subintervals
(of length $\leq 2n^{-1/2}$) of the
intervals
representing the cycles (possibly we may have to interpret these
subintervals cyclically). 
 Recall that we used a uniform random variable $U'\in[0,1)$ to select
the second endpoint of a uniformly placed link.  We now modify this
construction slightly, and will instead use \emph{two}
uniform independent $U',U''\in[0,1)$.
We begin by sampling $U'$, and we note which segment $S$ it falls in
(more precisely, which subinterval representing such a segment).
If this segment $S$ is \emph{touched} 
then we let $v$ be the vertex selected by $U'$ as
before and we do not use $U''$.  However, if $S$ is untouched then
we do not record the precise location of $U'$ within $S$;  instead
we use $U''$ to independently select a uniform location within
$S$ and we select the second endpoint 
of the link to be $v$ if $U''\in I(v)$.

The following result is now straightforward.
Intuitively,  it tells us that the
probability of splitting a long cycle is very close to $\tfrac12$,
moreover  the choice of whether or not to split is almost
independent of the location where we propose to split.

\begin{lemma}\label{lem:split-prob}
Assume that the event in \eqref{eq:bal} holds 
at $t=0$.   At step $t\geq1$ (i.e.\ in the transition 
 $\cycles_{t-1}\to\cycles_t$),
let $u$ be the vertex selected by $U_t$ and let 
$\cC(u)\in\cycles_{t-1}$ be the cycle containing $u$.
Fix the orientation of $\cC(u)$ so that $u$ has label $\up$.
Suppose that $U'_t$ selects a segment $S$ which:
(i) is untouched, (ii) is in the cycle $\cC(u)$, i.e.\ $S\se\cC(u)$,
and (iii) has size $|S|\geq\floor{n^{1/2}}$.
Let $\a\in\{\up,\dn\}$ be the label of the vertex $v$ selected by
$U''_t$.  Then (on the event described) the conditional probability 
$p_t=\PP(\a=\,\up\,\mid U'_t)$ that $v$ has the same 
orientation as $u$ satisfies 
\[
|p_t-\tfrac12|\leq n^{-1/12}\log^3 n. 
\]
\end{lemma}
\begin{proof}
We have that $p_t=\#(\up\mbox{ in }S)/|S|$ so
\[
|p_t-\tfrac12|=\tfrac12|\bal(S)|/|S|\leq n^{-1/12}\log^3 n.\qedhere
\]
\end{proof}

We now give a brief outline of the rest of this section.
First, in Section \ref{sec:couplingdescription}, we describe a
slight modification of a coupling due to Schramm \cite{schramm}.
The coupling evolves a pair of partitions of the interval $[0,1)$ such
that, firstly, 
the marginal dynamics have PD($\theta$) as an invariant distribution,
and, secondly, the two partitions become `close'.
Moreover, for $\theta=\tfrac12$ these dynamics
are very similar to the dynamics of $\cycles_t$ above
(Schramm defined the coupling for $\theta=1$
but as we will see and as has been noted before \cite{G-U-W}, the
extension to $\theta\in(0,1]$ is completely straightforward).
Then, in Section \ref{sec:proofofmainthm}, 
we focus on the case $\theta=\tfrac12$
and show how an adaptation 
of Schramm's coupling allows us to couple a 
PD($\tfrac12$)-sample
to the `discrete' partition coming from the cycles $\cycles_t$.
This will allow us to prove Theorem \ref{thm:main}.
Lemma \ref{lem:split-prob} comes in here and, intuitively speaking,
by using the pair $(U',U'')$ as described above we 
``trade accuracy for independence'':
$U'$ will tell us the exact location for splitting in the 
PD($\tfrac12$)-distributed partition, whereas in $\cycles_t$ this is
decided by $U''$.  As we will see, 
the locations in the segment $S$ defined by
$U'$ and $U''$ are close enough to
each other, and $p_t$ is close enough to $\tfrac12$, that
the two partitions become more and more similar.

\subsection{Schramm's coupling}\label{sec:couplingdescription}

Fix any $\theta\in(0,1]$,
later we will take $\theta=\tfrac12$.
We will define a sequence 
$\big((\cY^t,\cZ^t):t=0,1,\dotsc\big)$ of pairs of random
partitions of $[0,1)$ into countably many intervals $[a,b)$, in
such a way that (i) the marginal dynamics are stationary for
PD($\theta$), and (ii) regardless of
starting configuration, $\cY^t$ and $\cZ^t$ become `closer' 
in a sense to be defined later. 

The subintervals $[a,b)$ of
 $[0,1)$ constituting the partitions $\cY^t$
and $\cZ^t$ will be called \emph{blocks}.
We will think of 
the blocks of $\cY^t$ and $\cZ^t$ as 
distinguishable, and as before leave some flexibility
about the relative placement of the blocks within $[0,1)$.
By a slight abuse of notation we will identify a block
$\cY^t_i\se[0,1)$ with its length $|\cY^t_i|\in[0,1]$.

Some of the blocks of $\cY^t$ will be \emph{matched} with blocks of
$\cZ^t$, and this relation is symmetric (if $\cY^t_i$ is matched with
$\cZ^t_j$ then $\cZ^t_j$ is matched with $\cY^t_i$).  Other blocks are
unmatched.  Matched pairs of
blocks have the same size, and
such pairs will be created in some instances of the
process we are about to describe.  The total length of all unmatched
blocks will be denoted by $R=R^t$ and the total length of matched
blocks $Q=Q^t$.  We place the matched blocks at the end of $[0,1)$
and the unmatched blocks at the beginning, and within the matched and
unmatched parts we order the blocks by decreasing size. 
See Figure~\ref{generic-fig}.

\begin{figure}[hbt]
\centering
\includegraphics{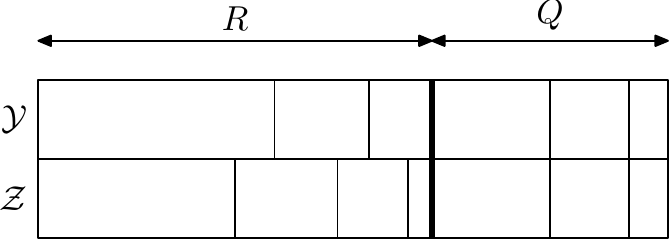}
\caption{
Example of a pair $(\cY,\cZ)$.  The matched blocks account
for a total $Q$ of the length, and the unmatched $R$, where $Q+R=1$.  
The thick vertical line indicates the border between the matched and
unmatched parts.
}
\label{generic-fig}
\end{figure}

A step 
of the coupling is completed with the help of three
independent random variables $U$, $U'$ and $W$, 
all uniformly distributed in $[0,1)$.
First $U$ is sampled, and if $U$ falls in the blocks
$\cY_i$ and $\cZ_j$ of $\cY$ and $\cZ$, respectively,
then we say that these two blocks 
of $\cY$ and $\cZ$ are \emph{highlighted}.
Moreover, the highlighted blocks are moved
to the front of $[0,1)$, see Figures \ref{u-unm-fig} and
\ref{u-m-fig}.
Then $U'$ is sampled and we do the following:
\begin{itemize}[leftmargin=*]
\item if (in either $\cY$ or $\cZ$) we have that 
$U'$ falls in a block different from the
highlighted one, then this block is merged with the highlighted block;
\item if
$U'$  falls in a highlighted block then we propose
a split of the highlighted block(s) at the position $U'$;
\item in the case of proposing a split, the split is carried out if we
  have that $W\leq\theta$.
\end{itemize}
Thus it is possible to merge blocks in both $\cY$ and $\cZ$,
to merge blocks in one but (propose a) split in the other, or to (propose a)
split in both.
 In the case when we propose a split in 
both $\cY$ and $\cZ$, 
note that the same $W$ is used for both, thus either both split or
neither.  In this case, if they split then at least two
of the newly created blocks are of the exact same size 
(see Figure \ref{v-m-split-fig}),  and those blocks
are then declared \emph{matched} and moved to the matched part.
Before the next step the blocks are sorted into the matched and
unmatched parts and ordered by size within those parts, as before.
Figures \ref{v-m-split-fig},  \ref{v-m-mg-fig} and \ref{v-mix-fig}
show some of the possible scenarios.

\begin{figure}[hbt]
\centering
\includegraphics{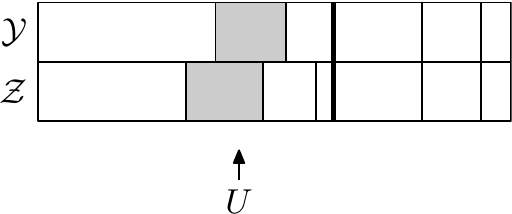}
\hspace{1cm}
\includegraphics{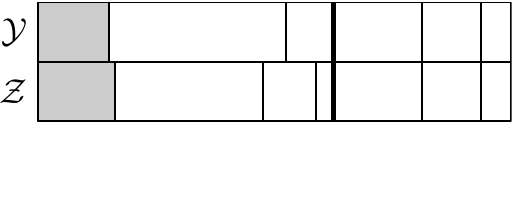}
\caption{
$U$ highlights a block in $\cY$ and a
block in $\cZ$ and they are moved to the front.  
In this case the highlighted blocks are unmatched.
}
\label{u-unm-fig}
\end{figure}

\begin{figure}[hbt]
\centering
\includegraphics{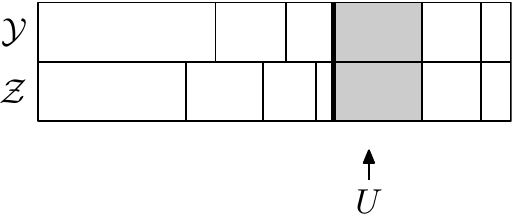}
\hspace{1cm}
\includegraphics{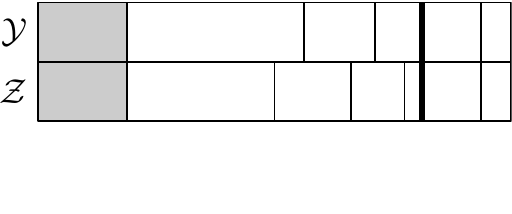}
\caption{
Example when  the highlighted blocks are matched.
}
\label{u-m-fig}
\end{figure}

\begin{figure}[hbt]
\centering
\includegraphics{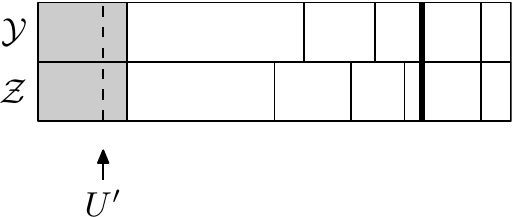}
\hspace{1cm}
\includegraphics{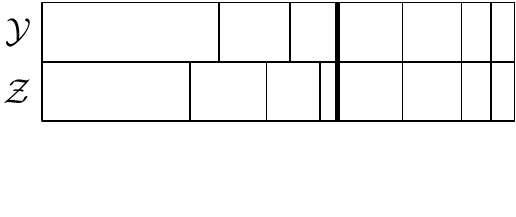}
\caption{
Example when a split is carried out in both $\cY$ and $\cZ$.
In this case the highlighted blocks are already matched and 
consequently all the formed blocks are matched.
}
\label{v-m-split-fig}
\end{figure}

\begin{figure}[hbt]
\centering
\includegraphics{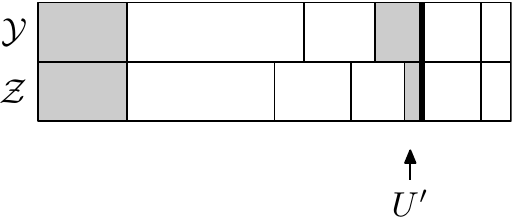}
\hspace{1cm}
\includegraphics{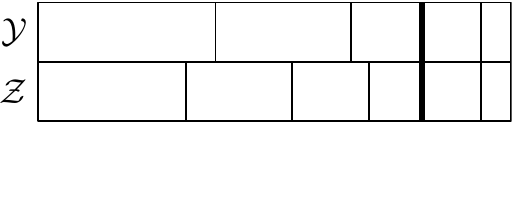}
\caption{
Example when
 the highlighted blocks are matched but are then merged with some
 unmatched blocks.
In this case the formed blocks are unmatched.
}
\label{v-m-mg-fig}
\end{figure}

\begin{figure}[hbt]
\centering
\includegraphics{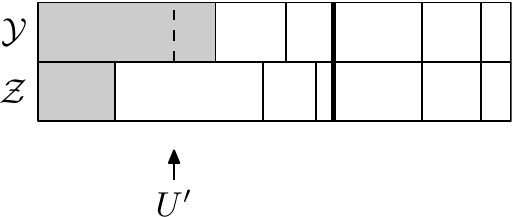}
\hspace{1cm}
\includegraphics{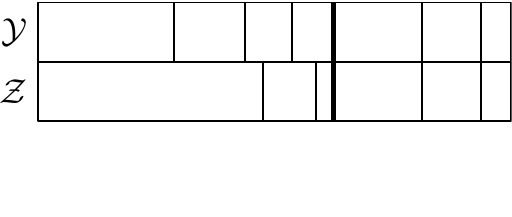}
\caption{
Example when
one block is split and one is merged.
}
\label{v-mix-fig}
\end{figure}

The  following result
about the marginal dynamics is due to
Tsilevich \cite{tsvi} (for $\theta=1$) and Pitman \cite{pitman}
(general $\theta$).
Another proof can be found in \cite[Theorem 7.1]{G-U-W}.

\begin{lemma}\label{PDinvariance}
If $\cY^0$ (respectively, $\cZ^0$) has distribution PD($\theta$)
then $\cY^t$ (respectively, $\cZ^t$) has distribution PD($\theta$)
for all $t\geq0$.
\end{lemma}

We will need quantitative results about how the sizes of the
largest unmatched blocks evolve under these dynamics.
We will present a sequence of results,
Lemmas \ref{L3.1} to \ref{L3.3}, 
which culminate in Corollary \ref{C3.4}.
As the proofs of these lemmas are identical or
nearly identical to the corresponding proofs in \cite{schramm} we
omit the details, but give comments
where there are differences in the case $\theta<1$.

Fix $\eps>0$ and introduce the following notation.  
Let $N_\eps(\cY^t)$ and $N_\eps(\cZ^t)$ denote the number of unmatched
blocks of size $\geq\eps$ in $\cY^t$ and $\cZ^t$, respectively,
and let $N^t_\eps=N_\eps(\cY^t)+N_\eps(\cZ^t)$
be  the total  number of unmatched blocks of size $\geq\eps$ after $t$ steps.
Let $\s(\eps,\cY^t)=\sum_{i} \cY^t_i\one_{\{\cY^t_i<\eps\}}$ be the
total length of blocks smaller than $\eps$ in $\cY^t$, and similarly
define $\s(\eps,\cZ^t)$.  Also let
$\ol\eps=\eps+\s(\eps,\cY^0)+\s(\eps,\cZ^0)$.

Before presenting the lemmas about the coupling, we note the following
a-priori estimates:
\begin{proposition}\label{prop:apriori}\hspace{1cm}
\begin{enumerate}[leftmargin=*]
\item If for some $C>0$ and all $\eps\in(0,1)$ we have 
\be\label{eq:sigma-bd}
\EE[\s(\eps,\cY^0)]\leq C\eps \log(\tfrac1\eps)
\ee
then for some $C'>0$ we have 
$\EE[N_\eps(\cY^0)]\leq C'\log^2(\tfrac1\eps)$.
\item If $\cY^0$ has distribution PD($\theta$) with $\theta\in(0,1]$
  then  $\EE[\s(\eps,\cY^0)]\leq \eps$.
\end{enumerate}
\end{proposition}

The proof is sketched (for $\theta=1$) in \cite{schramm}.
For completeness we give details in 
Appendix \ref{app:apriori}.
In the proof of Theorem \ref{thm:main}, 
$\cY^0$ will have the PD($\tfrac12$)-distribution while 
$\cZ^0$ will satisfy a bound of the form \eqref{eq:sigma-bd}.
Thus, in the following sequence of lemmas,  one
should think of $\ol\eps$ as being of the order 
$\leq \eps\log(\tfrac1\eps)$ as $\eps\to0$, and $N^0_\eps$ as of the order
$\leq \log^2(\tfrac1\eps)$.

In the next few results we will be working conditionally on $(\cY^0,\cZ^0)$,
hence $\ol\eps$ and $N_\eps^0$ 
will be treated as constants.
We let $q$ be a random time, independent 
of the chain $((\cY^t,\cZ^t):t\geq0)$, 
 and write $\eta=\max\{\PP(q=t):t\geq0\}$. 

Let $y^t_{(1)}$ and $z^t_{(1)}$ denote the largest
\emph{unmatched} blocks in $\cY^t$ and $\cZ^t$, respectively.  
In the following result, note that $R^t-y^t_{(1)}\vee z^t_{(1)}$ is
small if most of the unmatched length $R^t$ is covered by the largest unmatched 
block in either $\cY^t$ or in $\cZ^t$.
Hence the product $R^t(R^t-y^t_{(1)}\vee z^t_{(1)})$
is small if either $R^t$ is small (which is what we want),
or the unmatched part contains a large block (which can be
handled because such a situation is `unstable').

\begin{lemma}
\label{L3.1}
Conditionally on $\cY^0,\cZ^0$, 
\[
\EE\big[R^q(R^q-y^q_{(1)}\vee z^q_{(1)})\big]\leq
\frac\eta2 N_\eps^0+5\ol\eps\:\EE[q].
\]
\end{lemma}
When applying this and the following estimates, the main case 
will be when
$q$ is uniformly distributed on $\{0,1,\dotsc,\ceil{\eps^{-1/2}}-1\}$.
Then $\eta$ is approximately $\eps^{1/2}$ and $\EE[q]$ is of the order 
$\eps^{-1/2}$.  If $\ol\eps$ and $N^0_\eps$ are of
the order indicated above then the right-hand-side is small
(of the order $\leq \eps^{1/2}\log^2(\tfrac1\eps)$).
\begin{proof}
The proof is essentially identical to the proof in 
\cite[Lemma 3.1]{schramm} and uses that on the
event that up to time $q$ no blocks of size $\leq \eps$ are
created or merged,  $N^t_{\eps}$ is non-increasing for $t<q$. 
The only extra case which arises for $\theta<1$
is when, going from $t$ to  $t+1$, 
two blocks are merged in $\cY$ (respectively, $\cZ$)
but a split is proposed for $\cZ$ (respectively, $\cY$) and 
not accepted. In this
case we see that $N^{t+1}_{\eps}=N^t_{\eps}-1$, hence $N^t_{\eps}$ is
still non-increasing.
\end{proof}

Write $y_{(2)}^t$ for the second-largest unmatched block
in $\cY^t$.

\begin{lemma}\label{L3.2}
For $\rho\in(0,1)$ we have (conditionally on $\cY^0,\cZ^0$) 
\[
\PP\big(R^q-(y_{(1)}^q+y_{(2)}^q)>\rho\big) 
\leq 2^6\theta^{-1}\rho^{-4} \big(\tfrac\eta2 N_\eps^0+5\ol\eps\EE[q+1]\big).
\]
\end{lemma}
The lemma says that the two largest unmatched entries together
dominate the unmatched part (if $\eta$, $q$, $\ol\eps$ and $N^0_\eps$
are of the order indicated above, then the right-hand-side is small as
long as, say, $\rho\geq\eps^{1/10}$).

\begin{proof}
This proof is virtually identical to the proof of 
\cite[Lemma 3.2]{schramm}.  We consider whether the event $\cR=\{R^q-z^q_{(1)}<\rho/4\}$
occurs or not.  In the case when $\cR$ does not occur
we can apply Lemma \ref{L3.1} exactly as in
\cite{schramm}.  In the case when $\cR$ does occur, the key
observation in \cite{schramm} is that 
there is good probability that
$z^q_{(1)}$ splits into two blocks of size $\geq \rho/4$ 
while two unmatched blocks of $\cY^q$ merge, allowing us to
apply Lemma \ref{L3.1} in the next step instead. 
The only extra consideration for $\theta<1$ is that the
split must be accepted, which happens with probability $\theta$,
resulting in the factor $\theta^{-1}$ in the statement of the lemma.
\end{proof}

We next bound the `average' probability of having a large unmatched
block. Its corollary, Corollary \ref{C3.4}, is especially important for us.
\begin{lemma}\label{L3.3}
Let $\eps,\rho\in(0,1)$ and let $t>0$ and $k$ be such
that $t\geq 2^k/\rho$.  Then (conditionally on $\cY^0,\cZ^0$) 
\be
t^{-1}\sum_{s=0}^{t-1} \PP(y_{(1)}^s>\rho)\leq 
C [k^{-1}\rho^{-1}+2^{4k}\rho^{-5}(N_\eps^0/t+\ol\eps t)],
\ee
for some constant $C$.
\end{lemma}
\begin{proof}
The proof is identical to the proof of \cite[Lemma 3.3]{schramm} where we insert the bounds from Lemmas \ref{L3.1} and \ref{L3.2} when the bounds from  \cite[Lemma 3.1]{schramm} and  \cite[Lemma 3.2]{schramm} are used.
\end{proof}

We now make some additional assumptions on $(\cY^0,\cZ^0)$
and $q$, which allow us to obtain a more explicit bound on
$\PP(y_{(1)}^q\geq \rho)$.  
As usual we work conditionally on $(\cY^0,\cZ^0)$.

\begin{corollary}\label{C3.4}
Assume that $\ol\eps<1$ and that 
\be\label{gamma-eq}
(\ol\eps)^{1-\g}\leq \eta \leq (\ol\eps)^\g/(N_\eps^0\vee 1),
\quad\mbox{ for some } \g\in(0,\tfrac12).
\ee
Then for some $C=C(\g)$ we have that for all
$\rho\in(0,1)$, 
\be\label{c34-eq}
\PP(y_{(1)}^q\geq \rho)\leq \PP(q\geq 1/\eta)+
\frac{C}{\rho\log(1/{\ol\eps})}.
\ee
\end{corollary}
\begin{proof}
The proof is identical to the proof of \cite[Corollary 3.4]{schramm}
(in \cite{schramm} there is an additional parameter $\lambda$ which we
have set to 1).
\end{proof}

Note that if $q$ is uniformly distributed on
$\{0,1,\dotsc,\ceil{\eps^{-1/2}}-1\}$ and $N^0_\eps$ is of the order
at most $\log^2(\tfrac1\eps)$, as discussed above, then
\eqref{gamma-eq} holds.  Moreover, in this case
$\PP(q\geq 1/\eta)=0$.
If $\ol\eps$ is at most of the order $\eps\log(\tfrac1\eps)$
as discussed above then the right-hand-side of \eqref{c34-eq}
can be made arbitrarily small by choosing $\eps$ small.

\subsection{Proof of Theorem \ref{thm:main}}\label{sec:proofofmainthm}

We now turn to the proof of our main result.  Let $\eps>0$, to be
fixed later.
Recall the set-up of Section \ref{sec:prep}:  $\om$ is sampled from
$\PP_\b$ where $\b>1$, and we consider $\vec\om'_s=\vec\om_s$
for $s=|\om|-q$.  We take $q$
to be uniformly distributed on $\{0,1,\dotsc,\ceil{\eps^{-1/2}}-1\}$.
Let $\cA_0$ be the event that the following conditions hold:
\begin{itemize}[leftmargin=*]
\item for some $c>\tfrac12$ we have $s\geq cn$;
\item the event in \eqref{eq:bal} holds for $\vec\om'_s=\vec\om_s$;
  and 
\item the random graph $G(n,s)$, which has an edge wherever $\vec\om_s$
  has at least one link, has a unique giant connected component $V_G$
  containing between $0.99 zn$ and $1.01zn$ vertices, 
and any other connected
  component has size at most $\log^2n$.
(Here $z$ is the same as in Proposition \ref{fact:loop_closure}.)
\end{itemize}
 In the following discussion we will assume that $\cA_0$ holds, as 
$\PP(\cA_0^c)=o(1)$ as $n\to\oo$.
Note that the cycles $\cycles_0$ refine the components of $G(n,s)$,
hence (on $\cA_0$) a cycle is either contained in $V_G$ or it has 
size $\leq \log^2 n$.

We take $\cY^0$ to have distribution PD($\tfrac12$).
Roughly speaking, $\cZ^0$ will be obtained from the cycles
$\cycles_0$ and we want to 
use the coupling from Section \ref{sec:couplingdescription} to
obtain $(\cY^q,\cZ^q)$.  The main modification of the coupling is that
we 
use the construction in Lemma \ref{lem:split-prob} for splitting in
$\cZ$.  There are also several minor modifications to take into account.  
In what follows we work conditionally on $\vec\om'_0$.

We subdivide the cycles of $\cycles_0$ into segments
$S\in\fS$ as in Section \ref{sec:prep}.  
Write $m=|V_G|$.  We let $\cZ^0$ be 
a representation of the cycles $\cycles_0$ as
intervals as in Section \ref{sec:prep}, 
but now as subintervals of 
$[0,\tfrac nm)$ rather than $[0,1)$.  Thus
each vertex $v$ is represented by a
subinterval $I(v)$ of the form $[\tfrac im,\tfrac{i+1}m)$
where $0\leq i=i(v)\leq n-1$.  Note that $[0,\tfrac nm)$ has
length roughly $\tfrac1z$.   In keeping with the terminology of the
previous subsection, we refer to the intervals which represent the
cycles as \emph{blocks}. 
 The subintervals $I(v)$ representing the vertices 
are labelled using $\up,\dn$, as
 before.  Clearly, a cycle of size $\geq\floor{n^{1/2}}$
in $\cycles_0$ is represented as a block of size 
$\geq \floor{n^{1/2}}/m$ in $\cZ^0$.

We place the blocks representing 
the cycles of $\cycles_0$ which lie in the giant component $V_G$ at the
start of $[0,\tfrac nm)$, i.e.\ in $[0,1)$.  The
blocks representing the  remaining cycles are
placed in $[1,\tfrac nm)$. 
Within $[0,1)$ we will later have 
matched and unmatched blocks, as in
Section \ref{sec:couplingdescription},
and again we place the unmatched blocks first and within the matched
and unmatched parts we order the blocks by decreasing size. 
See Figure \ref{PDhalf-fig}.

\begin{figure}[hbt]
\centering
\includegraphics{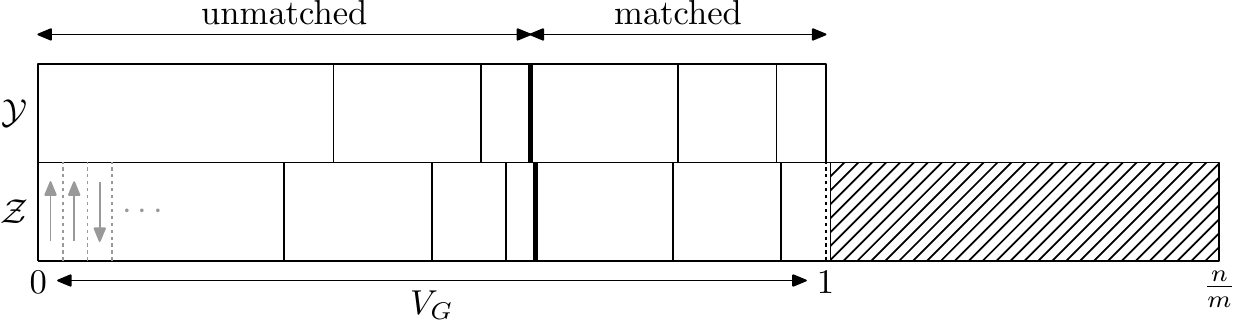}
\caption{
$\cZ^t$ consists of blocks representing the cycles $\cycles_t$,
placed in the interval $[0,\tfrac nm)$.
Cycles belonging to $V_G$ are placed first, roughly in the interval
$[0,1)$, and are sorted into those matched with a block of $\cY^t$ and
those not.  Matched blocks can differ slightly in size.  Also $\cZ^t$
will after a few steps have an `overhang' since the giant component
grows.  The hatched part consists of blocks representing 
cycles with vertices that are not in $V_G$. 
}
\label{PDhalf-fig}
\end{figure}

We will define dynamics for $(\cY^t,\cZ^t)$ such that
the marginal dynamics  for $\cY^t$ are as in Section 
\ref{sec:couplingdescription} with $\theta=\tfrac12$, and 
 the marginal dynamics  for $\cZ^t$ are as in Section
\ref{sec:prep}.  
Thus $\cY^t$ will have distribution PD($\tfrac12$) for all $t$, due to
Lemma \ref{PDinvariance}, and $\cZ^t$ will be a representation of the
cycles $\cycles_t$ as intervals. 

In order to be able to define a \emph{successful coupling} later on, 
we will need the notion of a \emph{forbidden set} 
$F^t\se[0,\tfrac nm)$ (for any given time $t$).
This set will arise due to small errors which accumulate during the process.
Initially, for $t=0$, we set $F^0=\es$. In later steps, we define $F^t$
as consisting of the following parts:
\begin{itemize}[leftmargin=*]
\item First, all segments which have been \emph{touched} up to time $t$ (or rather, the union of the
  $I(v)$ for $v$ belonging to touched segments) in $\cZ$ are forbidden.
\item Second, $F^t$ contains 
an \emph{overhang} 
(defined as $\{s\in(1,\tfrac nm]: s\in I(u) \mbox{ for some }u\in V_G\}$)
which arises because in $\cZ$ cycles outside
  the giant component may merge with cycles inside the giant
  component, meaning that the 
giant component grows with time.
\item Third, it will be necessary to allow matched blocks, defined shortly, to have
  slightly different sizes, rather than the exact same size as in
  Section \ref{sec:couplingdescription}.  When the blocks of
  $\cY^t$ and   $\cZ^t$ are lined up as in Figure \ref{PDhalf-fig}, the subset in
  $[0,\tfrac nm)$ where  part of a matched block does not overlap
  with its partner belongs to the forbidden set $F^t$. 
\item Also, $[0,\tfrac 1m)$ is forbidden.  To understand the meaning
  of this, recall that once
  $U$ has been sampled, the  block it highlighted is moved and rotated
  so that the corresponding $I(u)$ is moved to $[0,\tfrac 1m)$.
  Putting this interval in the forbidden set will simply be a way to
  enforce that all links have two different endpoints.
\end{itemize}
Once we have described precisely the transitions in our process we
will easily be able to bound the size $|F^t|$ by a very small number,
see \eqref{Ft-bd}.

Let us now define a step of the process.  
Steps will again be accomplished using independent uniform random
variables $U$, $U'$, $U''$ and $W$, but now $U$ and $U'$ will be
uniformly distributed in $[0,\tfrac nm)$ while $U''$ and $W$ are still
uniform in $[0,1)$.   For $\cZ$ we will also sample the mark
$m\in\{\cross,\dbar\}$ of the corresponding link in each step.
In what follows we will assume that $U$ and $U'$ never fall in the
forbidden set. For concreteness, if $U$ or $U'$ fell in the forbidden set we would declare the process failed and stop.

Firstly, if $U$ or $U'$ falls outside $[0,1)$ we perform
the corresponding transition in $\cZ$, using the rules from 
Section \ref{sec:prep},
but do nothing to $\cY$. 
Let us now assume that $U$, $U'$ both fall in $[0,1)$.  
When $U$ and $U'$ highlight different blocks these blocks are merged,
as before.  For $\cZ^t$ the labels $\up$, $\dn$ must be handled
appropriately, taking into account also the mark
$m\in\{\cross,\dbar\}$ of the link, as in Section \ref{sec:large}.
In the case when $U'$ falls in a block highlighted by $U$
(in either $\cY$, $\cZ$ or both) it has to be decided if a split
should be carried out.  There are three cases for how this is decided.
First, if a split is proposed in $\cY$ only (Figure
\ref{v-mix-fig}) then it is carried out if $W\leq \tfrac12$,
so this case is the same as in Section
\ref{sec:couplingdescription}.  Second, if a split is proposed only in
$\cZ$  then we decide whether to carry it out by looking at
the labels $\up$, $\dn$ in the intervals $I(u)$ and $I(v)$ selected
by $U$ and $U'$, 
as well as the mark $m\in\{\cross,\dbar\}$ of the link, 
and applying the rules of Section
\ref{sec:large}.  In this case we do not need to use $U''$.
However, the third and most important case is when a split is proposed
in both $\cY$ and $\cZ$.  In this case we do the following:
\begin{enumerate}[leftmargin=*]
\item In $\cY$ we record the exact location of $U'$.  If we decide to
  carry out the split in $\cY$, then it will be done at the location
  of $U'$.
\item  In   $\cZ$ we only record the segment $S$ in which $U'$ falls;
\item Then we use $U''$ to independently sample a uniform point
  within $S$ and make the splitting decision for $\cZ$ as in Section
  \ref{sec:prep}.    
\end{enumerate}
It only remains to specify how we decide whether or not to split in
$\cY$.  Let $v$ be such that   $U''\in I(v)$.   
Assume that the block of $\cZ$ highlighted by $U$ has 
size $\geq \floor{n^{1/2}}/m$, and that $U$, $U'$ did not fall in the
forbidden set $F^t$.  We are then able to apply Lemma
\ref{lem:split-prob}.
Thus the conditional probability $p_t$ that $I(u)$ and $I(v)$ have the
same label $\up$ is within $n^{-1/12}\log^3 n$ of
$\tfrac12$.  Depending on whether $p_t$ is bigger or smaller than
$\tfrac12$, and also on the mark $m\in\{\cross,\dbar\}$ of the link,
the conditional probability of splitting in $\cZ^t$ is thus either
slightly above or slightly below $\tfrac12$.  We wish to `maximally
couple' the decision whether or not to split in $\cZ$ with the
decision in $\cY$, but keeping the splitting probability for $\cY$ at
exactly $\tfrac12$.  Let us describe this assuming $p_t\leq \tfrac12$,
the other case is similar.
\begin{itemize}[leftmargin=*]
\item If the mark $m=\cross$, recall that this means that
  we split in $\cZ$ if $I(v)$ has label $\up$, i.e.\ with probability
  $p_t\leq\tfrac12$.  Our rule for 
  $\cY$ is then:  split in $\cY$ if $I(v)$ has label $\up$, but if
  $I(v)$ has label $\dn$ split in $\cY$ anyway
with probability $(\tfrac12-p_t)/(1-p_t)$
  (independently of all other choices).
\item If the mark $m=\dbar$, this means that
  we split in $\cZ$ if $I(v)$ has label $\dn$, i.e.\ with probability
  $1-p_t\geq\tfrac12$.  Our rule for 
  $\cY$ is then: do nothing (no split) in $\cY$ if $I(v)$ has label $\up$, but if
  $I(v)$ has label $\dn$ also do nothing in $\cY$ with probability
  $(\tfrac12-p_t)/(1-p_t)$   (independently of all other choices).
\end{itemize}
It is not hard to check that these rules ensure that the probability
of splitting in $\cY$ is exactly $\tfrac12$, independently of the
location $U'$ of the proposed split. We can also see from this that the probability of $\cY$ and $\cZ$
making different choices (i.e. one splits and the other one twists) 
is $|p_t-\tfrac12|\leq n^{-1/12}\log^3n$.

If the decision is to split in \emph{both} $\cY$ and $\cZ$, the blocks
created are declared \emph{matched} as in Section
\ref{sec:couplingdescription} (if the blocks which split were already
matched we get two pairs of matched blocks, otherwise one pair).
Note that $U'$ and $U''$  differ by at most $2 \floor{n^{1/2}}/m$ due
to the upper bound on the size of segments $S$;  
this will give us a bound on how much matched blocks can differ in size.

There is one final case in which we need to specify the rules for
deciding to split, which is when a split is proposed in both $\cY$ and
$\cZ$ but the block of $\cZ$ has size $<\floor{n^{1/2}}/m$.
This is unlikely and we will see that we can assume that this does not
occur, but to be definite let us say that in this case we split in
$\cY$ if $W\leq\tfrac12$.

Now we turn to bounding the size of the forbidden set $F^t$.
We claim that, for any $t\in\{0,1,\dotsc,q\}$ we have 
\be\label{Ft-bd}
|F^t|\leq \frac{7 \eps^{-1/2} n^{1/2}}{m}.
\ee
Indeed, after $t$ steps we have added at most $4t\floor{n^{1/2}}/m$
due to touched segments, at most $t \log^2n/m$ due to overhang, and at
most $2t \floor{n^{1/2}}/m$ due to the size-difference of matched
blocks.  Adding to this $1/m$ for $[0,\tfrac1m)$ and recalling that
$t\leq q\leq \eps^{-1/2}$ we arrive at \eqref{Ft-bd}.

Let us say that the coupling $(\cY^q,\cZ^q)$ was \emph{successful},
denoting this event by $\cG$, if the following occur for all
$t\in\{1,\dotsc,q\}$:
\begin{itemize}[leftmargin=*]
\item $U_t$, $U'_t$ do not fall in the forbidden set $F^t$; 
\item in step $t$ we do not propose to split a block of $\cZ$ which
  has size   $<\floor{n^{1/2}}/m$; and
\item if at step $t$ it is proposed to split a block in both $\cY$ and
  $\cZ$ then we either split in both or in neither.
\end{itemize}
Using \eqref{Ft-bd}  and Lemma \ref{lem:split-prob} and recalling that
$q\leq \eps^{-1/2}$ we get:
\be\begin{split}\label{fail-prob}
\PP(\cG^c)&\leq 2\eps^{-1/2}\frac{7\eps^{-1/2}n^{1/2}/m}{n/m}
+\eps^{-1/2} \frac{n^{1/2}/m}{n/m}
+ \eps^{-1/2}n^{-1/13}\\
&\leq 16\eps^{-1} n^{-1/13}.
\end{split}\ee
Here we bounded the probability that we make different decisions for
splitting a large block in $\cY$ and in $\cZ$ by $n^{-1/13}$, which is
valid for large enough $n$.  

We can now put the different pieces together and wrap up the
proof of Theorem \ref{thm:main}.  
Having defined $(\cY^q,\cZ^q)$, let us now order the blocks in both of
them by decreasing size, and let us think of them as two infinite sequences by
appending infinitely many 0's at the end.  Let $\d>0$ be arbitrary
and write 
\[
\cD=\big\{\|\cY^q-\cZ^q\|_\oo>\d\big\}.
\]
It suffices to show that $\PP(\cD)$ can be made arbitrarily small by
choosing $\eps>0$ small and $n$ large.

Recall that we have been working on
the event $\cA_0$ defined in the beginning in this proof.
Also recall the quantities $\s(\eps,\cY)$ and
$N_\eps(\cY)$ defined in Section \ref{sec:couplingdescription}.   
We now define $\s(\eps,\cZ)$ and $N_\eps(\cZ)$ similarly but counting
only those blocks which intersect $[0,1)$.  Given this,
 $\ol\eps$ and $N^0_\eps$ are given as in Section \ref{sec:couplingdescription}.   
By \eqref{large1} in Lemma \ref{large-lem} we have for some $C>0$ and
any $\eps>0$ that $\EE[\s(\eps,\cZ^0)]\leq C\eps\log(\tfrac1\eps)$.
Hence by Proposition \ref{prop:apriori} we also have 
$\EE[N_\eps(\cZ^0)]\leq C'\log^2(\tfrac1\eps)$.
By the same Proposition, the same bounds also apply to $\cY^0$.  
Hence $\EE[\ol\eps]\leq 3C\eps\log(\tfrac1\eps)$
and $\EE[N^0_\eps]\leq 2 C'\log^2(\tfrac1\eps)$.
Defining the events $\cA_1=\{\ol\eps\leq \eps^{3/4}\}$
and $\cA_2=\{N^0_\eps\leq \eps^{-1/4}\}$ and using Markov's
inequality, we get
\[
\PP(\cA_0\cap\cA^c_1)\leq 3C\eps^{1/4}\log(\tfrac1\eps),\qquad
\PP(\cA_0\cap\cA^c_2)\leq 2C'\eps^{1/4}\log^2(\tfrac1\eps).
\]
On $\cG\cap\cA_0\cap\cA_1\cap\cA_2$, we can apply Corollary
\ref{C3.4}, with $\g=\tfrac15$ say, and $\rho=\d/2$. 
(We use $\rho=\d/2$ rather than $\d$ to account for such things as the
size-difference of matched blocks;  thus $n$ should be taken
sufficiently large.)
This gives, for some $C''>0$ and $n$
large
\[\begin{split}
\PP(\cD)&\leq \PP(\cA_0^c)+
\PP(\cA_0\cap (\cA^c_1\cup\cA^c_2\cup\cG^c))+
\PP(\{y^q_{(1)}\geq \d/2\}\cap \cA_0\cap\cA_1\cap\cA_2\cap\cG)\\
&\leq o(1)+ C''\eps^{1/4}\log^2(\tfrac1\eps)+
\frac{2C}{\d\log(\eps^{-3/4})}.
\end{split}\]
The right-hand-side can be made arbitrarily small by picking $\eps>0$
small and $n$ large. 
(Since $\cG$ involves the entire process, to be completely rigorous
Lemmas \ref{L3.1}--\ref{L3.3} and Corollary \ref{C3.4} 
should be proved on the event $\cG$.  This can be done by working with
the time 
$\min\{q,\tau\}$ where $\tau$ is the first
time at which $\cG$ fails.
From \eqref{fail-prob} we see that, with high
probability, $\tau>q$ and so 
the only change is the addition of an $o(1)$ term.)
\qed

\appendix

\section{Proof of Proposition \ref{fact:loop_closure}}
\label{sep-app}

We use the notation from Proposition \ref{fact:loop_closure},
noting that
$\cS=\{Z_t>-1\mbox{ for all }t\geq0\}$.
Also note that $z=\PP(\cS)>0$ when $\b>1$, since if $z=0$ then  
$\PP(\exists t: Z_t=-1)=1$ which by the Markov property would
imply $\PP(\liminf_{t\to\oo} Z_t=-\oo)=1$, contradicting the fact that 
$Z_t\to+\oo$ almost surely.

It will be useful to
consider the \emph{times of record minima} $m_k$, which are defined as
follows.  Let $\tau_1,\tau_2,\dotsc$ 
denote the jump times of $Z$ (equivalently, of $N'$).  First we define
$m_1:=\tau_1$, and then inductively 
\[
m_{k+1}:=\min\{\tau_j>m_k: Z_{\tau_j-}<Z_{m_k-} \},
\qquad\mbox{where } \min\varnothing=\oo.
\]
Importantly, the $m_k$ are stopping times and they 
characterize the frontier time $\ell_1$ by
$\ell_1=\max\{m_k:m_k<\oo\}$, i.e.\ $\ell_1$ equals the last record
time.  The later frontier times $\ell_k$ may be expressed similarly using
the record minima of $Z^{(k)}$.
Using that the $m_k$ are stopping times we have 
for all $k\geq1$ that
\be\label{m-z-eq}
\PP(m_k<\oo)=(1-z)^{k-1}.
\ee

\begin{proof}[Proof that $z=1-e^{-\b z}$:]
Write $\d=1-z=\PP(\exists t: Z_t=-1)$.
By conditioning on the first time when $Z_t$ hits $-1$ we see that
$\PP(\exists t:Z_t=-j)=\d^j$ for all $j\geq1$.
Since $N'_t$ only takes integer
values it follows $\d^j=\PP(\exists k\geq0:N'_{k+j}=k)$.  
From this and 
the Markov property at time 1 we find that
\[
\d=\sum_{j\geq0} e^{-\b}\tfrac{\b^j}{j!}\d^j=e^{-\b(1-\d)},
\]
as claimed.
\end{proof}

\begin{proof}[Proof that, given $\cS$,
the sequence 
$\big\{\big(\D_{k},(Z^{(k)}_t)_{0\leq t<\D_k}\big)\big\}_{k=0}^{+\infty}$
is i.i.d.:]
We start by establishing that
\be\label{Zk-eq}
Z^{(k)} \eqdist \big(Z\mid Z\in\cS\big),\qquad
\mbox{for all }k\geq1.
\ee
This is reasonable since, for example,
after $\ell_1$ we know that $Z$ does not set a new record minimum,
which is the same as saying that $Z^{(1)}$ does not hit $-1$.
Using induction, \eqref{Zk-eq} follows from these two
equalities in law:
\be\begin{split}\label{Z1-Z0-eq}
&Z^{(1)}\eqdist \big(Z\mid Z\in\cS\big),
\quad\mbox{ and}\\
&\big(Z^{(1)}\mid Z\in\cS\big)\eqdist 
\big(Z\mid Z\in\cS\big).
\end{split}\ee

To prove \eqref{Z1-Z0-eq},
let $\cB$ be some event.  Using the description of $\ell_1$ in
terms of the stopping times $m_k$ we see that
\[\begin{split}
\PP(Z^{(1)}\in\cB)&=\sum_{k\geq1} 
\PP\big(m_k<\oo,m_{k+1}=\oo,(Z_{m_k+t}-Z_{m_k})_{t\geq0}\in\cB\big)\\
&=\sum_{k\geq1} 
\PP\big(m_k<\oo,(Z_{m_k+t}-Z_{m_k})_{t\geq0}\in\cB\cap\cS\big)\\
&=\sum_{k\geq1}  \EE\big[\one_{\{m_k<\oo\}}
\PP\big((Z_{m_k+t}-Z_{m_k})_{t\geq0}\in\cB\cap\cS\mid m_k\big)
\big]\\
&=\PP(Z\in\cB\cap\cS)\sum_{k\geq1} \PP(m_k<\oo)
=\PP(Z\in\cB\cap\cS)\sum_{k\geq1} (1-z)^{k-1}\\
&=\PP(Z\in\cB\cap\cS)/\PP(Z\in\cS),
\end{split}\]
thus $Z^{(1)}\eqdist (Z^{(0)}\mid Z^{(0)}\in\cS)$.

It will be useful to note the following:
\[\begin{split}
\PP(\ell_1\leq a\mid Z\in\cS)&=
\sum_{k\geq1} 
\PP\big(m_k\leq a,m_{k+1}=\oo,Z\in\cS\big)/\PP(\cS)\\
&=\sum_{k\geq1} 
\PP\big(m_k\leq
a,Z_{m_{k}-}>-1,(Z_{m_k+t}-Z_{m_k})\in\cS\big)/\PP(\cS)\\
&=\sum_{k\geq1}  \PP\big(m_k\leq a, Z_{m_k}>0).
\end{split}\]
In particular, letting $a\to\oo$,
\[
1=\sum_{k\geq1}  \PP\big(m_k<\oo,Z_{m_k}>0).
\]
Using this we find that 
\[\begin{split}
\PP(Z^{(1)}\in\cB\mid Z\in\cS)
&=\PP(\cS)^{-1}\sum_{k\geq1}
\PP(m_k<\oo,Z_{m_k}>0,(Z_{m_k+t}-Z_{m_k})_{t\geq0} \in\cB\cap\cS)\\
&=\frac{\PP(Z\in\cB\cap\cS)}{\PP(\cS)}
\sum_{k\geq1}  \PP\big(m_k<\oo,Z_{m_k}>0)\\
&=\PP(Z\in\cB\mid Z\in\cS).
\end{split}\]
Thus 
$\big(Z^{(1)}\mid Z\in\cS\big)\eqdist  \big(Z\mid Z\in\cS\big).$

It remains to prove independence.  For this
it suffices to show that, for any $r\geq2$
and events $\cB_1,\dotsc,\cB_r$ we have
\[
\PP((Z^{(i)}_t)_{0\leq t<\D_i}\in\cB_i\;\forall i=1,\dotsc,r)=
\prod_{i=1}^r \PP((Z^{(i)}_t)_{0\leq t<\D_i}\in\cB_i).
\]
We give details for the case $r=2$, the other cases are  similar. 
Using \eqref{Zk-eq} we have 
\begin{multline}\label{X1-eq}
\PP((Z^{(1)}_t)_{0\leq t<\D_1}\in\cB_1,
(Z^{(2)}_t)_{0\leq t<\D_2}\in\cB_2)\\=
\PP((Z^{(0)}_t)_{0\leq t<\D_0}\in\cB_1,
(Z^{(1)}_t)_{0\leq t<\D_1}\in\cB_2\mid Z^{(0)}\in\cS).
\end{multline}
Note that
\be\begin{split}\label{X0-eq}
&\PP((Z^{(0)}_t)_{0\leq t<\D_0}\in\cB\mid Z^{(0)}\in\cS)\\
&=\PP(\cS)^{-1}
\sum_{k\geq1} \PP\big(m_k<\oo,Z_{m_k}>0, (Z_t)_{0\leq t<m_k}\in\cB,
(Z_{m_k+t}-Z_{m_k})_{t\geq0}\in\cS\big)\\
&=\sum_{k\geq1} \PP\big(m_k<\oo,Z_{m_k}>0, (Z_t)_{0\leq t<m_k}\in\cB\big)
\end{split}\ee
The right-hand-side of \eqref{X1-eq} equals
\[\begin{split}
\PP(\cS)^{-1}\sum_{k\geq1}
\PP\Big(m_k<\oo,Z_{m_k}>0,& (Z_t)_{0\leq t<m_k}\in\cB_1,
(Z_{m_k+t}-Z_{m_k})_{t\geq0}\in\cS;\\
&(Z_{m_k+t}-Z_{m_k})_{0\leq t<\ell_2-\ell_1}\in\cB_2\Big).
\end{split}\]
By conditioning on $m_k$ and $(Z_t)_{0\leq t<m_k}$ we find that this
equals 
\[
\PP((Z^{(0)}_t)_{0\leq t<\D_0}\in\cB_2\mid Z\in\cS)
\sum_{k\geq1} \PP(m_k<\oo,Z_{m_k}>0,(Z_t)_{0\leq t<m_k}\in\cB_1),
\]
which by \eqref{X0-eq} and \eqref{Zk-eq} 
equals
$\PP((Z^{(1)}_t)_{0\leq t<\D_1}\in\cB_1) 
\PP((Z^{(2)}_t)_{0\leq t<\D_2}\in\cB_2)$.
\end{proof}

\begin{proof}[Proof of \eqref{eq:tauYtails} and \eqref{eq:Dtails}]
Both will follow from $\PP(\ell_1\geq t)\leq Ce^{-ct}$.
Indeed, for \eqref{eq:tauYtails} we have that
\[
\PP(\tau^Y<\oo,\tau^Y\geq t)\leq 
\PP(Z_{\ell_1-}\leq -1, \ell_1\geq t)\leq 
\PP(\ell_1\geq t),
\]
and for \eqref{eq:Dtails} we have by what was shown above
\[
\PP(\D_k\geq t\mid \cS)= \PP(\ell_1\geq t\mid \cS)
\leq \frac1z \PP(\ell_1\geq t).
\]
We have
\[\begin{split}
\PP(\ell_1\geq t)\leq \sum_{k\geq \floor{t}} \PP(\ell_1\in[k,k+1))
\leq \sum_{k\geq \floor{t}} \PP(Z_k<1)
=\sum_{k\geq \floor{t}} \PP(N'_k<k+1),
\end{split}\]
since if $\ell_1\in[k,k+1)$ then in particular $Z_t$ must be $<0$ for
some $t\in[k,k+1)$ which requires that $Z_k<1$.
Now $N'_k$ is Po($\b k$)-distributed, so a 
simple computation with Laplace-transforms gives
\[
\PP(N'_k<k+1)
\leq e\b \exp[-k(\b-1-\log \b)].
\]
Since $\b>1$ we have that $\b-1-\log \b>0$, and the bound follows. 
\end{proof}

\section{Proof of Proposition \ref{prop:apriori}}
\label{app:apriori}

Let us write simply $\cY$ for $\cY^0$.
For the first part,
let $K=\lceil \log_2(1/\eps)\rceil$.  We have
\[\begin{split}
\sum_{k=0}^K 2^k\s(2^{-k},\cY)&\geq 
\sum_{i} \cY_i \sum_{k=0}^K 2^k \one_{\{2^{-K}\leq \cY_i<2^{-k}\}}\\
&= \sum_{i} \cY_i \one_{\{\cY_i\geq \eps\}} 
\sum_{k=0}^{\lfloor\log_2(1/\cY_i)\rfloor} 2^k \\
&\geq \sum_{i\geq1} \cY_i \one_{\{\cY_i\geq \eps\}} \frac1{\cY_i}
= N_\eps(\cY).
\end{split}\]
Hence using \eqref{eq:sigma-bd},
\[
\EE[N_\eps(\cY)]\leq \sum_{k=0}^K 2^k\EE[\s(2^{-k},\cY)]
\leq C\sum_{k=0}^K k\leq C' K^{2},
\]
which gives the claim.

Next, if $\cY$ is PD($\theta$)-distributed  then
a size-biased sample from $\cY$ is
Beta(1,$\theta$)-distributed.  This means that if we select a random
index $I$ in such a way that $\PP(I=i\mid \cY)=|\cY_i|$, then 
\[
\EE[\s(\eps,\cY)]=
\EE\big[\sum_{i} \one_{\{|\cY_i|<\eps\}}|\cY_i|\big]
=\PP(\cY_I<\eps) = 1-(1-\eps)^\theta\leq \eps,
\]
as claimed. \qed

\end{document}